\documentclass[a4 paper]{amsart}
\numberwithin{equation}{section}

\usepackage{amsmath, amssymb, amscd, amsthm, mathtools, amsfonts, mathrsfs, tikz-cd, bbm, enumitem}
\usepackage[mathscr]{euscript}
\usepackage[utf8]{inputenc}
\usepackage[english]{babel}
\usepackage[hidelinks]{hyperref}

\newcommand{\N}{\mathbb{N}}

\newcommand{\Z}{\mathbb{Z}}

\newcommand{\C}{\mathbb{C}}

\newcommand{\A}{\mathbb{A}}
\newcommand{\G}{\mathbb{G}}
\newcommand{\E}{\mathbb{E}}
\renewcommand{\P}{\mathbb{P}}
\newcommand{\bbS}{\mathbb{S}}

\newcommand{\OO}{\mathcal{O}}
\renewcommand{\AA}{\mathcal{A}}
\newcommand{\BB}{\mathcal{B}}
\newcommand{\CC}{\mathcal{C}}
\newcommand{\NN}{\mathcal{N}}

\newcommand{\LL}{\mathcal{L}}
\newcommand{\MM}{\mathcal{M}}

\newcommand{\RR}{\mathcal{R}}
\newcommand{\FF}{\mathcal{F}}
\newcommand{\QQ}{\mathcal{Q}}

\newcommand{\EE}{\mathcal{E}}

\newcommand{\BBB}{\mathscr{B}}
\newcommand{\CCC}{\mathscr{C}}
\newcommand{\DDD}{\mathscr{D}}

\newcommand{\EEE}{\mathscr{E}}
\newcommand{\AAA}{\mathscr{A}}
\newcommand{\XXX}{\mathscr{X}}
\newcommand{\YYY}{\mathscr{Y}}
\newcommand{\ZZZ}{\mathscr{Z}}

\newcommand{\PPP}{\mathscr{P}}
\newcommand{\GGG}{\mathscr{G}}

\newcommand{\sMap}{\mathscr{M}\mathrm{ap}}

\newcommand{\bDelta}{\mathbf{\Delta}}
\newcommand{\bullie}{{\scalebox{0.5}{$\bullet$}}}

\newcommand{\cn}{\mathrm{cn}}
\mathchardef\hyph="2D

\tikzset{%
	symbol/.style={%
		draw=none,
		every to/.append style={%
			edge node={node [sloped, allow upside down, auto=false]{$#1$}}}
	}
}

\DeclareMathOperator{\fib}{fib}

\DeclareMathOperator{\Map}{Map}

\DeclareMathOperator{\id}{id}

\DeclareMathOperator*{\colim}{colim}

\let\lim\relax 
\DeclareMathOperator*{\lim}{lim}

\DeclareMathOperator{\Mod}{Mod}

\newcommand{\op}{\mathrm{op}}

\DeclareMathOperator{\Spec}{Spec}

\DeclareMathOperator{\Proj}{Proj}
\DeclareMathOperator{\Sym}{Sym}
\DeclareMathOperator{\Free}{Free}
\DeclareMathOperator{\vDiv}{vDiv}
\newcommand{\cl}{\mathrm{cl}}

\DeclareMathOperator{\im}{im}

\newcommand{\ext}{{\operatorname{ext}}}
\newcommand{\fp}{{\operatorname{fp}}}

\DeclareMathOperator{\Res}{Res}

\newcommand{\reg}{\mathrm{reg}}

\newcommand{\Rees}{\mathscr{R}\mathrm{ees}}
\newcommand{\Bun}{\mathscr{B}\mathrm{un}}

\newcommand{\Poly}{\mathscr{P}\mathrm{oly}}

\newcommand{\dSch}{\mathrm{d}\mathscr{S}\mathrm{ch}}
\newcommand{\Sch}{\mathscr{S}\mathrm{ch}}
\newcommand{\sAlg}{\mathrm{s}\mathscr{A}\mathrm{lg}}
\newcommand{\Aff}{\mathscr{A}\mathrm{ff}}
\newcommand{\Alg}{\mathscr{A}\mathrm{lg}}

\renewcommand{\Mod}{\mathscr{M}\mathrm{od}}

\newcommand{\St}{\mathscr{S}\mathrm{t}}

\newcommand{\Cat}{\mathscr{C}\mathrm{at}}
\newcommand{\Set}{\mathscr{S}\mathrm{et}}

\newcommand{\Fun}{\mathrm{Fun}}
\newcommand{\Space}{\mathscr{S}}

\newcommand{\QCoh}{\mathrm{Q}\mathscr{C}\mathrm{oh}}

\newcommand{\Pic}{\mathrm{Pic}}
\newcommand{\GrMod}{\mathrm{Gr}\mathscr{M}\mathrm{od}}
\newcommand{\Mon}{\mathscr{M}\mathrm{on}}
\DeclareMathOperator{\Bl}{Bl}

\newcommand{\Sur}{\mathscr{S}\mathrm{ur}}

\newcommand{\Einfty}{\mathbb{E}_\infty\hyph\Alg}
\newcommand{\Sp}{\mathscr{S}\mathrm{p}}
\newcommand{\Ecninfty}{\mathbb{E}^\cn_\infty\hyph\Alg}

\newcommand{\PrL}{\mathscr{P}\mathrm{r}^{\mathrm{L}}}
\newcommand{\PrR}{\mathscr{P}\mathrm{r}^{\mathrm{R}}}
\newcommand{\Art}{\mathrm{alg}\mathscr{S}\mathrm{t}}

\mathchardef\hyph="2D

\makeatletter
\newtheorem*{rep@theorem}{\rep@title}
\newcommand{\newreptheorem}[2]{%
	\newenvironment{rep#1}[1]{%
		\def\rep@title{#2 \ref{##1}}%
		\begin{rep@theorem}}%
		{\end{rep@theorem}}}
\makeatother

\theoremstyle{definition}
\newtheorem{Def}{Definition}[subsection]
\newtheorem{Not}[Def]{Notation}
\newtheorem{Rem}[Def]{Remark}
\newtheorem{Rem*}[]{Remark}
\newtheorem{Exm}[Def]{Example}
\newtheorem{Con}[Def]{Construction}

\theoremstyle{plain}
\newtheorem{Prop}[Def]{Proposition}
\newtheorem{Thm}[Def]{Theorem}
\newreptheorem{theorem}{Theorem}
\newtheorem{Lem}[Def]{Lemma}
\newtheorem{Cor}[Def]{Corollary}

\title{Graded algebras, projective spectra and blow-ups in derived algebraic geometry}
\author{Jeroen Hekking}
\date{\today}

\begin{document}

\begin{abstract}
	We define graded, quasi-coherent $\OO_S$-algebras over a given base derived scheme $S$, and show that these are equivalent to derived $\G_{m,S}$-schemes which are  affine over $S$. We then use this $\G_{m,S}$-action to define the projective spectrum $\Proj \BB$ of a graded algebra $\BB$ as a quotient stack, show that $\Proj \BB$ is representable by a derived scheme over $S$, and describe the functor of points of $\Proj \BB$ in terms of line bundles. The theory of graded algebras and projective spectra is then used to define the blow-up of a closed immersion of derived schemes. Our construction will coincide with the existing one for the quasi-smooth case. The construction is done by generalizing the extended Rees algebra to the derived setting, using Weil restrictions. We close by also generalizing the deformation to the normal cone to the derived setting.
\end{abstract}
	
	\maketitle
	\setcounter{tocdepth}{1}
	\tableofcontents

\section{Introduction}
Blow-ups are ubiquitous in algebraic geometry. One  constructs the blow-up of a closed immersion $f:Z \to X$ by taking the projective spectrum of the Rees algebra associated to $f$. Here are thus two very basic tools of algebraic geometry that go into this construction: quasi-coherent, graded algebras and projective spectra of these algebras.

Recently, the blow-up $\Bl_ZX$ of a derived scheme $X$ in a  centre $Z$ has been constructed in \cite{KhanVirtual} in the case that $Z \to X$ is quasi-smooth, which is the derived analogue of a regular embedding. The strategy used  is to describe the functor of points of $\Bl_ZX$ and then to show that it is representable, thus circumventing the need to consider graded algebras or projective spectra. 

The goal of this article is to define the blow-up of a closed immersion $Z \to X$ of derived schemes in the general case by means of the projective spectrum of the associated Rees algebra. Since, perhaps surprisingly, graded algebras and projective spectra were not available in the derived setting at the time of this writing, an intermediate goal is to first develop these concepts, which are also interesting in their own right. A large part of the work has gone into understanding derived graded algebras, about which we will first say a few words.

\subsection{Graded algebras}
Classically, for a scheme $S$, the following categories are, possibly contravariantly, equivalent:
\begin{enumerate}
	\item $\Z$-graded, quasi-coherent $\OO_S$-algebras,
	\item Schemes affine over $S$ with a $\G_{m,S}$-action,
	\item $\OO_S[t^{\pm 1}]$-comodules in the symmetric monoidal category of quasi-coherent $\OO_S$-algebras,
	\item Quasi-coherent $\OO_{B\G_{m,S}}$-algebras,
	\item Commutative monoids in the symmetric monoidal category $\QCoh^\Z(S)$ of $\Z$-graded, quasi-coherent $\OO_S$-modules,
	\item Lax symmetric monoidal functors  $\Z \to \QCoh(S)$. 
\end{enumerate}
The equivalences between (1),(2),(3) are either done by hand, or using the Barr--Beck theorem. For (2) $\iff$ (4), one uses that $T \to B\G_{m,S}$ is affine if and only if the pullback $P \to S$ along $S \to B\G_{m,S}$ is as well. Finally, (5) and (6) are equivalent, as shown by hand or by using Day convolution, and both are easily seen to be equivalent to (1).

This list induces six  equivalent kinds of input for the $\Proj$-construction, by taking $\N$-graded objects in (1) and their counterparts in the other categories. Therefore, in generalizing the $\Proj$-construction to the derived setting, one is now faced with the question which of these inputs to take as the starting position. 

The derived versions of (2),(3) and (4) turn out to be equivalent to one another, as we shall see. In setting up the theory of $\Z$-graded, quasi-coherent $\OO_S$-algebras in the context of derived algebraic geometry, it is reasonable to ask that the result is also equivalent to the derived version of (2). Indeed, suppose that $S$ is affine, say $S=\Spec A$. Then a simplicial $A$-algebra with a levelwise $A[t^{\pm 1}]$-coaction gives rise to a simplicial graded $A$-algebra, by applying the classical construction levelwise. 

We have therefore taken the equivalence between (1) and (2) as design criterion in generalizing (1) to the derived setting. And indeed, the first main result of this paper is the following:

\begin{reptheorem}{Thm:Equiv}
	Let $S$ be a derived scheme. Then there is a contravariant equivalence between the category of quasi-coherent, $\Z$-graded $\OO_S$-algebras and the category of $\G_{m,S}$-schemes affine over $S$.
\end{reptheorem}

 Although the derived versions of (5) and (6) are also equivalent to one another, we show that in general these are not equivalent to the derived versions of (1)--(4).

We then use this equivalence to define the projective spectrum of an $\N$-graded, quasi-coherent $\OO_S$-algebra $\BB$ generated in degree 1, over a derived scheme $S$, as the quotient stack $[(\Spec \BB \smallsetminus V(\BB_+)) / \G_{m,S}]$. The second main result is then:

\begin{reptheorem}{Thm:ProjRep}
	For $S,\BB$ as above, $\Proj \BB \coloneqq [(\Spec \BB \smallsetminus V(\BB_+)) / \G_{m,S}] $ is representable by a derived scheme over $S$.
\end{reptheorem}

Furthermore, for $S,\BB$ as above, we recover three results from the classical case. First, we give an exact functor from quasi-coherent, graded $\BB$-modules to quasi-coherent $\OO_{\Proj \BB}$-modules. Next, we introduce Serre's twisting sheaves and check familiar behaviour of these sheaves. Third, we show that for a given map of derived schemes $\psi: T\to S$, the space of maps $T \to \Proj \BB$ over $S$ is the space of graded homomorphisms $\psi^*\BB \to \MM$, where $\MM$ is an $\N$-graded $\OO_T$-algebra freely generated by a line bundle in degree 1, and $\psi^*\BB \to \MM$ is a  homomorphism of graded $\OO_T$-algebras for which the induced map $\pi_0 (\psi^* \BB_1) \to \pi_0( \MM_1)$ is surjective. This recovers $\OO_{\Proj \BB}(1)$ as universal line bundle.

\subsection{The problem in the derived setting}
\label{Par:Problem}
Let $S$ be a derived scheme. Morally, a $\Z$-graded, quasi-coherent $\OO_S$-algebra $\BB$ should be a quasi-coherent $\OO_S$-algebra $\BB$ together with a decomposition $\BB \simeq \bigoplus_d \BB_d$ in $\QCoh(S)$ such that the multiplication maps $\BB^{\otimes n} \to \BB$ induce homotopy-coherent systems $\BB_{d_1} \otimes \cdots \otimes \BB_{d_n} \to \BB_{d_1+\dots+d_n}$. 

Writing down this infinite amount of coherency data by hand is not  easy. It needs to guarantee that on the one hand the concept is flexible enough so that quasi-coherent $\OO_S$-algebras $\AA$ with an $\OO_S[t^{\pm 1}]$-coaction, which are given as homotopy coherent cosimplicial diagrams of the form $\AA \rightrightarrows \AA[t^{\pm 1}] \dots$, give rise to $\Z$-graded, quasi-coherent $\OO_S$-algebras. On the other hand, it needs to be strict enough so that forgetting the grading gives us a quasi-coherent $\OO_S$-algebra, i.e.,\ locally a simplicial commutative ring and not only an $\E_\infty$-ring.

One approach is via model categories. Let $C$ be the category of simplicial objects in the category of $\Z$-graded rings. We could define a model structure on $C$ by declaring a morphism $B \to B'$ to be a fibration, resp.\ a weak equivalence, when it is so on the underlying simplicial set. Then for a simplicial commutative ring $R$, we let the category $\CCC_R$ of $\Z$-graded $R$-algebras be the coslice category $\CCC_{R/}$, where $\CCC$ is the $\infty$-category associated to $C$, and $R$ is endowed with the trivial grading. 

A $\Z$-graded $\OO_S$-algebra $\BB$ would then locally on $S$ be given by a $\Z$-graded simplicial $\OO_S(U)$-algebras, for $U \subset S$ affine open. To define an $\OO_S[t^{\pm 1}]$-coaction on such a $\BB$, we need a morphism $\BB \to \BB[t^{\pm 1}]$ of $\OO_S$-algebras which sends $\BB_d$ to $\BB_dt^d$. First assume that $S = \Spec R$, and that $\BB$ is given by a $\Z$-graded simplicial $R$-algebra $B$. To give something which is homotopically well-defined, $B[t^{\pm 1}]$ must be the derived tensor product $B \otimes_R R[t^{\pm 1}]$. The question is then whether we can compute this derived tensor product and replace $B$ in such a way that $B \to B[t^{\pm 1}]$ can be defined by the classical construction levelwise, i.e.,\ by sending a homogeneous simplex $b$ of degree $d$ to $bt^d$.

We can answer this last question affirmatively, but only after we have shown Theorem \ref{Thm:Equiv}. To use only the model-categorical approach from the start seems to be intractable even in the local picture. What is more, to glue such a construction into an $\OO_S[t^{\pm 1}]$-coaction on all of $\BB$ would involve gluing different cofibrant models on affine patches. Making this into an $\infty$-functor from $\Z$-graded $\OO_S$-algebras to $\OO_S$-algebras with  $\OO_S[t^{\pm 1}]$-coaction seems hopeless. 

What is therefore needed is a  functor, intrinsic to the language of $\infty$-categories, from the $\infty$-category $\Alg$ of simplicial commutative rings to the $\infty$-category of $\infty$-categories, sending $R \in \Alg$ to the $\infty$-category $\Alg^\Z_R$ of $\Z$-graded $R$-algebras. Better yet, this should come with a universal property for colimit-preserving functors out of $\Alg^\Z_R$. We will indeed give such a construction. The universal property of $\Alg^\Z_R$ will allow us to give a functor $F$ from $\Alg^\Z_R$ to the $\infty$-category $\Alg_R^{\G_m}$ of $R$-algebras with $R[t^{\pm 1}]$-coaction. Functoriality of $\Alg^\Z_R$ in $R$ will allow us to globalize this  to any base derived scheme.

The mapping spaces of $\Alg^\Z_R$ and $\Alg_R^{\G_m}$ are too different from one another to show that $F$ is fully faithful by hand. Luckily, it will turn out that the functor $F:\Alg^\Z_R\to \Alg_R^{\G_m}$ is a left adjoint for abstract reasons, using the universal property of $\Alg^\Z_R$, which will give us a candidate for the inverse of $F$. This is good, because writing down an inverse functor $\Alg_R^{\G_m} \to \Alg^\Z_R$ by hand is again not so easy. Indeed, an object $B$ in $\Alg_R^{\G_m}$ is given by a cosimplicial diagram $C(B)$ of the form $B \rightrightarrows B[t^{\pm 1}] \dots $ in the $\infty$-category of $R$-algebras. If this diagram was a strictly commutative diagram in the 1-category of simplicial commutative rings, then we could apply the classical construction levelwise, to get a grading on $B$. Unfortunately, standard strictification techniques do not guarantee that $C(B)$ can be strictified, without changing the objects $C^n(B)$ up to homotopy. Indeed, it will be a consequence that one can actually strictify any $\Alg_R^{\G_m}$ to a levelwise $R[t^{\pm 1}]$-coaction.

\subsection{Rees algebras and blow-ups}
\label{Par:ReesClas}
Let $Z \to X$ be a closed immersion of derived scheme. Since we will define the blow-up $\Bl_ZX$ as the projective spectrum of the Rees algebra $\RR_{Z/X}$, the behaviour of $\Bl_ZX$ is thus completely controlled by that of $\RR_{Z/X}$. Thus, the question whether our construction of $\Bl_ZX$ is `correct' is reduced to the same question  about $\RR_{Z/X}$. To this end, let us first consider a universal property of the extended Rees algebra in the classical case.

 Let $A \to B$ be a surjection of (discrete) rings with kernel $I$, and $R_{B/A}^\ext$ the extended Rees algebra $\bigoplus_{n \in \Z} I^nt^n \subset A[t^{\pm 1}]$, where $I^{-m} = A$ for $m \geq 0$. Let $Q$ be a $\Z$-graded $A[t^{-1}]$-algebra for which $t^{-1} \in Q$ is not a zero divisor. Here and in what follows, we take $\deg t^{-1} = -1$. We then have an isomorphism of sets
	\begin{align*}
	{\mathrm{Alg}^\Z_{A[t^{-1}]}}(R^\ext_{B/A},Q) \cong {\mathrm{Alg}^\Z_A}(B, Q/({t^{-1}}) )
	\end{align*}
between $\Z$-graded $A[t^{-1}]$-algebra maps and $\Z$-graded $A$-algebra maps, natural in $Q$. Indeed, there is a graded $A$-algebra map $B \to Q/(t^{-1})$ if and only if the structure map $A \to Q$ maps $I$ into $(t^{-1})$. It thus suffices to show that a map $\varphi: R^\ext_{B/A} \to Q$ of $\Z$-graded $A[t^{-1}]$-algebras is determined by its restriction to degree 0. And indeed this is so: the composition $R_{B/A}^\ext \to Q \xrightarrow{\times t^{-1}} (t^{-1})$ shows that $\varphi_{\mid It}$ is determined by $\varphi_0$, and similarly for higher degrees.

Let now $A \to B$ be a map of simplicial rings, surjective on $\pi_0$. We will define the derived version of the extended Rees algebra in such a way that it satisfies a somewhat more robust version of this universal property: it will turn out that we have an equivalence of spaces
\begin{align*}
\Alg^\Z_{A[t^{-1}]}(R^\ext_{B/A},Q) \cong \Alg_A(B,(Q/({t^{-1}}))_0)
\end{align*}
for any $\Z$-graded $A[t^{-1}]$-algebra $Q$. Observe that $Q \mapsto (Q/({t^{-1}}))_0$ is easily defined. One could therefore hope to define $R_{(-)/A}^\ext$ by means of this universal property via some adjoint functor theory. 

This universal property will however not lead to an adjunction as stated, since in general only a map $A \to B$  which is surjective on $\pi_0$ will have an extended Rees algebra. We can remedy this by restricting the category $\Alg_A$ to the category of surjections $A \to B$, and $\Alg^\Z_{A[t^{-1}]}$ to the category of those $\Z$-graded $A[t^{-1}]$-algebras $Q$ for which $A[t^{-1}] \to Q_{\leq 0}$ is an equivalence. However, by doing so we leave the world of presentable $\infty$-categories. Because for general $\infty$-categories, adjoint functor theorems are less powerful, the strategy to define $R^\ext_{(-)/A}$ as a left adjoint to $Q \mapsto (Q/({t^{-1}}))_0$ by means of an adjoint functor theorem did not succeed. 

The solution was found using Weil restrictions. To sketch the idea, let $z: X \to \A^1_X$ be the zero section. Then the theory of Weil restrictions says that pulling back stacks along $z$ has a right adjoint, written $\Res_{X/\A^1_X}$. We thus have a natural equivalence
\begin{align*}
\St_{/\A^1_X}(T,\Res_{X/\A^1_X}(Z)) \simeq \St_{/X}(T\times_{\A^1_X}X,Z)
\end{align*}
of mapping spaces of stacks. If $\Res_{X/\A^1_X}(Z)$ is affine over $\A^1_X$, say $\Res_{X/\A^1_X}(Z) = \Spec \RR$, then for $T = \Spec \QQ$ we will get
\begin{align*}
\Alg_{\OO_X[t^{-1}]}(\RR,\QQ) \simeq \Alg_{\OO_X}(\BB,\QQ/(t^{-1}))
\end{align*}
where $\BB$ is such that $Z = \Spec \BB$. This is almost the universal property that we want, save that we have lost the grading on $\RR$.

To recover the grading, one starts not with $z$ but with the induced map $\zeta:B\G_{m,X} \to [\A^1_X/\G_{m,X}]$. We define $D_{Z/X}$ as the pullback of $\Res_\zeta([Z/\G_{m,X}])$ along $\A^1_X \to [\A^1_X/\G_{m,X}]$. We then show the following:
\begin{reptheorem}{Thm:ReesRep}
	The stack $D_{Z/X}$ is representable by a scheme affine over $\A^1_X$.
\end{reptheorem}
By construction, $D_{Z/X}$ will then be of the form $\Spec  \RR_{Z/X}^\ext$, where $\RR_{Z/X}^\ext$ is a $\Z$-graded $\OO_X[t^{-1}]$ algebra which satisfies the desired universal property. We will take this $\RR_{Z/X}^\ext$ as the extended Rees algebra of $Z$ over $X$, and put $\Bl_ZX \coloneqq \Proj\RR_{Z/X}^\ext$.

We present two additional reasons why our recipe for blowing up in the derived setting is `correct'. First, for $Z\to X$ a closed immersion of classical schemes, we can recover the classical extended Rees algebra of $Z$ over $X$ by forcing $t^{-1}$ to be regular on $\pi_0\RR^\ext_{Z/X}$, thereby capturing the classical picture in a controlled fashion. 

Second, for $Z \to X$ a closed immersion of derived schemes, a \textit{virtual Cartier divisor} $(S,D)$ over $(X,Z)$ is a scheme $S$ over $X$ together with morphisms $D \xrightarrow{h} S_Z \to S$ such that $D \to S$ is locally on $S$ of the form $\Spec R \to \Spec R/(f)$, and such that $\OO_{S_Z} \to h_* \OO_D$ is an isomorphism on $\pi_0$ and surjective on $\pi_1$. Let $\vDiv_ZX$ be the space of virtual Cartier divisors over $(X,Z)$. In \cite{KhanVirtual}, the blow-up of $X$ in $Z$ is defined as $\vDiv_ZX$ in the case that $Z \to X$ is quasi-smooth, in which case it is shown to be representable by a derived scheme. We will then show:
\begin{reptheorem}{Thm:BlDiv}
	There is a virtual Cartier divisor $(\Bl_ZX, \Proj_Z( \RR_{Z/X}^\ext / (t^{-1}) ))$ over $(X,Z)$. If $Z \to X$ is quasi-smooth, then the induced map $h_{Z/X}:\Bl_ZX \to \vDiv_ZX$ is an equivalence.
\end{reptheorem}

We will now introduce some terminology and notation. After this, we give a summary of this paper in \S \ref{Par:Summ}, where we also outline the solution to the problems discussed in \S \ref{Par:Problem}.


\subsection{Elements of derived algebraic geometry}
\label{Par:DAGEls}
We use the language of $\infty$-category theory and derived algebraic geometry as in \cite{ToenDerived}, \cite{ToenHAGI}, \cite{LurieSpectral}, \cite{GaitsgoryStudy}. The $\infty$-category of spaces is denoted by $\Space$.

Write $\Alg$ for the $\infty$-category of simplicial commutative rings, simply called \textit{rings}. One way to construct $\Alg$ is by endowing the category of simplicial objects in the category of commutative rings with  a simplicial model structure, where weak equivalences and fibrations are taken on the underlying simplicial sets. Then $\Alg$ is the $\infty$-category associated to this model category. Alternatively, $\Alg$ is the $\infty$-category of product-preserving functors $\Poly^\op \to \Space$, where $\Poly$ is category  of polynomial algebras $\Z[T_1,\dots,T_n]$ and algebra maps between them.\footnote{Here, as always, we consider 1-categories as $\infty$-categories via the nerve construction.} 

For $R \in \Alg$, we let $\cn\Mod_{R}$ be the $\infty$-category associated to the model category of simplicial $R$-modules, and $\Mod_R$ the stabilization of $\cn\Mod_R$. Alternatively, $\Mod_R$ is the $\infty$-category of modules of $R$ considered as an $\E_\infty$-algebra, which is automatically stable. Then $\cn\Mod_{ R}$ consists of the connective $R$-modules $M$, i.e.,\ those for which $\pi_n M = 0$ for all $n<0$.

Still for $R \in \Alg$, we  write $\Alg_R$ for the $\infty$-category $\Alg_{R/}$. Points in $\Alg_R$ are called \textit{$R$-algebras}. Likewise, points in $\Mod_R$ are called \textit{$R$-modules}. When we work in the point-set models of these $\infty$-categories, we indicate this by talking about \textit{simplicial} $R$-algebras resp.\ \textit{simplicial} $R$-modules (in the connective case). 

We write $\Sch$ for the $\infty$-category of derived schemes. For $X$ a derived scheme, we let $\QCoh(X)$ be the $\infty$-category of quasi-coherent $\OO_X$-modules, which has a full subcategory $\cn\QCoh  (X)$ spanned by the connective $\OO_X$-modules. The category $\QCoh(X)$ is  closed symmetric monoidal, presentable and stable  \cite[Prop.\ 2.2.4.1; 2.2.4.2, Rem.\ 6.5.3.8]{LurieSpectral}, with left- and right-complete $t$-structure \cite[Prop.\ 2.2.5.4]{LurieSpectral}.  

If $X$ is affine, say $X = \Spec A$, then we have $\QCoh(X) \simeq \Mod_A$. In general, it holds that $\QCoh(X) \simeq \lim_{\Spec A \to X} \Mod_A$ as $\infty$-categories, where the indexing category runs over all open immersions $\Spec A \to X$, and the limit is taken in the $\infty$-category of stable $\infty$-categories. A point $\MM$ in $\QCoh(X)$ is thus a homotopy-coherent system of $\OO_X(U)$-modules $\MM(U)$ indexed by the affine open subsets $U \subset X$. If all of these modules are connective, such a system thus gives rise to a  sheaf (in the homotopical sense, explained below) of simplicial abelian groups on the underlying space of $X$, with action $\OO_X \times \MM \to \MM$, which makes $\MM(V)$ into an $\OO_X(V)$-module for each $V$ open in $X$. 

Similarly, we define the $\infty$-category $\Alg(X)$ of \textit{quasi-coherent $\OO_X$-algebras} for a derived scheme $X$ as the limit $\lim_{\Spec A \to X} \Alg_{A}$ over all affine open subsets of $X$. 

By \textit{stack} we mean a sheaf $\Alg \to \Space$ for the \'e{tale} topology. We write $\St$ for the full subcategory of presheaves on $\Alg^\op$ spanned by stacks.
For a stack $\XXX$ we let $\QCoh(\XXX)$ be the limit $\lim_{\Spec R \to \XXX} \Mod_R$ over all maps $\Spec R \to \XXX$. When $\XXX$ is a derived scheme, then this notion agree with the previous definition. For a map of stacks $\varphi: \XXX \to \YYY$, we have an adjunction
\begin{align*}
\varphi^*: \QCoh(\YYY) \rightleftarrows \QCoh(\XXX): \varphi_*
\end{align*}
See e.g.\ \cite[\S I.3.2]{GaitsgoryStudy}.

Roughly speaking, a derived algebraic stack is a stack $\XXX$ that allows a smooth atlas. The precise definition is not important for us. What matters is the following fact: let $\Art$ be the $\infty$-category of derived algebraic stacks, and $\PrL$ be the $\infty$-category of presentable $\infty$-categories with left-adjoint functors between them. Then any functor $F:\Alg \to \PrL$ can be right-Kan extended to a functor $\hat{F}:\Art \to \PrL$. This is because for $\XXX$ algebraic, the indexing category to compute $\hat{F}(\XXX)$ consisting of maps of the form $\Spec R \to \XXX$ can be taken to be small, and $\PrL$ is closed under small limits.

\subsection{Conventions}
The term `levelwise' will refer to the simplicial level. Since we will be dealing with graded simplicial objects, we reserve the subscript notation $B_d$ for the graded piece of degree $d$. For any simplicial object $X$, we reserve the notation $x \in X$ to mean points (i.e.,\ zero-simplices) $x$ in $X$.

From here on, everything is derived, by which we mean that all (co)limits are homotopy (co)limits c.q.\ (co)limits in the $\infty$-sense, all functors are $\infty$-functors, all schemes are derived, all diagrams commute up to homotopy, etc., unless otherwise stated. In particular, sheaves are taken in the homotopical sense, meaning that for a cover $\{U_\alpha\to V\}_\alpha$ the sheaf condition reads that
\begin{align*}
\FF(V) \to \lim_{n \in \bDelta} \left( [n] \mapsto \prod\nolimits_{(\alpha_0,\dots,\alpha_n)} \FF(U_{\alpha_0} \times_V \dots \times_V U_{\alpha_n}) \right)
\end{align*}
should be an equivalence. Observe: if this limit is taken in a 1-category, then it can be computed on the subcategory $[1] \rightrightarrows [0]$, which retrieves the classical notion of a sheaf.

Colimits of simplicial diagrams are called \textit{geometric realizations}. Dually, limits of cosimplicial diagram are called \textit{totalizations}. Functors that preserve colimits are called \textit{cocontinuous}.

Maps of rings, algebras, modules,... are called \textit{surjective} if they are so on $\pi_0$. For $\varphi:M \to M'$ a map of $R$-modules, the fibre is always taken in $\Mod_R$. If $\varphi$ is surjective, then this coincides with the fibre in $\cn\Mod_R$.

\subsection{Summary}
\label{Par:Summ}
Let $S$ be a scheme and $R$ a ring. In \S \ref{Sec:MGradAlg} we introduce the category of $M$-graded $R$-algebras, for any discrete commutative monoid $M$. 

A key insight in the proof of Theorem \ref{Thm:Equiv} is that the category $\Alg^\Z$ of $\Z$-graded rings can be constructed in two ways. It is either the category of those presheaves on the category $\Poly^\Z$ of $\Z$-graded polynomial rings over $\Z$ that send tensor products to products, or it is the category associated to a model 1-category of simplicial graded rings. Thus, in the first incarnation, one does not start with an ring $B$ and  a decomposition $B \simeq \bigoplus_d B_d$ in $\Mod_{\Z}$ as in \S \ref{Par:Problem}. Instead, one only specifies the graded pieces $B_d$, and defines the algebraic structure by giving `substitution maps' $B_{d_1} \times \dots \times B_{d_n} \to B_{d_1+\dots+d_n}$, for any graded polynomial $P(X_1,\dots,X_n)$ over $\Z$ that has $X_i$ in degree $d_i$, in a homotopy coherent way. 

We then define the category $\Alg^\Z_R$ of $\Z$-graded $R$-algebras as the slice category $\Alg^\Z_{R/}$. We show that the construction of $\Alg^\Z_R$ is functorial in $R$ in Proposition \ref{Prop:NRadj}, which allows us to globalize the notion of graded algebras in \S \ref{Par:GradOS} to a category $\Alg^\Z(S)$.

We recall group actions and principal bundles in the setting of derived algebraic geometry in \S \ref{Sec:EquivGeom}. Here we also introduce the category $\Aff^{\G_m}(S)$ of $\G_m$-schemes affine over  $S$. Objects of $\Aff^{\G_m}(S)$ will be cosimplicial diagrams in the category of quasi-coherent $\OO_S$-algebras which satisfy the derived analogue of coaction axioms. It is shown that  $\Aff^{\G_m}(S)$ is equivalent to the category $\Alg(B\G_{m,S})$.

Section \S \ref{Sec:GradAlgGm} is more technical, and contains the proof of Theorem \ref{Thm:Equiv}. The strategy is as follows. First we give an adjunction $\Phi: \Alg^\Z \rightleftarrows \Alg^{\G_m}:\Psi$. A $\Z$-graded algebra can be strictified to a simplicial object in the 1-category of discrete $\Z$-graded algebras. Although on these algebras $\Phi$ will coincide with the classical construction levelwise, we will not use this description to introduce $\Phi$. Indeed, the reason that we first do the absolute case, is that colimit-preserving functors out of $\Alg^\Z$ are classified by coproduct-preserving functors out of the 1-category $\Poly^\Z$, which is a tractable problem.

After this construction, we show that $\Phi$ is fully faithful in Proposition \ref{Prop:unitinv}, by showing that the unit is invertible. We then show that $\Psi$ is conservative, by showing that it commutes with the forgetful functors to $\Alg$. We then globalize to an adjunction equivalence $\Phi_S\dashv \Psi_S: \Alg^\Z(S) \rightleftarrows \Alg^{\G_{m,S}}(S)$ in \S \ref{Par:Glob}. 

In \S \ref{Sec:Proj} we introduce the projective spectrum of a quasi-coherent, $\N$-graded $\OO_S$-algebra $\BB$ generated in degree 1  as $\Proj \BB \coloneqq [(\Spec \BB \smallsetminus V(\BB_+)) / \G_{m,S}] $. We show that $\Proj \BB$ is representable by giving an open cover of affine schemes of the form $\Spec B_{(f)}$. Here, $B_{(f)}$ is the degree-zero part of the graded localization $B_f$, where $B$ is an $\N$-graded algebra, and $f$ is a homogeneous element of $B$ in degree 1. Lemma \ref{Lem:Deg0loc} is crucial in this approach, which says that the induced $\G_m$-action on $\Spec B_f$ is free. 

In \S \ref{Sec:Rees} we introduce the extended Rees algebra of a closed immersion $Z \to X$ as a $\Z$-graded, quasi-coherent $\OO_X[t^{-1}]$-algebra $\RR_{Z/X}^\ext$, using Weil restrictions. We show that this construction is stable under base-change, and give an explicit formula for the Rees algebra in the case that $Z \to X$ is a finite quotient. This is a closed immersion of the form $\Spec B \to \Spec A$, where $B$ is obtained from $A$ by a single finite cell attachment. Then in \S \ref{Par:ReesAdj} we recover the adjunction mentioned in \S \ref{Par:ReesClas}. We close with applying our formula to a regular embedding of classical schemes.

Finally, in \S \ref{Sec:Bl} we introduce the blow-up $\Bl_ZX \coloneqq \Proj \RR_{Z/X}$ of arbitrary closed immersions $Z \to X$, show that our description coincides with $\vDiv$ in the quasi-smooth case, and show that we recover the classical picture as the underlying classical scheme of $\Bl_ZX$. We close by proposing a generalization of the deformation to the normal cone to the derived setting in \S \ref{Par:DefCZX}.

Not everything in the body text is strictly necessary to show the main results. Most readers can safely skim over \S \ref{Sec:EquivGeom}: this section is mostly there to fix notation. There are some smaller results, like the adjunction involving the Rees algebra in \S \ref{Par:ReesAdj}, which were only relevant as intermediate results in earlier versions. However, they now serve to give a more complete picture of the theory, and are thus left in. Some results that were needed in earlier versions but became obsolete in the current version have been collected in the appendix, since they seemed interesting in their own right, but did not add to the present story.

\subsection{Acknowledgements}
I would like to thank David Rydh for fruitful discussions and for extensive comments on previous versions. The key idea in the construction of Rees algebras, namely to use Weil restrictions, is his. I would also like to thank Adeel Khan for his helpful feedback and for several encouraging Zoom sessions. Lastly, I would like to thank Tyler Lawson for our  e-mail conversation that gave me the main ideas for the proof of Proposition \ref{Prop:FPisFC}.

\section{$M$-graded algebras}
\label{Sec:MGradAlg}
Recall that $\Alg$ is the category of functors $\Poly^\op \to \Space$ that preserve products, where $\Poly$ is the 1-category of finitely generated polynomial rings of the form $\Z[x_1,\dots,x_n]$. One can strictify these objects in the sense that $\Alg$ is equivalent to the $\infty$-category associated to the model 1-category of simplicial commutative rings. We now generalize this picture to the graded setting.
\subsection{Graded rings}
Throughout, fix a commutative monoid $M$ (always discrete). Let $\Poly^M$ be the 1-category of graded polynomial rings which are of the form $\Z[x(d_1),\dots,x(d_n)]$, where $x(d_i)$ is of degree $d_i\in M$, and with graded ring homomorphisms between them.\footnote{It would be more precise to write $x(d_i)$ as $x_i(d_i)$, since the degrees are not assumed to be different. The double index will be suppressed nonetheless.} 

For any category $\CCC$, let $\PPP(\CCC)$ be the category of presheaves on $\CCC$, and $\PPP_\Sigma(\CCC)$ the full subcategory of $\PPP(\CCC)$ spanned by the finite product-preserving presheaves on $\CCC$, i.e.\ those functors $\CCC^\op \to \Space$ that send finite coproducts in $\CCC$ to finite products in $\Space$.
\begin{Def}
	The category $\Alg^M\coloneqq \PPP_\Sigma(\Poly^M)$ is the category of \textit{$M$-graded rings}.
\end{Def}

\begin{Rem}
	\label{Rem:StrictAlgMA}
	Let $\A$ be the 1-category of simplicial objects in the 1-category of $M$-graded, discrete rings. Objects of $\A$ are called \textit{simplicial, $M$-graded rings}. Then $\A$ has a simplicial model structure, where a map $B \to B'$ is a weak equivalence resp.\ a fibration if the maps $B_d \to B_d'$ are weak equivalences resp.\ fibrations of simplicial sets, for all $d \in M$. Consider $\Poly^M$ as a full subcategory of $\A$.
	
	Let $\AAA$ be the $\infty$-category associated to $\A$. Then the natural map $\AAA \to \Alg^M$ is an equivalence by \cite[Cor.\ 5.5.9.3]{LurieHTT}, since discrete $M$-graded rings are product-preserving 1-presheaves on $\Poly^M$. This map $\AAA \to \Alg^M$ is induced by the Yoneda embedding: we send a fibrant-cofibrant object $B$ in $\A$ to the functor 
	\begin{align*}
	\Z[x(d_1),\dots,x(d_n)] \mapsto \Map_{\A} (\Z[x(d_1),\dots,x(d_n)],B ) \simeq B_{d_1} \times \dots \times B_{d_n}
	\end{align*}
	which is an object of $\PPP_\Sigma(\Poly^M)$. See loc.\ cit.\ for details. 
\end{Rem}

\subsection{Graded $R$-algebras}
\label{Par:AlgMR}
Throughout, fix $R \in \Alg$. Let $\epsilon: \Poly\to \Poly^M$ be the  functor that sends $\Z[x]$ to $\Z[x(0)]$ and which preserves finite coproducts. Write $\Poly_R$ for the full subcategory of $\Alg_R$ spanned by $R$-algebras of the form $R[x_1,\dots,x_n]$, and let $\delta:\Poly \to \Poly_R$ be the functor that sends $\Z[x]$ to $R[x]$ and which preserves finite coproducts. Let $(-)_0$ be the functor $\Alg^M \to \Alg$ which is given by precomposing with $\epsilon$, i.e.,\ this functor sends an $M$-graded ring $B$ to its degree-zero part $B_0$.

\begin{Exm}
	We consider $R$ as an $M$-graded ring with trivial grading as follows. Taking a model for $R$, we may assume that $R$ is a simplicial ring. Then endow this simplicial object in the 1-category of discrete rings level-wise with the trivial grading (i.e.,\ concentrated in degree 0). This gives a simplicial object in the 1-category of discrete graded rings, hence an object of $\Alg^M$ by Remark \ref{Rem:StrictAlgMA}. 
	
	We will see in Example \ref{Exm:fjadjs} the following more intrinsic definition of the trivial grading, without involving simplicial models. Namely, we can consider $R$ as a functor $\Poly^\op \to \Space$. Then $R$ endowed with trivial grading is the left Kan extension of $R$ along  $\epsilon$. 	
\end{Exm}

\begin{Def}
	\label{Def:GradAlg}
	The category $\Alg^M_R$ of \textit{$M$-graded $R$-algebras} is the fibre product $\Alg^M \times_{\Alg} \Alg_R$ of $(-)_0$ along the forgetful functor $U:\Alg_R \to \Alg$.
\end{Def}

\begin{Rem}
	\label{Rem:AlgRMcoslice}
	Observe that $U:\Alg_R \to \Alg$ is a fibration in the Joyal model structure by \cite[Cor.\ 2.4.6.5]{LurieHTT}.	It follows that $\Alg^M_R$ is the underived fibre product of simplicial sets. Hence, $\Alg^M_R$ is the category of pairs $(B, R \to B_0)$, where $B$ is an $M$-graded ring and $R \to B_0$ a ring map. In other words, $\Alg^M_R$ is the coslice category $\Alg^M_{R/}$, where $R$ is endowed with the trivial grading.
\end{Rem}

\begin{Rem}
	\label{Rem:gradpiece}
	From Remark \ref{Rem:AlgRMcoslice} and Remark \ref{Rem:StrictAlgMA} it follows that, for $B \in \Alg^R_M$, we have a decomposition $B \simeq \bigoplus\nolimits_{d \in M} B_d$ with $B_d \in \QCoh(R)$ the graded piece of degree $d$.
\end{Rem}

Let $R[x(d_1),\dots,x(d_n)]$ be the tensor product $R \otimes_\Z \Z[x(d_1),\dots,x(d_n)]$ in $\Alg^M$, for $d_1,\dots,d_n \in M$. Then $R[x(d_1),\dots,x(d_n)]$ gives an object of $\Alg^M_R$ in an obvious way. Let $\Poly^M_R$ be the full subcategory of $\Alg^M_R$ spanned by the objects of this form.  

We need the following two notions from \cite[\S 5.5.8]{LurieHTT}. A non-empty simplicial set $K$ is \textit{sifted} if the diagonal $d:K \to K \times K$ is cofinal, that is, if precomposition with $d$ preserves colimits. Now an object $c$ in a category $\CCC$ is \textit{compact projective} if it corepresents a functor $\CCC \to \Space$ that preserves sifted colimits. For example, for a category $\DDD$ with finite coproducts, the image of $d \in \DDD$ under the map $\DDD \to \PPP_\Sigma(\DDD)$ is compact projective.

\begin{Lem}
	\label{Lem:AlgRForgetSifted}
	The functor $U:\Alg_R \to \Alg$ and the functor $(-)_0 :\Alg^M \to \Alg$ preserve sifted colimits.
\end{Lem}

\begin{proof}
			By \cite[Prop.\ 5.5.8.10]{LurieHTT}, sifted colimits in $\PPP_{\Sigma}(\CCC)$ are computed as colimits in $\PPP(\CCC)$, i.e.,\ pointwise, for any category $\CCC$. It follows that for any functor $\varphi:\CCC \to \DDD$ which preserves finite coproducts, the map $\varphi^*:\PPP_\Sigma(\DDD) \to \PPP_\Sigma(\CCC)$ preserves sifted colimits. Now the claim follows since the forgetful functor $\Alg_R \to \Alg$ is $\delta^*$, and $(-)_0:\Alg^M \to \Alg$ is $\epsilon^*$.
\end{proof}

\begin{Lem}
	\label{Lem:PolMRcompproj}
	The objects in $\Poly^M_R$ are compact projective in $\Alg^M_R$.
\end{Lem}

\begin{proof}
%
	Consider objects in $\Alg^M_R$ as pairs $(B,B')$ with $B \in \Alg_R$ and $B' \in \Alg^M$. By \cite[Prop.\ 6.2.19]{CisinskiHigh} and Lemma \ref{Lem:AlgRForgetSifted}, the colimit of a sifted diagram in $\Alg^M_R$ of the form $\{(B_\alpha,B'_\alpha)\}_\alpha$ is $(\colim B_\alpha, \colim B'_\alpha)$. In particular, the projection $\Alg^M_R \to \Alg^M$ preserves sifted colimits. For any $d \in M$, the functor $\Alg^M \to \Space: B \mapsto B_d$ also preserves sifted colimits, since $\Z[x(d)]$ is compact projective in $\Alg^M$. It follows that the composition
	\begin{align*}
	\Alg^M_R \to \Alg^M \to \Space: B \mapsto B_d
	\end{align*}
	preserves sifted colimits, i.e.,\ that $R[x(d)]$ is compact projective in $\Alg^M_R$. By \cite[Prop.\ 5.5.8.25]{LurieHTT}, all objects in $\Poly^M_R$ are now compact projective in $\Alg^M_R$, since these objects are coproducts of the form $R[x(d_1)] \sqcup \dots \sqcup R[x(d_n)]$.
\end{proof}

\begin{Rem}
	We can in fact show that the inclusion $\Poly^M_R \to \Alg^M_R$ induces an equivalence $\PPP_\Sigma(\Poly^M_R) \simeq \Alg^M_R$. Since the proof is somewhat technical, and the result is not essential in what follows, the argument can be found in Appendix \ref{Par:AlgMRpres}.
\end{Rem}

\subsection{Change of rings and of monoids}
Let $\Mon$ be the 1-category of discrete commutative monoids and monoid homomorphisms between them. Following \cite{LurieHTT}, we write $\PrL$ resp.\ $\PrR$ for the category that has presentable categories as objects and functors which are left resp.\ right adjoints as morphisms. 

\begin{Prop}
	\label{Prop:NRadj}
	For $j: M \to M'$ a homomorphism of commutative monoids and $\varphi:R \to R'$ a ring map, we have an adjunction
	\begin{align*}
	\varphi_!j_!: \Alg^{M}_{R} \rightleftarrows \Alg^{M'}_{R'} : j^!\varphi^!
	\end{align*}
	where $j^!\varphi^!$ is given on objects by precomposing with a map $\Poly^{M}_{R} \to \Poly^{M'}_{R'}$ that sends $R[x(d)]$ to $R'[x(j(d))]$, and $\varphi_!j_!$ is given by left Kan extension.
	
	These adjunctions are natural in $M$ and $R$ in the sense that we have a functor $\Alg^{(-)}_{(-)}:\Alg^\op \times \Mon^\op \to \PrR$ which is given by $(\varphi,j) \mapsto j^!\varphi^!$ on morphisms. 
\end{Prop}	

\begin{proof}
	We first give the adjunction in the following two special cases:
	\begin{enumerate}[label=(\roman*)]
		\item Where $\varphi$ is the identity on $\Z$ and $j:M \to M'$ is arbitrary;
		\item Where $\varphi:R\to R'$ is arbitrary and $j$ is the identity on $M$, for some $M$.
	\end{enumerate}
	
	\begin{proof}[Case (i).]
		\let\qed\relax
		We first define the finite coproduct-preserving functor 
		\begin{align*}
		[j]:\Poly^{M} \to \Poly^{M'}: \Z[x(d)] \mapsto \Z[x(j(d))]
		\end{align*}
		It suffices to decide how $[j]$ acts on the following hom-sets
		\begin{align*}
		\Poly^{M}(\Z[x(d)],\Z[x(e_1),\dots,x(e_n)])  \to \Poly^{M'}(\Z[x(j(d))],\Z[x(j(e_1)),\dots,x(j(e_n))])
		\end{align*}
		Now since
		\begin{align*}
		\Poly^{M}(\Z[x(d)],\Z[x(e_1),\dots,x(e_n)]) \simeq \Z[x(e_1),\dots,x(e_n)]_{d}
		\end{align*}
		it  suffices to give natural maps $\Z[x(e_1),\dots,x(e_n)]_{d} \to \Z[x(j(e_1)),\dots,x(j(e_n))]_{j(d)}$. For this, we take the map of $\Z$-modules which is given by $x(e_i)^m \mapsto x(j(e_i))^m$. Since $\Poly^M,\Poly^{M'} $ are discrete categories, naturality is unproblematic.
		
		Write now $j^!$ for the functor $\PPP(\Poly^{M'}) \to \PPP(\Poly^{M})$, given by precomposing with $[j]:\Poly^{M} \to \Poly^{M'}$. Since $[j]$ preserves tensor products, $j^!$ restricts to a functor $j^!:\Alg^{M'} \to \Alg^{M}$, by definition of $\Alg^{M'}$ and $\Alg^{M}$.
		
		By the dual of \cite[Prop.\ 6.4.9]{CisinskiHigh}, we have a left adjoint $j_!: \PPP(\Poly^{M}) \to \PPP(\Poly^{M'})$, which is given on objects by pointwise left Kan extension. From the latter fact and the proof of \cite[Prop.\ 5.5.8.15]{LurieHTT}, \cite[Prop.\ 25.1.1.5]{LurieSpectral}, it follows that $j_!$ restricts to the unique cocontinuous functor $\Alg^{M} \to \Alg^{M'}$, which on $\Poly^{M}$ is given by $[j]$ followed by the inclusion $\Poly^{M'} \to \Alg^{M'}$. 
	\end{proof}
	
	\begin{proof}[Case (ii).]
		\let\qed\relax
		We have an adjunction $\varphi_!:\Alg_{R} \rightleftarrows \Alg_{R'}:\varphi^!$, where the left adjoint is given by taking tensor products and the right adjoint by precomposition with $\varphi$. The right adjoint $\varphi^!$ commutes with the forgetful functors to $\Alg$. On the other hand, we also have endofunctors $(-)\otimes_{R} R'$ on $\Alg$ and $\Alg^{M'}$. Together with $\varphi_!$, these give a morphism of cospans from $\Alg^{M} \to \Alg \leftarrow \Alg_{R}$ to $\Alg^{M} \to \Alg \leftarrow \Alg_{R'}$.
		
		We can thus take fibre products on both sides of the adjunction over $\Alg$ with $\Alg^{M}$, to get the adjunction
		\begin{align*}
		\varphi_!:\Alg^{M}_{R} \rightleftarrows \Alg^{M}_{R'}: \varphi^!
		\end{align*}  
		Observe that by Remark \ref{Rem:AlgRMcoslice}, we can still describe $\varphi^!$ via precomposition with $\varphi$.
	\end{proof}
	
	Now for the general case, consider the adjunction $j_!: \Alg^{M} \rightleftarrows \Alg^{M'}: j^!$ from the first case. Observe that $j_!$ sends $R \in \Alg$ with trivial $M$ grading to $R$ with trivial $M'$ grading. Hence, taking coslices gives an adjunction $j_!: \Alg^{M}_{R} \rightleftarrows \Alg^{M'}_{R}: j^!$ by Remark \ref{Rem:AlgRMcoslice}. Composing with the second case gives us the adjunction
	\begin{align*}
	\varphi_!j_!:\Alg^{M}_{R} \rightleftarrows \Alg^{M'}_{R} \rightleftarrows \Alg^{M'}_{R'}: j^!\varphi^!
	\end{align*}
	Since for $R \to R'$ in $\Alg$ and $j:M \to M'$ in $\Mon$ it holds for all $d \in M$ that $R'[x(j(d))] \simeq R[x(j(d))] \otimes_{R} R'$, we have that $\varphi_!j_!$ restricts to the functor $\Poly^{M}_{R} \to \Poly^{M'}_{R'}: R[x(d)] \mapsto R'[x(j(d))]$. From the adjunction $\varphi_!j_! \dashv j^!\varphi^!$ it follows that $j^!\varphi^!$ is given by precomposition with $\varphi_!j_!$.
	
	It remains to show naturality in $M$ and $R$. The functors $[j]:\Poly^{M} \to \Poly^{M'}$ from the first case give a functor $H:\Mon \to \Cat: M \mapsto \Poly^M$. Since $H$ lands in 1-categories, naturality is unproblematic. Now composition with $H$ induces a functor 
	\begin{align*}
	 \Mon^\op \to \PrR: M \mapsto \Alg^M
	\end{align*}
	which on morphisms is given by sending $j:M \to M'$ to $j^!: \Alg^{M'} \to \Alg^{M}$, as in the first case. Taking coslice categories, we get a functor
	\begin{align*}
	\Mon^\op \to \PrR: M \mapsto \Alg^M_R
	\end{align*}
	for any $R \in \Alg$. By naturality of the coslice construction, this gives us a functor $\Alg^\op \to \Fun(\Mon^\op,\PrR)$, and thus a functor
	\begin{align*}
	\Alg^\op \times \Mon^\op \to \PrR: (R,M) \mapsto \Alg_R^M
	\end{align*}
	which sends $(\varphi,j): (R,M) \to (R',M')$ to the functor $j^!\varphi^!$, which is what we wanted.
\end{proof}

\begin{Rem}
	Taking adjoints, we also get a functor $\Alg^{(-)}_{(-)}:\Alg \times \Mon\to \PrL$, which on morphisms is given by sending $(\varphi,j)$ to $\varphi_!j_!$.
\end{Rem}

\begin{Rem}
	Let $p:\CCC \to \Alg \times \Mon$ be the Cartesian fibration associated to the functor $\Alg^\op \times \Mon^\op \to \PrR$ which sends $(R,M)$ to $\Alg_R^M$ and $(\varphi,j)$ to $j^!\varphi^!$. We can describe $\CCC$ as the category  of triples $(R,M,B)$, where $M$ is a commutative monoid, $R$ a ring, and $B$ an $M$-graded $R$-algebra. A morphism $(R,M,B) \to (R',M',B')$ is given by a ring map $\varphi:R \to R'$, a monoid homomorphism $j:M \to M'$ and a morphism $h:B \to j^!\varphi^!B'$ in $\Alg^{M}_{R}$. Then $p$ is the forgetful functor.
\end{Rem}

\begin{Not}
	As in the proof of Lemma \ref{Prop:NRadj}, we write the adjunction $\id_!j_! \dashv j^!\id^!$ as $j_! \dashv j^!$, and $\varphi_!\id_! \dashv \id^!\varphi^!$ as $\varphi_! \dashv \varphi^!$.
\end{Not}

\begin{Exm} \label{Exm:fjadjs}
	We recover the following adjunctions from Proposition \ref{Prop:NRadj}.
	\begin{enumerate}
		\item For $j$ the trivial map $0 \to 0$ we recover the extension-restriction by scalars adjunction.
		\item \label{Itm:trivgrad}	For $j$ the inclusion $0 \to M'$ we have that $j^! : \Alg^{M'}_R\to \Alg_R$ is the functor $(-)_0$. From the adjunction, it follows that $j_!R[x]$ is $R[x(0)]$. Since $\Alg_R$ is generated by polynomial algebras under colimits, in general $j_!$ therefore sends an $R$-algebra $A$ to the graded algebra which is $A$ concentrated in degree 0. We say that $j_!$ endows $A$ with the \textit{trivial grading}.
		\item For $j$ the map $M \to 0$ it holds that $j_!R[x(d)] = R[x]$, which follows from the proof of Proposition \ref{Prop:NRadj}. Since $j_!$ is cocontinuous, it follows that $j_!$ is the functor which forgets the grading. Observe that $j^!A$ is the derived analogue of the monoid algebra $A[M]$. 
		\item For general $\varphi: R \to R'$ and $j:M \to M'$, given $B$ in $\Alg_{R}^{M}$, the underlying $R'$-algebra of $\varphi_!j_!B$  is $B \otimes_{R} R'$. This follows from the naturality of the adjunctions and the previous examples.
	\end{enumerate}
\end{Exm}

\subsection{Graded, quasi-coherent algebras on algebraic stacks}
\label{Par:GradOS}	
Throughout, fix an algebraic stack $\XXX$ and a scheme $S$. Using Proposition \ref{Prop:NRadj}, we now define the relative case. We are mostly interested in graded, quasi-coherent algebras on schemes: we only need to consider graded, quasi-coherent algebras in one case, namely on the algebraic stack $B\G_m$. We therefore only generalize to algebraic stacks when needed or when it is as much effort as the case of schemes.

\begin{Def}
	\label{Def:AlgMSt}
	Extend the functor $\Alg^{(-)}_{(-)}:\Alg \times \Mon \to \PrL$ to a functor $\Art^\op \times \Mon \to \PrL$, written $(\XXX,M) \mapsto \Alg^M(\XXX)$, by right Kan extension along the inclusion $\Alg \times \Mon \to \Art^\op \times \Mon$ (allowed by \cite[Prop.\ 5.5.3.13]{LurieHTT}). For a stack $\XXX$ and a monoid $M$, we call $\Alg^M(\XXX)$ the category of \textit{$M$-graded, quasi-coherent $\OO_{\XXX}$-algebras}.
\end{Def}

For $f: \XXX' \to \XXX$ a morphism of stacks, and $j: M \to M'$ a morphism of commutative monoids, we have an adjunction $f^*j_! \dashv j^!f_* \colon \Alg^M(\XXX)\rightleftarrows\Alg^{M'}(\XXX')$, by construction of $\Alg^{(-)}_{(-)}$.

\begin{Rem}
	\label{Rem:KanExtProd}
	Let $\CCC,\DDD, \EEE$ be categories, $\CCC_0$ a subcategory of $\CCC$, with $\EEE$ cocomplete, and $F_0: \CCC_0\times \DDD \to \EEE$ a functor. Write  $h$ for the inclusion $\CCC_0 \to \CCC$, and let $F$ be  the left Kan extension of $F_0$ along $(h \times \id_{\DDD})$.

	Observe that for any $(c,d) \in \CCC \times \DDD$, the functor ${\CCC_0}_{/c} \to (\CCC_0 \times \DDD)_{/(c,d)}$ that sends $f:c_0 \to c$ to the pair $(f, \id_{d})$ is a right adjoint, hence cofinal by \cite[Cor.\ 6.1.13]{CisinskiHigh}.\footnote{In \cite{CisinskiHigh}, what we call a cofinal functor is called a final functor.} The functor $F(-,d):\CCC \times \{d\} \to \EEE$ is therefore the left Kan extension of $F_0:\CCC_0 \times \{d\} \to \EEE$ along $(h \times \id_{\{d\}})$.
\end{Rem}

From Remark \ref{Rem:KanExtProd} it follows that $\Alg^M(\XXX)$  is the right Kan extension of $\Alg^M(-)$ along $\Alg \to \Art^\op$, i.e.\ $\Alg^M(\XXX)$ is the limit $\lim_{\Spec R \to \XXX} \Alg^M_R$ of categories, with transition maps $\Alg^M_{R} \to \Alg^M_{R'}$ for $\Spec R'\to \Spec R$ over $\XXX$ induced by pulling back.

\begin{Rem}
	\label{Rem:RelSpec}
	Call a morphism of algebraic stacks $\XXX' \to \XXX$ \textit{affine} if for any scheme $U$ over $\XXX$ the morphism $U_{\XXX'} \to U$ is affine. For the trivial monoid $\{0\}$, we get that $\Alg(\XXX) \coloneqq \Alg^{\{0\}}(\XXX)$ is (contravariantly) equivalent to the category $\Aff_{/\XXX}$ of algebraic stacks affine over $\XXX$. 
	
	This known fact can roughly be seen as follow. We have a functor $\Aff_{/\XXX} \to \Alg(\XXX)$ induced by the functors $\Aff_{/\XXX} \to \Alg_R$, indexed over maps $\Spec R \to \XXX$, that send $\XXX' \to \XXX$ to the $R$-algebra $A$ such that $\Spec A =  \XXX'_R$. To check that this functor is an equivalence, we use that $\St_{/\XXX} \simeq \lim_{\Spec R \to \XXX} \St_{/ \Spec R}$. Thus the statement can be checked affine-locally on $\XXX$, which is clear. 
	
	For a composition $\XXX \xrightarrow{f} \YYY \xrightarrow{g} \ZZZ$ with $g$ and $gf$ affine it holds that $f$ is affine as well. Indeed, $f$ factorizes as the closed immersion $\XXX \to \XXX \times_{\ZZZ} \YYY$, which is affine, followed by the projection $\XXX \times_{\ZZZ} \YYY \to \YYY$. The latter map is also affine, since being affine is stable under base change.
	
	For $\AA$ a quasi-coherent $\OO_{\XXX}$-algebra, we write the corresponding stack over $\XXX$ as $\Spec \AA$.
\end{Rem}
	
Let $\BB$ be an $M$-graded, quasi-coherent $\OO_{\XXX}$-algebra. Write $j$ for the map $M \to 0$. Then the underlying quasi-coherent $\OO_{\XXX}$-algebra $j_!\BB$, written simply as $\BB$, comes with a decomposition $\BB \simeq \bigoplus_{d \in M} \BB_d$ in $\QCoh(\XXX)$, such that affine-locally $\BB_d$ is the graded piece of degree $d$ of $\BB$. This follows from the fact that $j_!$ commutes with $f^*$, for any $f:\Spec R \to \XXX$, together with Remark \ref{Rem:gradpiece}.

\begin{Not}
	We call $\Z$-graded rings and $\Z$-graded algebras simply \textit{graded rings} and \textit{graded algebras}.
\end{Not}

\begin{Exm}
	\label{Exm:Ageq0}
	For $j: \N \to \Z$ the inclusion,  the functor $j^!:\Alg^\Z(\XXX) \to \Alg^\N(\XXX)$ sends $\AA$ to an $\N$-graded algebra which has underlying $\OO_{\XXX}$-module $\AA_{\geq n} \coloneqq \bigoplus_{n \geq 0 } \AA_n$. 
\end{Exm}

\begin{Exm}
	\label{Exm:Sym}
	Consider the adjunction $\Sym_R \dashv U$, where $U$ is the forgetful functor $\Alg_R \to \cn\Mod_R$. Let $j$ be the map $\N \to 0$. 
	
	Writing the presheaf $j^!A$ on $\Poly^\N_R$ as $R[x(d)] \mapsto Ax(1)^d$, we get $j^!A \simeq A[x(1)]$ in the obvious way, for any $R$-algebra $A$. Using the adjunctions, we thus get natural equivalences
	\begin{align*}
	\Alg_R(\Sym_R(UR),A) \simeq A \simeq A[x(1)]_1 \simeq \Alg^\N_R(R[x(1)],A[x(1)]) \simeq \Alg_R(j_!j^!R,A)
	\end{align*}
	for all $A$ in $\Alg_R$. It follows that $\Sym_{A}(A) \simeq j_!j^!A$ for any $R$-algebra $A$, by \cite[Prop.\ 25.2.3.1]{LurieSpectral}, by naturality of the adjunctions in Proposition \ref{Prop:NRadj}, and since $j^!$ commutes with base-change. We write this $R$-algebra as $A[t]$.
	
	Inverting $t$ gives us $A[t^{\pm 1}] \coloneqq A[t,t^{-1}]$, which is equivalent to $i_!i^!A$ where $i$ is the map $\Z \to 0$. Globalizing gives us graded, quasi-coherent $\OO_S$-algebras $\AA[t]$ and  $\AA[t^{\pm 1}]$, for any quasi-coherent $\OO_S$-algebra $\AA$.
	
	Since $A[t] \simeq A  \otimes_\Z \Z[t]$ and $\Z \to \Z[t]$ is a cofibration, we can compute $A[t]$ by adjoining the free variable $t$ levelwise to a simplicial model for $A$. This  shows that $\pi_*(A[t]) \cong \pi_*(A)[t]$ as $\pi_0(A)$-modules. From  \cite[Prop.\ 4.1.18]{LurieDAGV} it follows that also $\pi_*(A[t^{\pm 1}]) \cong \pi_*(A)[t^{\pm 1}]$ as $\pi_0(A)$-modules.
\end{Exm}

\begin{Lem}
	\label{Lem:magic}
	Let $f: \XXX \to \YYY$ be a morphism of algebraic stacks. Write $j$ for the map $\Z \to 0$, so that $j_!$ is the functor that forgets the grading. Then the following diagram commutes
	\begin{center}
		\begin{tikzcd}
			\Alg^\Z(\XXX) \arrow[r, "\varphi_*"] \arrow[d, "j_!"] & \Alg^\Z(\YYY) \arrow[d, "j_!"] \\\
			\Alg(\XXX) \arrow[r, "\varphi_*"] & \Alg(\YYY)
		\end{tikzcd}
	\end{center}
\end{Lem}

\begin{proof}
	Consider the unit $\eta: \id \to j^!j_!$ on $\Alg^\Z(\XXX)$. Upon applying $\varphi_*$ to both sides, we get a map $\varphi_* \to \varphi_*j^!j_! \simeq j^!\varphi_*j_!$, where we have used naturality of $\Alg^{(-)}_{(-)}$ in the equivalence. Now taking the adjoint of this map gives us a natural transformation
	\begin{align*}
		\sigma:j_! \varphi_* \to \varphi_*j_!
	\end{align*}
	between the two ways around the square. We claim that $\sigma$ is an equivalence. Since the question is smooth-local on both $\XXX$ and $\YYY$, we may assume that $f$ is of the form $\Spec B \to \Spec A$. The latter case is obvious.
\end{proof}

\subsection{Graded, quasi-coherent modules on algebraic stacks} 
\label{Par:GradModSt}
\begin{Def}
	The category of \textit{$M$-graded, quasi-coherent $\OO_{\XXX}$-modules} is defined as $\QCoh^M(\XXX) \coloneqq \QCoh(\XXX)^{\times M}$. 
\end{Def}

We often abuse notation, and introduce an $M$-graded, quasi-coherent $\OO_{\XXX}$-module $\{\NN_d\}_{d \in M}$ by only introducing the underlying quasi-coherent $\OO_{\XXX}$-module $\NN \coloneqq \bigoplus_{d \in M} \NN_d$. We say $\NN$ is \textit{connective} when the underlying $\OO_{\XXX}$-module is so, and write the category of connective objects as $\cn\QCoh^M(\XXX)$.

\begin{Lem}
	\label{Lem:Pingrad}
	For any $M$-graded, quasi-coherent, connective $\OO_{\XXX}$-module $\NN$, and all $k \geq 0$, we have a natural isomorphism $\pi_k \NN \cong \bigoplus_{d \in M} \pi_k(\NN_d)$. It follows that $\pi_*(\NN)$ is $M$-graded as $\pi_0(\OO_{\XXX})$-module, with graded pieces of degree $d$ of the form $\pi_*(\NN_d)$.
\end{Lem}

\begin{proof}
	Use the Dold-Kan correspondence.
\end{proof}

\begin{Rem}
	The category $\QCoh^M(\XXX)$ is symmetric monoidal in the following way. Informally, for given $M$-graded, quasi-coherent $\OO_{\XXX}$-modules $\NN, \NN'$, we let $\NN \otimes \NN'$ be the graded module which is $\bigoplus_{e+f=d} \NN_e \otimes \NN'_f$ in degree $d$. 
	
	Formally, we let $M^\otimes$ be the symmetric monoidal category associated to $M$, where the tensor product is given by addition in $M$, and with only identity morphisms. Then $\QCoh^M(\XXX)$ is the category of functors $M \to \QCoh(\XXX)$. Thus, by Day convolution we get a symmetric monoidal structure on $\QCoh^M(\XXX)$, which is also available in the $\infty$-setting by \cite{SaulDay}. Since $M$ has no non-identity morphisms, this monoidal structure is particularly easy to compute, which recovers the informal description.
\end{Rem}

\begin{Rem}
	When $S=\Spec R$, then $\cn\Mod^M_{ R} \coloneqq \cn\QCoh^M  (\Spec R)$ can be computed as the $\infty$-category associated to the model category of $M$-graded simplicial $R$-modules.
\end{Rem}

\subsection{Graded, connective $\E_\infty$-algebras}
\label{Par:GradConEinft}
For a symmetric monoidal category $\CCC$, write $\Alg(\CCC)$ for the category of commutative algebra objects in $\CCC$ as in \cite[Def.\ 2.1.3.1]{LurieHA}. Then $\Alg(\cn\QCoh  ^M(S))$ is the category of \textit{$M$-graded, connective  $\mathbb{E}_\infty$-algebras over $\OO_S$}. By \cite[Prop.\ 2.12]{SaulDay}, this category is equivalent to the category of lax symmetric monoidal functors $M^\otimes \to \cn\QCoh  (S)$. 

If $M = \{0\}$ and $S=\Spec R$, this recovers the notion of connective $\E_\infty$-algebras over $R$ by the proof of \cite[Prop.\ 2.5.1.2]{LurieSpectral}. 
\begin{Lem}
	\label{Lem:GrE00}
	The map $\Alg(\cn\QCoh  ^M(S)) \to \lim_{\Spec R \to S} \Alg(\cn\Mod_{ R}^M)$, given by restriction, is an equivalence.
\end{Lem}
\begin{proof} 
	A lax symmetric monoidal functor $M^\otimes \to \cn\QCoh   (S)$ is the same thing as a compatible family of lax symmetric monoidal functors $M^\otimes \to \cn\Mod_{ R}$, indexed by all maps $\Spec R\to S$, in the category $\CCC$ of symmetric monoidal categories and lax monoidal functors between them. 
\end{proof}
\begin{Rem}
	\label{Rem:Einfty}
	We have a map $F: \Alg^M(S) \to \Alg(\cn\QCoh  ^M(S))$. Indeed, by Lemma \ref{Lem:GrE00}, it suffices to give maps when $S = \Spec R$, natural in $R$. But this is straightforward: given an $M$-graded $R$-algebra $B$, we may assume that the underlying $M$-graded ring is a simplicial $M$-graded ring. Then we have a strictly commutative diagram in the category of $M$-graded, simplicial $R$-modules, which gives a commutative algebra object in $\cn\QCoh  ^M(R)$. 
	
	The map $F$ is in general  not an equivalence by \cite[\S 25.1.2]{LurieSpectral}. For example, when $M = \{0\}$, we know that $F$ in general is only an equivalence in characteristic 0.	
\end{Rem}

Composing the functor $F: \Alg^M(S) \to \Alg(\cn\QCoh  ^M(S))$ with the forgetful functor $\Alg(\cn\QCoh  ^M(S)) \to \cn\QCoh  ^M(S)$ gives us a functor $U:\Alg^M(S) \to \cn\QCoh  ^M(S)$. For $\BB \in \Alg^M(S)$ we call $U\BB$ the underlying graded, quasi-coherent $\OO_S$-algebra. As in the ungraded case, we might suppress the notation $U$.

\subsection{Homogeneous localization}
Let $R$ be a ring and $B$ an $\N$-graded $R$-algebra.  We will recover the following classical picture:

\begin{Rem}
	\label{Rem:DisGrad}
	If $B$ is discrete and $M$ is a discrete, graded $R$-module, then for $f \in R$ homogeneous of degree 1, the underived tensor product $M_f \coloneqq M \otimes_R R_f$ is graded, where the homogeneous elements of degree $d-k$ are of the form $m/f^k$, with $m \in M$ homogeneous of degree $d$. We write the graded piece of degree 0 as $M_{(f)}$. 
	
	For a non-graded discrete $R$-module $N$, the underived tensor product $N\otimes_R R[t^{\pm 1}]$ is graded in the obvious way. For $M$ as above and $f \in R$ homogeneous of degree 1, the map $M_{(f)}\otimes_R R[t^{\pm 1}] \to M_f: x\otimes t^n \mapsto xf^n$ is an isomorphism. 
\end{Rem}

Recall that the ungraded case of localization has been generalized to the derived setting in \cite[Prop.\ 4.1.18]{LurieDAGV}. Our approach will be similar.

Let $A$ be an $R$-algebra, and put $A^\times \coloneqq \Alg_R(R[t^{\pm 1}],A)$. Recall that $\pi_0(A^\times) \cong \pi_0(A)^\times$, and that $\pi_n(A^\times) \cong \pi_n(A)$ for $n>0$. Hence $A^\times \to A$ is monic in $\Space$.

\begin{Def}
	Let $f$ be a homogeneous point in $B$. A map of $\Z$-graded $R$-algebras $\lambda:B \to B'$ exhibits $B'$ as the \textit{homogeneous localization} of $B$ at $f$ if for any other $\Z$-graded $R$-algebra $C$ it holds that composition with $\lambda$ induces an equivalence
	\begin{align*}
	\Alg^\Z_R(B',C) \simeq \Alg^\Z_{R,f}(B,C),
	\end{align*}
	where $\Alg^\Z_{R,f}(B,C)$ is the subspace of $\Alg^\Z_{R}(B,C)$ spanned by those maps that send $f$ into $C^\times$.
\end{Def}

\begin{Prop}
	Let $f$ be a homogeneous point in $B$ of degree $d$. Then there is a homogeneous localization of $B$ at $f$, written $B\to B_f$. Moreover, this localization is unique in $\Alg^\Z_R$ up to contractible choice, and the underlying $R$-algebra of $B_f$ corresponds to the ungraded localization at $f$.
\end{Prop}

\begin{proof}
	For existence, we define $B_f$ as the pushout
	\begin{center}
		\begin{tikzcd}
			R[s] \arrow[d] \arrow[r, "s\mapsto f"] & B \arrow[d] \\
			R[s^{\pm 1}] \arrow[r] & B_f
		\end{tikzcd}
	\end{center}
	in $\Alg^\Z_R$, with $\deg(s) =d$. We claim that $B_f$ is  the homogeneous localization of $B$ at $f$.
	
	For $C \in \Alg^\Z_R$ we have have that $\Alg^\Z_R(R[s^{\pm 1}],C )$ is equivalent to the subspace of $C_d$ spanned by those elements which are invertible in $\pi_0 C$, i.e.,
	\begin{align*}
	\Alg^\Z_R(R[s^{\pm 1}],C ) \simeq C_d \times_C C^\times
	\end{align*}
	By construction of $B_f$ and by composing pullback squares, we thus get that 
	\begin{center}
		\begin{tikzcd}
			\Alg^\Z_{R}(B_f,C) \arrow[r] \arrow[d] & C^\times\arrow[d] \\
			\Alg^\Z_R(B,C) \arrow[r] & C
		\end{tikzcd}
	\end{center}
	is Cartesian. The claim follows from the fact that the bottom horizontal arrow maps $\varphi$ to $\varphi(f)$, and that the right vertical arrows is monic.
	
	Unicity is clear. Let $j$ be the map $\Z \to 0$. It remains to show that $j_!(B_f) \simeq (j_! B)_f$. To this ends, by adjunction and the universal property of graded localization, we have that
	\begin{align*}
	\Alg_R(j_!(B_f),A) \simeq \Alg^\Z_{R,f}(B,A[t^{\pm 1}])
	\end{align*}
	where $\deg (t) = 1$. Now the equivalence $\Alg_R(j_! B, A ) \simeq \Alg_R^\Z(B,A[t^{\pm 1}])$ gives an equivalence $\Alg_{R,f}(j_!B, A)\simeq\Alg^\Z_{R,f}(B,A[t^{\pm 1}])   $. Indeed, for a given $\varphi: j_! B \to A$ with transpose $\bar{\varphi}:B \to A[t^{\pm 1}]$, the map $\pi_0(\bar{\varphi})$ sends $f$ to $\varphi(f)t^d$. Now since $\pi_0(A[t^{\pm 1}]) \cong \pi_0(A)[t^{\pm 1}]$ by Example \ref{Exm:Sym}, the element $\varphi(f)t^d$  is invertible in $\pi_0(A[t^{\pm 1}])$ if and only if $\varphi(f)$ is invertible in $\pi_0(A)$.
\end{proof}

For $f$ a homogeneous point in $B$ of degree 1, write $B_{(f)}$ for the degree zero part of $B_f$. We then have the following.

\begin{Lem}
	\label{Lem:Deg0loc}
	The natural map $B_{(f)} [t^{\pm 1}] \to B_f$, sending $t$ to $f$, is an equivalence of graded $R$-algebras.
\end{Lem}

\begin{proof}
	Consider $\pi_*(B)$ as $\pi_0(B)$-module. Then $\pi_*(B_{(f)}[t^{\pm 1}]) \cong \pi_*(B_{(f)})[t^{\pm 1}]$ by Example \ref{Exm:Sym}, and we have the isomorphisms
	\begin{align*}
	\pi_*(B_{(f)})[t^{\pm 1}] = \pi_*((B_f)_0)[t^{\pm 1}] \cong (\pi_*(B)_f)_0[t^{\pm 1}] \cong \pi_*(B)_f \cong \pi_*(B_f)
	\end{align*}
	Here, we have used: the definition of $B_{(f)}$; Lemma \ref{Lem:Pingrad}; Remark \ref{Rem:DisGrad} and \cite[Prop.\ 4.1.18]{LurieDAGV}. 
\end{proof}

\subsection{Irrelevant ideal}
\label{Par:IrrId}
Let $i$ be the map $0 \to \N$ and $j$ the map $\N \to 0$. We can summarize the adjunctions $i_! \dashv i^!$ and $j_! \dashv j^!$ as explained in Example \ref{Exm:fjadjs} in the following diagram
\begin{center}
	\begin{tikzcd}
		\Alg(S) \arrow[r, bend left, "\textrm{trivial}", "i_!"'] & \Alg^\N(S) \arrow[r, bend left, "\textrm{forget}", "j_!"']  \arrow[l, bend left, "{(-)_0}", "i^!"'] & \Alg(S) \arrow[l, bend left, "{(-)[t]}", "j^!"']
	\end{tikzcd}
\end{center}

Throughout, write $\iota: i_!i^! \to \id$ for the counit of the adjunction $i_! \dashv i^!$. In the classical case, the map $\iota_{\BB}: \BB_0 \to \BB$ has a retraction $\BB \to \BB_0$ with fibre the irrelevant ideal $\BB_+$. This map $\BB \to \BB_0$ comes about since there is also an adjunction $i^! \dashv i_!$. We will now generalize this to the derived picture.

Recall that a functor is \textit{conservative} if it reflects equivalences.

\begin{Prop}
	\label{Prop:SpclAdj}
	There is an adjunction $i^! :\Alg^\N(S) \rightleftarrows \Alg(S) : i_!$ with counit $\pi: \id \to i_!i^!$ such that $\pi \circ \iota \simeq \id$. 
\end{Prop}

\begin{proof}
	We first do the case $S = \Spec \Z$.
	
	Write $\sAlg$ resp.\ $\sAlg^\N$ for the model 1-category of simplicial rings resp.\ simplicial, $\N$-graded rings. Write $\tau$ for the inclusion $\sAlg \to \sAlg^\N$ that endows $A \in \sAlg$ with the trivial grading. Then the functor $\sigma: \sAlg^\N \to \sAlg$ that sends $B$ to $B_0$ is both a left and right adjoint to $\tau$. 
	
	In fact, $\tau$ is both left and right Quillen. Indeed, it is right Quillen since it preserves fibrations and trivial fibrations, and it is left Quillen since $\sigma$ preserves fibrations and trivial fibrations. Both statements follow from the fact that fibrations and weak equivalences in $\sAlg^\N_R$ are defined degree-wise.
	
	Consequently, both $\tau$ and $\sigma$ preserve cofibrations and trivial cofibrations. Therefore, the left derived functor of $\tau$ can be computed by composition with a fibrant-cofibrant replacement, and thus coincides with the right derived functor. The same holds for $\sigma$.
	
	By \cite[Prop.\ 1.5.1]{HinichDwyer}, the derived functors of a Quillen adjunction induce an adjunction on the associated $\infty$-categories. Since the  derived functor of $\tau$ is $i_!$ and of $\sigma$ is $i^!$, we indeed have an adjunction $i^! \dashv i_!$.
	
	As in the proof of Proposition \ref{Prop:NRadj}, Claim (ii), we promote $i^! \dashv i_!$ to an adjunction $\Alg^\N_R \rightleftarrows \Alg_R$ with counit $\pi$, for any $R \in \Alg$. To globalize this to $S$, observe that $j_!i_!$ is the identity on $\Alg(S)$, and that $j_!$ is conservative, hence that $i_!$ preserves all limits and colimits. We therefore have an adjunction
	\begin{align*}
		i^\sharp \dashv i_! : \Alg^\N(S) \rightleftarrows \Alg(S)
	\end{align*}
	Now $i^\sharp$ coincides with $i^!$ because this is true affine-locally by the case $S = \Spec R$, and since left adjoints are unique  up to contractible choice.
	
	It remains to show that for $\BB \in \Alg^\N(S)$, the composition $\pi_{\BB} \iota_{\BB}: \BB_0 \to \BB \to \BB_0$ is an equivalence. This can be checked locally, so we assume that $S = \Spec R$ and that $\BB$ corresponds to $B \in \Alg^\N(R)$.
	
	Consider $B$ as a pair $(B_0,B) \in \Alg_R \times_{\Alg} \Alg^\N$. Then $\pi_B \iota_B$ is the composition
	\begin{align*}
	(B_0,B_0) \to (B_0,B) \to (B_0, B_0)
	\end{align*}
	induced by the identity on $B_0$ and by the maps $B_0 \to B \to B_0$ from the adjunctions in the case $R = \Z$. Hence, it suffices to do the case $R = \Z$. But this case follows from the classical case, together with the fact that both $\tau$ and $\sigma$ preserve trivial fibrations and trivial cofibrations.
\end{proof}
 
\begin{Def}
	For $\BB \in \Alg^\N(S)$, write $\BB_+$ for the graded, quasi-coherent $\OO_S$-module $\bigoplus_{n>0} \BB_n$. We call $\BB_+$ the \textit{irrelevant ideal}. 
\end{Def}

\begin{Prop}
	The sequence $\BB_+ \to \BB \xrightarrow{\pi_\BB} \BB_0$ is a fibre sequence.
\end{Prop}

\begin{proof}
	The question is local on $S$, so assume that $S = \Spec R$ and that $\BB$ corresponds to $B \in \Alg^\N_R$. Since $B_0 \to B \to B_0$ is a sequence of graded maps which composes into the identity, the map $B \to B_0$ is given by the projection on the underlying graded $R$-modules. 
	
	Let $F$ be the fibre of $B \to B_0$. Since the map
	\begin{align*}
	\pi_n(B) \cong \bigoplus\nolimits_{d \geq 0} \pi_n(B_d) \to \pi_n(B_0)
	\end{align*}
	is given by the projection, it is surjective. The claim now follows by considering the long exact sequence in homotopy groups associated to $F \to B \to B_0$.
\end{proof}

The surjection $\BB \to \BB_0$ gives us a closed immersion $\Spec \BB_0 \to \Spec \BB$, which we write as $V(\BB_+) \to \Spec \BB$. 

\section{Equivariant geometry}
\label{Sec:EquivGeom}
We review some theory on group actions, following \cite[\S 6.1.2]{LurieHTT}, \cite{NikolausPrincipal}.
\subsection{Groupoids and groups}
Let $\CCC$ be a category with fibre products. We say that a simplicial object $X:{\bDelta}^\op \to \CCC$ \textit{satisfies the Segal condition} if for all $n\geq 1$ the map $X_n \to X_1 \times_{X_0} \dots \times_{X_0} X_1$ induced by the spine inclusion is an equivalence, and that \textit{morphisms are invertible} if the map $X_2 \to X_1 \times_{X_0} X_1$ induced by $\Lambda^2[2] \to \Delta[2]$ is an equivalence. 

A \textit{groupoid object} in $\CCC$ is a functor $\GGG:\bDelta^\op \to \CCC$ which satisfies the Segal condition and such that morphisms are invertible. The \textit{category of groupoid objects in $\CCC$} is the full subcategory of $\Fun(\bDelta^\op,\CCC)$ spanned by groupoid objects. 

\begin{Def}
	A \textit{group object in $\CCC$} is a groupoid object $\GGG$ in $\CCC$ such that $\GGG_0$ is contractible.
\end{Def}
We write a group object $\GGG$ in $\CCC$ with $\GGG_1 = G$ as $B(*,G)$. We often suppress  notation and refer to $\GGG$ by $G$. 

\begin{Exm}
	\label{Exm:Ggrpd}
	Let $H$ be an abstract group. Consider $H$ as a one-point category. Then the nerve of $H$ gives us a group object in $\Set$. 
	
	Applying the same procedure to a classical group scheme $G$ gives us a simplicial diagram in $\St$ which sends $[n]$ to $G^{n}$. This is a group object in $\St$, and it has the property that the colimit of this diagram is the familiar stack $BG$.	
\end{Exm}

\subsection{Group actions}
Let $\CCC$ still be a category with fibre products, and let $G$ be a group object in $\CCC$.

\begin{Def}
	A \textit{$G$-object in $\CCC$} is a groupoid object in $\CCC$ of the form
	\begin{align*}
	B(P,G):\bDelta^\op \to \CCC: [n] \mapsto G^{n} \times P
	\end{align*}
	such that $d_1: G \times P \to P$ is the projection, and such that the levelwise projection $B(P,G) \to B(*,G)$ is a morphism of groupoid objects. We call $d_0:G \times P \to P$ the \textit{action of $G$ on $P$} and $P$ the \textit{underlying object} of $B(P,G)$.
	
	The category $\CCC^G$ of $G$-objects is the full subcategory of groupoid objects over $G$ spanned by  $G$-objects. Morphisms of $G$-objects are called \textit{$G$-equivariant maps}.
\end{Def}

We refer to a $G$-object $B(P,G)$ as $P$ when convenient, similarly as in the classical case.

\begin{Exm}
	\label{Exm:GrpActSimp}
	Let $G$ be a classical group scheme acting on a classical scheme $X$. Then consider the action category, which has the points of $X$ as objects and as morphisms $x \to y$ those $g\in G$ such that $gx = y$. The nerve of this category is the quotient groupoid $B(X,G)$, which sends $[n]$ to $G^{n} \times X$. We thus think of a one-simplex $(g,x)$ in $B(X,G)$ as an edge from $x$ to $gx$. Let us spell out the general formulas for the face and degeneracy maps for future reference and to fix the precise indexing.
	
	For $0 \leq i \leq n$, we have that
	\begin{align*}
	&&d_i: G^{n} \times X &\to G^{ (n-1)} \times X \\
	&&(g_1,\dots,g_n,x) & \mapsto \begin{cases}
	(g_1,\dots,g_{n-1},g_n x) & \text{ if } i=0 \\
	(g_1,\dots,g_{n-i}g_{n-i+1},\dots,g_n,x) & \text{ if } 0 < i < n \\
	(g_2,\dots,g_n,x) & \text{ if } i=n
	\end{cases}
	\end{align*}
	Likewise, for $0 \leq j \leq n$ the map $s_j: G^{n} \times X \to G^{ (n+1)}\times X$ sends the element $(g_1,\dots,g_n,x)$ to $(g_1,\dots,g_{n-j},e,g_{n-j+1},\dots,x)$.
	
	Of particular interest is the case when $G = \G_m$ and $X = \Spec B$. In this case $B(X,G)$ induces and is determined by the corresponding cosimplicial diagram $C(B,\Z[t^{\pm 1}])$ in the category of discrete rings that sends $[n]$ to $B[t_1^{\pm 1},\dots,t_n^{\pm 1}]$. The coface maps are given by, for $0 \leq i \leq n$
	\begin{align*}
	d^i: B[t_1^{\pm 1},\dots,t_{n-1}^{\pm 1}] & \to B[t_1^{\pm 1},\dots, t_n^{\pm 1}] \\
	bt_1^{e_1}\cdots t_{n-1}^{e_{n-1}} & \mapsto \begin{cases}
	t_1^{e_1}\cdots t_{n-1}^{e_{n-1}} \sum_d b_d t_n^d & \text{ if } i =0 \\
	bt_1^{e_1} \cdots (t_{n-i}t_{n-i+1})^{e_{n-i}} \cdots t_n^{e_{n-1}} & \text{ if } 0 < i < n \\
	bt_2^{e_1} \cdots t_n^{e_{n-1}} & \text{ if } i = n
	\end{cases}
	\intertext{where the $b_d$ are determined by $\sigma(b) = \sum_d b_d t^d$ for $\sigma: B \to B[t^{\pm 1}]$ the coaction. The codegeneracy maps $s^j:B[t_1^{\pm 1},\dots,t_{n+1}^{\pm 1}] \to B[t_1^{\pm 1},\dots,t_n^{\pm 1}]$ for $0\leq j \leq n$ are $B$-algebra maps  that send the elements $t_1,\dots,t_{n+1}$ to $t_1,\dots,t_{n-j},1,t_{n-j+1},\dots,t_n$.}
 	\end{align*}
\end{Exm}

\subsection{Principal bundles}
\label{Par:PrinBun}
Suppose now that $G$ is a group object in $\St$. For example, $G$ can be any classical group scheme. Then for $B(P,G)$ a given $G$-action in $\St$, we define the corresponding \textit{quotient stack} $[P/G]$ as the colimit of $B(P,G)$ in $\St$.

\begin{Def}
	Let $T$ be a stack. A \textit{principal $G$-bundle over $T$} is a morphism $B(P,G) \to T$ together with a $G$-action $B(P,G)$ such that $P \to T$ is colimiting, i.e., gives an equivalence $[P/G] \simeq T$. 	A \textit{morphism of principal $G$-bundles} is a morphism of $G$-objects $B(P,G) \to B(P',G)$ that fixes $T$. This gives the category $\Bun(T,G)$ of principal $G$-bundles over $T$.
\end{Def}

The main result here is that $* \to BG$ classifies principal $G$-bundles. That is, by \cite{NikolausPrincipal} we have an equivalence of categories
\begin{align*}
\St(T,BG) \to \Bun(T,G) : (T \to BG) \mapsto (T \times_{BG} B(*,G) \to T)
\end{align*}
In fact, pulling back along $B(*,G) \to BG$ gives us an equivalence $\St_{/BG} \to \St^G$.

\begin{Exm}
	Let $S$ be a  scheme. Let $G$ be a group scheme over $S$ acting on a  scheme $X$ over $S$. Then for $T \to S$ in $\Sch$ it holds that $\St_{/S}(T,[X/G])$ is equivalent to the space of diagrams
	\begin{center}
		\begin{tikzcd}
		P \arrow[d] \arrow[r] & X \arrow[d] \\
		T \arrow[r] & S
		\end{tikzcd}	
	\end{center}
	where $P \to T$ is a principal $G$-bundle and $P \to X$ is $G$-equivariant.  Indeed, for given $T \to [X/G]$ we can form the pullback diagrams
	\begin{center}
		\begin{tikzcd}
		P \arrow[r] \arrow[d] & X \arrow[r] \arrow[d] & S  \arrow[d] \\
		T \arrow[r] & {[X/G]} \arrow[r] & BG
		\end{tikzcd}
	\end{center}
\end{Exm}

\begin{Rem}
	\label{Rem:ActClas}
	Let $X$ be a $G$-object in any category, with $d_0 = \sigma$ the action. Write the face maps of $G$ as $\delta_i$. From the Segal condition it follows that $d_i: G^n \times X \to G^{n-1} \times X$ is homotopic to $(\id_{G^{n-1}} , \sigma)$ if $i=0$, and to $(\delta_i, \id_X)$ otherwise. Likewise, it follows that $s_j$ is homotopic to the map which is the unit of $G$ in the $j$-th coordinate, and the identity on the other coordinates.
\end{Rem}

\subsection{$B\G_{m,S}$-algebras}
Let $S$ be a scheme. Consider $\G_{m,S} \coloneqq \G_m \times S$ as a group object in $\St_{/S}$. For $S=\Spec R$ we write $\G_{m,S}$ as $\G_{m,R}$. For $\varphi:S' \to S$ a morphism of schemes we have that $B\G_{m,S'}$ is the pullback of $B\G_{m,S}$ along $\varphi$. This is because pullbacks commute past colimits in $\St$, as they do in any topos by universality of colimits.
\begin{Def}
	The category $\Aff^{\G_m}(S)$ of \textit{$\G_{m,S}$-schemes affine over $S$} is the full subcategory of $\St^{\G_{m,S}}_{/S}$ spanned by those $\G_{m,S}$-stacks which are schemes  affine over $S$.
\end{Def}

When $S = \Spec R$ or when $S = \Spec \Z$ we write $\Aff^{\G_m}(S)$ as $\Aff^{\G_m}_R$ or as $\Aff^{\G_m}$ respectively. Write $\Alg^{\G_m}$ for the opposite of $\Aff^{\G_m}$, i.e.,\ this is the category of rings with a $\G_m$-coaction.\footnote{It would be more correct to call this a $\Gamma(\G_m)$-coaction. We have suppressed the $\Gamma$ to ease notation.} Likewise, define $\Alg_R^{\G_{m}} \coloneqq (\Aff^{\G_m}_R)^\op $ and $\Alg^{\G_m}(S) \coloneqq \Aff^{\G_m}(S)^\op$.

\begin{Prop}
	\label{Prop:AlgBGm}
	Pulling back along the quotient map $\rho: S \to B\G_{m,S}$ induces an equivalence $\Alg(B\G_{m,S}) \to \Alg^{\G_{m,S}}(S)$.
\end{Prop}

\begin{proof}
	Recall from \S \ref{Par:PrinBun} that pulling back along $\rho$ gives us an equivalence 
	\begin{align*}
	(\St_{/S})_{/B\G_{m,S}} \to \St_{/S}^{\G_{m,S}}
	\end{align*}
	Since $\rho$ is affine, it holds that $\XXX \to B\G_{m,S}$ is affine if and only if $\XXX_S  \to S$ is, for any algebraic stack $\XXX$. Thus the claim follows from Remark \ref{Rem:RelSpec}. 
\end{proof}

When convenient, we will tacitly identify $\Alg^{\G_m}(S)$ with $\Alg(B\G_{m,S})$ via the equivalence above.

\begin{Rem}
	\label{Rem:AlgGMloc}
	We actually have a functor $\Alg^{\G_m}(-): \Sch^\op \to \PrL$, since $\St^{\G_{m,S}}_{S}$ is functorial in $S$ and since being affine is stable under pullbacks. The equivalence in Proposition \ref{Prop:AlgBGm} is natural in $S$, hence $\Alg^{\G_m}(-): \Sch^\op \to \PrL$ is the right Kan extension of the restriction $\Alg^{\G_m}(-): \Alg \to \PrL$. It follows that the map 
	$H:\Aff^{\G_m}(S) \to \lim_{\Spec R \to S} \Aff^{\G_m}_R$
	is an equivalence.
\end{Rem}

\begin{Exm}
	\label{Exm:AlgBGM}
	Let $\rho:S \to B\G_{m,S}$ be the projection map and $\pi: B\G_{m,S} \to S$ the structure map. Then $\rho_*$ sends $\RR \in \Alg(S)$ to $\RR[u^{\pm 1}] \in \Alg^{\G_{m,S}}(S)$ with diagonal coaction of $\G_{m,S}$, and $\rho^*$ is the functor which forgets the coaction. The functor $\pi^*$ endows $\RR \in \Alg(S)$ with trivial action, and $\pi_*$ sends $\AA \in \Alg^{\G_{m,S}}(S)$ to the invariants $\AA^{\G_m}$.
\end{Exm}

\section{Graded algebras and $\G_m$-actions}
\label{Sec:GradAlgGm}

Fix a base scheme $S$. The goal of this section is to show the following.
\begin{reptheorem}{Thm:Equiv}
	There is a contravariant equivalence from the category of quasi-coherent, $\Z$-graded $\OO_S$-algebras to the category of $\G_{m,S}$-schemes affine over $S$.
\end{reptheorem}

We write the simplicial diagram of a $\G_{m,S}$-scheme $X = \Spec \AAA$ affine over $S$ as $B(X) = \{B_n(X)\}_n$. Recall that $B(X)$ lives over the trivial $\G_{m,S}$-scheme $B(S) = B(S,\G_{m,S})$. Observe that $B(X)$ determines and is determined by a cosimplicial diagram in $\Alg(S)$, which we write as $C(\AAA) = \{C^n(\AAA)\}_n$. Then $C(\AAA)$ lives under the cosimplicial diagram $C(\OO_S)$ corresponding to the trivial $\G_{m,S}$-action $B(S)$ on $S$. We write $C^n(\AAA)$ as $\AAA[t_1^{\pm 1},\dots,t_n^{\pm 1}]$, where the indexing is as in Example \ref{Exm:GrpActSimp}. We use the obvious similar notation in the affine case. In particular, the cosimplicial diagram associated to the trivial $\G_m$-action on $\Spec \Z$ is written as $C(\Z)$.

As for $\G_m$-schemes, we often abuse notation by introducing objects of $\Alg^{\G_m}$ by saying that $A$ is a ring with $\G_m$-coaction, with associated cosimplicial diagram $C(A)$. 

\subsection{The affine $\G_m$-scheme associated to a graded ring}
\label{Subsec:AffGMtoGR}
Observe that $\Alg^{\G_m}$ is a full subcategory of $\Alg^{\bDelta}_{C(\Z)/}$, and that $\Alg^{\bDelta}_{C(\Z)/}$ is cocomplete by \cite[Cor.\ 5.1.2.3]{LurieHTT}. For a given diagram $F: K \to \Alg^{\G_m}$, the colimit in $\Alg^{\bDelta}_{C(\Z)/}$ is computed pointwise, and lands in $\Alg^{\G_m}$, since $(-)\otimes \Z[t^{\pm 1}]$ commutes with colimits. Hence $\Alg^{\G_m}$ is cocomplete.

We have an obvious coproduct-preserving map $\Poly^\Z  \to \Alg^{\G_m}$, which sends $\Z[x(d)]$ to the cosimplicial diagram associated to the coaction $\Z[x(d)] \to \Z[x(d),t^{\pm 1}]$ determined by $x(d) \mapsto x(d)t^d$. By left Kan extension, we get a functor
\begin{align*}
\Phi: \Alg^\Z \to \Alg^{\G_m}
\end{align*}
Since $\Phi$ preserves finite coproducts in $\Poly^\Z$, this functor $\Phi$ preserves all small colimits. See \cite[Prop.\ 5.5.8.15]{LurieHTT}, \cite[Prop.\ 25.1.1.5]{LurieSpectral} and their proofs.

\begin{Rem}
	\label{Rem:PhionGSR}
	We can describe $\Phi$ explicitly on fibrant-cofibrant simplicial graded rings. Let $B$ be such an object. Then the grading is levelwise, hence we have a levelwise coaction of $\G_m$, which extends levelwise to a cosimplicial diagram in the category of discrete rings. Combining these levels gives us a cosimplicial diagram in $\Alg$, which recovers $\Phi(B)$.
\end{Rem}

\begin{Prop}
	The cocontinuous functor $\Phi$ has a right adjoint.
\end{Prop}

\begin{proof}
	This follows from \cite[Cor.\ 5.5.2.9, Rem.\ 5.5.2.10]{LurieHTT}, since $\Alg^\Z$ is presentable and $\Phi$ cocontinuous.
\end{proof}

From here on, take a right adjoint $\Psi$ to $\Phi$, so that we are in the situation
\begin{align}
\label{Eq:Adj}
\Phi: \Alg^\Z \rightleftarrows  \Alg^{\G_m}: \Psi
\end{align}

\subsection{Twisted coaction}
In this paragraph, fix $A$ in $\Alg^{\G_m}$ with cosimplicial diagram $C(A)$. 

\begin{Rem}
	\label{Rem:Adcl}
Recall that if $A$ is classical, then we know that it has a $\Z$-grading. The graded piece $A_d$ is the (underived) equaliser of $A \rightrightarrows A[t^{\pm 1}]$, where one map is the coaction $A \to A[t^{\pm 1}]$ multiplied with $t^{-d}$, and the other map is the inclusion. We will generalize this formula for $A_d$ to the derived setting, which will help us both in showing that $\Phi$ is fully faithful and in showing that $\Psi$ is conservative.	
\end{Rem}

\begin{Def}
	Let $C(A[u^{\pm 1}])$ be the cosimplicial diagram in $\Alg$ associated to the diagonal $\G_m$-action on $\Spec A \times \Spec \Z[u^{\pm1}]$. 
\end{Def}

We grade the rings $C^n(A[u^{\pm 1}])$ via the grading on $\Z[u^{\pm 1}]$ which has $u$ in degree $-1$. Then $C(A[u^{\pm 1}])$ is actually a cosimplicial diagram in $\Alg^\Z$.

\begin{Def}
	\label{Def:dtwist}
	For $d\in \Z$ define \textit{the $d$-th twist of $A$} as the cosimplicial space 
	\begin{align*}
	C(A(d)):\bDelta \to \Space : [n] \mapsto \Alg^\Z\Big(\Z[x(d)],C^n(A[u^{\pm 1}]) \Big)
	\end{align*}
	where the cosimplicial structure is induced by $C(A[u^{\pm 1}])$.
\end{Def}
Observe that $C(A(d))$ sends $[n]$ to $C^n(A)$. It also has the same coface and codegeneracy maps as $C(A)$, except that the maps $d^0:C^{n-1}(A(d)) \to C^n(A(d))$ are multiplied with $t_n^{-d}$, for all $n$.

Recall that a \textit{coaugmentation} of a cosimplicial diagram $X  :\bDelta \to \CCC$ in a given category $\CCC$ is an extension of $X$ to a diagram on the category $\bDelta_+$ of finite ordinals and order-preserving maps.  A coagumentation is \textit{limiting} if the induced map $X(\emptyset) \to \lim X$ is an equivalence.
\begin{Def}
	\label{Def:alpha}
	Let $d\in \Z$. The map 
	\begin{align*}
	\Alg^{\G_m}(\Phi(\Z[x(d)]),A) \to \Alg(\Z[x(d)],A) \simeq \Alg^\Z(\Z[x(d)],C^0(A[u^{\pm 1}])) 
	\end{align*}
	that sends $f$ to $f^0$ induces a coaugmentation
	\begin{align*}
	\alpha_{A,d}: \Alg^{\G_m}(\Phi(\Z[x(d)]),A) \to C(A(d))
	\end{align*}
	in $\Space$. Since the $\alpha_{A,d}$ are natural in $d$, by definition of $\Alg^\Z$ and the adjunction $\Phi \dashv \Psi$, these maps induce a coaugmentation
	\begin{align*}
	\alpha_A: \Psi(A) \to C(A[u^{\pm 1}])
	\end{align*}
	in $\Alg^\Z$.
\end{Def}

\subsection{Equivariant mapping spaces}
Let still $A \in \Alg^{\G_m}$ be given. Observe: if we know that $\Phi \dashv \Psi$ is an adjoint equivalence, then the grading on $A$ must be such that 
\begin{align*}
	A_d \simeq \Alg^\Z(\Z[x(d)],\Psi(A)) \simeq \Alg^{\G_m}(\Phi(\Z[x(d)]),A)
\end{align*}
We will show that $\Alg^{\G_m}(\Phi(\Z[x(d)]),A)$ is indeed a model for $A_d$, by showing that it is a limit of $C(A(d))$, which generalizes the classical picture by Remark \ref{Rem:Adcl}. 
\begin{Prop}
	\label{Prop:CTwist}
	For any $A$ in $\Alg^{\G_m}$ and $d\in \Z$, the coaugmentation $\alpha_{A,d}: \Alg^{\G_m}(\Phi(\Z[x(d)]),A) \to C(A(d))$ is limiting in $\Space$.
\end{Prop}

To show this, we want to describe the mapping spaces in $\Alg^{\G_m}$ as a certain end, which will be a limit over a twisted arrow category. Let us recall the terminology and result.

For a 1-category $I$, the twisted arrow category $T(I)$ of $I$ is the 1-category where the objects are arrows $x \to y$ in $I$, and a morphism from $x \to y$ to $x' \to y'$ is a pair of morphisms $x' \to x, y \to y'$ which make the obvious `twisted' square commute. More abstractly, $T(I)$ is the category of elements of the functor $I(-,-):I^\op\times I \to \Set$. 

\begin{Lem}
	\label{Lem:AlgGMmap}
	Let $I$ be a 1-category. For $F,G: I \to \CCC$ two $I$-indexed diagrams in a category $\CCC$, the mapping space $\CCC^I(F,G)$ is the limit of the diagram
	\begin{align*}
	T(I) \to \Space: (x \to y) \mapsto \CCC(F(x),G(y))
	\end{align*}
\end{Lem}

\begin{proof}
	See e.g.\ \cite{GepnerLax}.
\end{proof}

\begin{proof}[Proof of Proposition \ref{Prop:CTwist}]
	Write $C(\Z[x(d)])$ for the cosimplicial diagram associated to $\Phi(\Z[x(d)])$. Observe that
	\begin{align*}
		\Alg^{\G_m}(\Phi(\Z[x(d)]),A) &= \Alg^{\bDelta}_{C(\Z)/}(C(\Z[x(d)]),C(A)) \\ &= \Alg^{\bDelta}(C(\Z[x(d)]),C(A)) \times_{\Alg^{\bDelta}(C(\Z),C(A))} \{ C(\Z) \to C(A) \}
	\end{align*}
	Now we apply Lemma \ref{Lem:AlgGMmap} to the mapping spaces $\Alg^{\bDelta}(-,-)$. Using that limits commute past each other, it holds that 
	\begin{align*}
		\Alg^{\G_m}(\Phi(\Z[x(d)]),A) \simeq \lim_{[n] \to [m] \in T(\bDelta)} H
	\end{align*}
	for the functor $H$ that sends the object $[n] \to [m]$ in  $T(\bDelta)$ to
	\begin{equation}
		\label{Eq:H}
\begin{split}
		 \Alg(C^n(\Z[x(d)]),C^m(A) ) \times_{\Alg(C^n(\Z),C^m(A))} \{C^n(\Z) \to C^m(\Z) \} \\
	  \simeq \Alg(\Z[x(d)],C^m(A) ) \simeq \Alg(\Z[x],A[t_1^{\pm 1},\dots,t_m^{\pm 1}])
\end{split}
	\end{equation}
	We will now describe what $H$ does on morphisms.
	
	Let a morphism $\gamma$ in $T(\bDelta)$ from $\theta_1:[n_1] \to [m_1]$ to $\theta_2:[n_2] \to [m_2]$ be given by maps $\alpha:[n_2] \to [n_1]$ and $\beta:[m_1] \to [m_2]$. For $i=1,2$, let $\gamma_i$ be the map in $T(\bDelta)$ from $\id_{[m_i]}$ to $\theta_i$ determined by $\theta_i$. Also, let $\sigma$ be the map from $\theta_1$ to $\theta_1\alpha$ determined by $\alpha$, and let $\tau$ be the map from $\theta_1\alpha$ to $\theta_2$ determined by $\beta$. This gives a factorization $\tau\sigma = \gamma$. 
	
	Observe that $H(\gamma_i)$ is determined by precomposing a given map $C^{m_i}(\Z[x(d)]) \to C^{m_i}(A)$ with the map $(\theta_i)_*:C^{n_i}(\Z[x(d)]) \to C^{m_i}(\Z[x(d)])$. Likewise, $H(\sigma)$ is given by precomposition with $\alpha_*: C^{n_2}(\Z[x(d)]) \to C^{n_1}(\Z[x(d)])$, and $H(\tau)$ by postcomposition with $\beta_*: C^{m_1}(A) \to C^{m_2}(A)$. 
	
	The map $H(\gamma_i)$, when considered as map $C^{m_i}(A) \to C^{m_i}(A)$ by (\ref{Eq:H}), is given by multiplying with $t_{m_i}^{\theta_i(0)}$,  which is invertible. We can thus define a diagram $H': T(\bDelta) \to \Space$ that sends $\theta:[n] \to [m]$ to $H(\id_{[m]})$, and any map $\gamma: \theta_1 \to \theta_2$ to $H(\gamma_2)^{-1} H(\gamma) H(\gamma_1)$, with notation as before. Then $\lim H \simeq \lim H'$.
	
	Now let $p$ be the functor $T(\bDelta) \to \bDelta$ that sends $[n] \to [m]$ to $[m]$. Consider $H'(\gamma)$ as map $C^{m_1}(A) \to C^{m_2}(A)$. Then it holds
	\begin{align*}
		H'(\gamma) &\simeq H(\gamma_2)^{-1} \circ H(\tau) \circ H(\sigma) \circ H(\gamma_1) \\
		&\simeq  (\times t_{m_2}^{-\theta_2(0)}) \circ C(A)(\beta) \circ (\times t_{m_1}^{\alpha(0)})\circ (\times t_{m_1}^{\theta_1(0)}) \\
		&\simeq C(A)(\beta) \times t_{m_2}^{\alpha(0)+\theta_1(0)-\theta_2(0)} = C(A)(\beta) \times t_{m_2}^{-\beta(0)}=C(A(d))(\beta)
	\end{align*}
	The precomposition of $C(A(d)): \bDelta \to \Space$ with $p$ is thus homotopic to $H$. 
	
	We claim that the functor $p$ is final, meaning that composition with $p$ does not change limits. By \cite[Prop.\ 4.1.3.1]{LurieHTT}, it suffices to show that $p/[m]$ is contractible, for any $m\geq 0$. We write objects of $p/[m]$ as $[n_1] \to [m_1] \to [m]$.  Consider the functor $F:p/[m] \to p/[m]$ that sends $[n_1] \to [m_1] \to [m]$ to $[n_1] \to [m] = [m]$. Then there is a natural transformation $\id_{p/[m]} \to F$, and the essential image of $F$ is equivalent to $\bDelta_{/[m]}$. It thus suffices to show that $\bDelta_{/[m]}$ is contractible, which is clear.
		
	It now follows that $\lim H \simeq \lim C(A(d))$. This equivalence identifies 
	$f \in \lim H \simeq \Alg^{\G_m}(\Phi(\Z[x(d)]), A)$  with the point in $\lim H'$ given by the family $\{\varphi^n\}_n$, where $\varphi^n$ is the composition
	\begin{align*}
		\varphi^n: \Z[x(d)] \subset C^n(\Z[x(d)]) \xrightarrow{f^n} C^n(A)
	\end{align*}
	The equivalence is thus compatible with the coaugmentation, which remained to be shown.
\end{proof}

\subsection{From graded rings to affine $\G_m$-schemes is fully faithful}
\begin{Prop}
	\label{Prop:unitinv}
	The unit $\eta: \id_{\Alg^{\Z}} \to \Psi\Phi$ is invertible.
\end{Prop}

\begin{proof}
	Let $B$ in $\Alg^\Z$ be given, and consider the map $\eta_B: B \to \Psi\Phi(B)$. Since we can check the statement degree-wise, by adjunction it suffices to show that 
	\begin{align*}
	F:\Alg^\Z(\Z[x(d)],B) \to \Alg^{\G_m}(\Phi(\Z[x(d)]),\Phi(B))
	\end{align*}
	is an equivalence, for fixed $d \in \Z$.
	
	Write $C(B)$ for the cosimplicial diagram associated to $\Phi(B)$, let $C(B(d))$ be the $d$-th twist of $\Phi(B)$, and let $G$ be the coaugmentation $\alpha_{\Phi(B),d}$ from Proposition \ref{Prop:CTwist}. We will see that $F$ is a map of limiting coaugmentations, one of which is $G$. In other words, we will give a limiting coaugmentation $H:\Alg^\Z(\Z[x(d)],B) \to C(B(d))$ which commutes with $F$ and $G$. 
	
	Taking a model of $B$, we assume that $B$ is a simplicial graded ring. Consider the obvious coaugmentation $H: B_d \to C(B(d))$, which is just the inclusion $B_d \subset B$. Observe that $GF \simeq H$. Indeed, since $\Z[x(d)]$ is cofibrant in $\Alg^\Z$ and any object in $\Alg^\Z$ is fibrant, the mapping spaces $\Alg^\Z(\Z[x(d)],C^n(B[u^{\pm 1}]))$ can be computed as underived mapping spaces. Now $GF$ sends a point $b\in B_d$ represented by a morphism $\Z[x(d)] \to B$ to the map $\Z[x(d)] \to B[u^{\pm 1}]$ given by $x(d) \mapsto bu^{-d}$. It follows that $GF \cong H$.
	
	We claim that $H$ is limiting. To this end,	consider the Bousfield-Kan spectral sequence for computing the homotopy groups of $\lim C(B(d))$. This spectral sequence has $E^2_{uv} = \pi^u(\pi_v C(B(d)))$ on the second page. Here, for a cosimplicial abelian group $P$ we write $\pi^u(P)$ for the $u$-th cohomology group of the cochain complex $P^0 \to P^1 \to \dots$ with differentials $\partial^n: P^n \to P^{n+1}$ given by $\partial^n = \sum_{i=0}^n (-1)^i d^i$. 
	
	Let $v$ be given. Write $P$ for the cochain complex associated to the cosimplicial abelian group $\pi_v C(B(d))$. We identify $P^n$ with $\pi_v(B)[t_1^{\pm 1},\dots,t_n^{\pm 1}]$. We write the cochain complex which has $\pi_v(B_d)$ concentrated in degree 0 simply as $\pi_v(B_d)$. The projection $\pi_v(B) \to \pi_v(B_d)$ induces a map $P \to \pi_v(B_d)$, and the inclusion $\pi_v(B_d) \to \pi_v(B)$ induces a map $\pi_v(B_d) \to P$. We will show that the composition $z:P \to \pi_v(B_d) \to P$ is homotopic to the identity on $P$. This implies that the spectral sequence degenerates and that $E^2_{0v}  = \pi_v(B_d)$. Together this implies that $\pi_v \lim C(B(d)) \cong \pi_v(B_d)$, which will thus finish the proof.

	We will give maps $\sigma^n: P^n \to P^{n-1}$ of abelian groups for $n\geq 0$, where $P^{-1} \coloneqq 0$, which satisfy $\partial \sigma + \sigma \partial  = \id-z$. We define $\sigma^n$ on an element of the form $b_et_1^{k_1}\cdots t_n^{k_n}$, with $b_e \in \pi_v(B)$ homogeneous of degree $e$, as
	\begin{align*}
		\sigma^n(b_et_1^{k_1}\cdots t_n^{k_n}) = \begin{cases}
			b_et_1^{k_1}\cdots t_{n-1}^{k_{n-1}} & \text{ if } e-d=k_n \\
			0 & \text{ otherwise}
		\end{cases}
	\end{align*} 
	 Since $P^n \cong \bigoplus_e \pi_v(B_e)[t_1^{\pm 1},\dots,t_n^{\pm 1}]$, this gives a unique well-defined map. 
	
	To check that $\partial \sigma + \sigma \partial  = \id-z$, let $p=bt_1^{k_1}\cdots t_n^{k_n} \in P^n$ be given, with $n \geq 1$.  Without loss of generality, assume that $b$ is homogeneous of degree $e$. Observe that
	\begin{align*}
		\partial^n(p) = b(t_1^{k_1}\cdots t_n^{k_n}t_{n+1}^{e-d}-t_1^{k_1}\cdots (t_nt_{n+1})^{k_n} + \dots + (-1)^{n+1} t_2^{k_1} \cdots t_{n+1}^{k_n})
	\end{align*}
	We have a similar description for $\partial^{n-1}$. 
	Now a case distinction on whether $e - d = k_n$ shows that $(\partial^{n-1} \sigma^n + \sigma^{n+1} \partial^n)(p)  = p = (\id - z)(p)$. 
	
	For $b \in P^0$ homogeneous of degree $e$, it holds that 
	\begin{align*}
		(\partial^{-1} \sigma^0 + \sigma^1 \partial^0)(b) = \sigma^1\partial^0(b) = \sigma^1(bt^{e-d} - b)= (\id -z)(b)
	\end{align*}
	which again follows from a simple case distinction on whether $d=e$.
\end{proof}

Proposition \ref{Prop:unitinv} implies that $\Phi$ is fully faithful. From here on, we thus consider $\Alg^\Z$ as a full subcategory of $\Alg^{\G_m}$. In particular, we will omit $\Phi$ from our notation when convenient. Then we can describe $\Psi$ as the restricted Yoneda embedding.   Indeed, for $A \in \Alg^{\G_m}$, we have that $\Psi(A)=\Alg^{\G_m}(-,A)$ as presheaves on $\Poly^\Z$, by adjunction.

\subsection{From affine $\G_m$-schemes to graded rings is conservative} 
Let us briefly consider the classical setting. Let $A$ be a discrete ring with $\G_m$-coaction. Then we have a diagram $A[u^{\pm 1}] \rightrightarrows A[u^{\pm 1},t^{\pm 1}]$, where one map is induced by the diagonal $\G_m$-action on $\Spec A \times \G_m$ and the other map is the inclusion. This is a diagram of graded rings, with $u$ in degree $-1$. The limit of this diagram in the category of discrete, graded rings is the graded ring associated to $A$. 

We have a similar description available in the derived setting, which will imply that $\Psi$ is conservative. For this, we need $\G_m$-coactions on $\Z$-graded rings. To this end, consider $C(\Z)$ as diagram in $\Alg^\Z$ by endowing it with the trivial grading. Identify the category $\Alg^\Z(B\G_m)$ with the category of those cosimplicial diagrams in $\Alg^\Z$ under $C(\Z)$ that, after forgetting the grading, give diagrams corresponding to $\G_m$-coactions. In other words, $\Alg^\Z(B\G_m)$ is the fibre product of the inclusion $\Alg^{\G_m} \subset \Fun(\bDelta,\Alg)_{C(\Z)/}$ with the functor $\Fun(\bDelta,\Alg^\Z)_{C(\Z)/} \to \Fun(\bDelta,\Alg)_{\C(\Z)/}$ induced by the forgetful functor $U:\Alg^\Z \to \Alg$, and is thus a full subcategory of $\Fun(\bDelta,\Alg^\Z)_{C(\Z)/}$.

\begin{Exm}
	For $A \in \Alg^{\G_m}$ the object $C(A[u^{\pm 1}])$, that has $u$ in degree $-1$, lives in $\Alg^\Z(B\G_m)$.
\end{Exm}

\begin{Lem}
	\label{Lem:PsiLim}
	The functor $\Psi$ is equivalent to the composition of the functor $C(-[u^{\pm 1}]): \Alg^{\G_m} \to \Alg^\Z(B\G_m)$ with $\lim: \Alg^\Z(B\G_m) \to \Alg^\Z$.
\end{Lem}

\begin{proof}
	For $A \in \Alg^{\G_m}$, the coaugmentation of $\Z$-graded rings $\alpha_A:\Psi(A) \to C(A[u^{\pm 1}])$ from Definition \ref{Def:alpha} is limiting. Indeed, applying $\Alg^\Z(\Z[x(d)],-)$ to $\alpha_A$ yields the limiting coaugmentation of spaces $\alpha_{A,d}$ from Proposition \ref{Prop:CTwist}. Since all the $\alpha_{A,d}$ are also natural in $A$, the map $\alpha_A$ is natural in $A$, whence the claim follows.
\end{proof}

Recall that for $\pi: B\G_m \to \Spec \Z$ the structure map, the adjunction $\pi^* \dashv \pi_*: \Alg \rightleftarrows \Alg(B\G_m)$ corresponds to the adjunction $\Alg \rightleftarrows \Alg^{\G_m}$, where the left adjoint endows a ring with trivial action and the right adjoint is taking invariants.  

\begin{Lem}
	\label{Lem:piLim}
	The functor $\pi_*:\Alg^\Z(B\G_m) \to \Alg^\Z$ is homotopic to the functor $\lim: \Alg^\Z(B\G_m) \to \Alg^\Z$. Likewise,  $\pi_*:\Alg(B\G_m) \to \Alg$ is homotopic to $\lim: \Alg(B\G_m) \to \Alg$.	
\end{Lem}

\begin{proof}
	Write $\CCC$ for the category $\Fun(\bDelta,\Alg^\Z)_{C(\Z)/}$, and let $F:\Alg^\Z \to \CCC$ be the functor which sends $B \in \Alg^\Z$ to the cosimplicial diagram $B \rightrightarrows B[t^{\pm}] \dots$ of $\Z$-graded $B$-algebras corresponding to the trivial $\G_m$-coaction on $B$. 
	
	The functor $F$ is left adjoint to $\lim: \CCC \to \Alg^\Z$. Indeed, for $B \in \Alg^\Z$ and $C \in \CCC$, by purely formal arguments, it holds that 
	\begin{align*}
		\CCC(FB, C) \simeq \Fun(\bDelta,\Alg^\Z)(\Delta B,C) \simeq \Alg^\Z(B, \lim C)
	\end{align*}
	where $\Delta B$ is the constant diagram.
	
	Now $\pi^*$ endows a $\Z$-graded ring with trivial $\G_m$-coaction. Thus  the composition of $\pi^*$ with the inclusion $\Alg^\Z(B\G_m) \to \CCC$ is $F$. From $\pi^* \dashv \pi_*$ and $F \dashv \lim$ it follows that $\pi_*$ is the restriction of $\lim$ to $\Alg^\Z(B\G_m)$. 
	
	The argument that $\pi_*:\Alg(B\G_m) \to \Alg$ is $\lim: \Alg(B\G_m) \to \Alg$ is the same.
\end{proof}

Write $U:\Alg^\Z \to \Alg$ and $V:\Alg^{\G_m} \to \Alg$ for the forgetful functors.

\begin{Prop}
	The functor $\Psi$ commutes with $U$ and $V$.
\end{Prop}

\begin{proof}
	Consider the following diagram
	\begin{center}
		\begin{tikzcd}
			\Alg^{\G_m} \arrow[drr, swap, "{C(-)[u^{\pm 1}]}"] \arrow[rr, "{C(-)[u^{\pm 1}]}"] && \Alg^{\Z}(B\G_m) \arrow[r, "\pi_*"] \arrow[d, "U"] & \Alg^{\Z} \arrow[d, "U"] \\
			&& \Alg(B\G_m) \arrow[r, "\pi_*"] & \Alg
		\end{tikzcd}
	\end{center}
	It is clear that the triangle on the left commutes. The square on the right commutes by Lemma \ref{Lem:magic}. Now b y Lemma \ref{Lem:PsiLim} and \ref{Lem:piLim}, the composition of the top arrows is $\Psi$, and the composition of the diagonal arrow with the bottom arrow is $V$.
\end{proof}

\begin{Cor}
	The functor $\Psi$ is conservative, and hence $\Phi \dashv \Psi$ is an adjoint equivalence.
\end{Cor}

\begin{proof}
	Since $V$ is conservative, the first statement follows from $U \Psi \simeq V$. Now the second statement is purely formal: since the unit is invertible by Proposition \ref{Prop:unitinv}, it suffices to show that the map $\Psi(\epsilon_A) : \Psi\Phi \Psi(A) \to \Psi(A)$ induced by the counit is invertible, for any $A \in \Alg^{\G_m}$. By adjunction, we have $\id_{\Psi(A)}\simeq \Psi(\epsilon_A) \circ \eta_{\Psi(A)}$. Since $\eta_{\Psi(A)}$ is invertible, the claim follows.
\end{proof}

\subsection{Globalizing the adjoint equivalence $\Phi \dashv \Psi$}
\label{Par:Glob}
Consider now the adjunctions
\begin{center}
	\begin{tikzcd}
	\Alg \arrow[r, shift left, "\mathrm{triv}"] & \Alg^\Z \arrow[l, shift left, "{(-)_0}"] \arrow[r, shift left, "\Phi"] & \Alg^{\G_m} \arrow[l, shift left, "\Psi"]
	\end{tikzcd}
\end{center}
where $\mathrm{triv} \dashv (-)_0$ is the adjunction from Example \ref{Exm:fjadjs} induced by $j:0 \to \Z$. By composing, we recover the adjunction $\pi^* \dashv \pi_*$ from Example \ref{Exm:AlgBGM}. 

\begin{Lem}
	\label{Lem:AffAbs}
	We have equivalences  $\Alg_{R}^{\G_{m,R}} \to \Alg^{\G_m} \times_{\Alg} \Alg_R$, for any ring $R$, where the fibre product is taken with respect to $\pi_*$.
\end{Lem}

\begin{proof}
	Write $C(R)$ for the cosimplicial diagram associated to $R$ endowed with the trivial coaction.  Observe that $\Alg_R^{\G_{m,R}}$ is then the under-category $(\Alg^{\G_{m}})_{C(R)/}$. The equivalence follows from the adjunction $\pi^* \dashv \pi_*$, by \cite[Prop.\ 4.1.1]{RiehlElements}.
\end{proof}

\begin{Thm}
	\label{Thm:Equiv}
	For any scheme $S$, we have an adjoint equivalence 
	\begin{align*}
		\Phi_S\dashv \Psi_S: \Alg^\Z(S) \rightleftarrows \Alg^{\G_{m,S}}(S)
	\end{align*}
\end{Thm}

For $S=\Spec R$ we write $\Phi_S\dashv \Psi_S$ as $\Phi_R\dashv \Psi_R$.

\begin{proof}
	The case $S = \Spec R$ follows from Lemma \ref{Lem:AffAbs} and the case $S = \Spec \Z$. This also shows that the adjunctions $\Phi_R \dashv \Psi_R$ are natural in $R \in \Alg$. Then, for general $S$, we write $\Alg^\Z(S)$ and $\Alg^{\G_{m,S}}(S)$ as a limit of the categories $\Alg^\Z_R$ and $\Alg^{\G_{m,R}}_R$, indexed over all maps $\Spec R \to S$ by Definition \ref{Def:AlgMSt} and Remark \ref{Rem:AlgGMloc}.
\end{proof}

The proof shows that $\Phi_S \dashv \Psi_S$ is natural in $S$. We will henceforth identify the functor $\Alg^\Z(-)$ with $\Alg^{\G_{m}}(-)$. 

\subsection{Graded algebras as $\G_m$-comodules}
We record a result which is interesting in its own right, concerning an application of the Barr--Beck--Lurie theorem as in \cite[\S 4.7]{LurieHA}. Since we will not use the result in what follows, we will not unpack all definitions from \cite{LurieHA} that we use.

Recall that a \textit{monad} on a category $\CCC$ is an algebra object $T$ in the monoidal category $\Fun(\CCC,\CCC)$ under composition. For such $T$, we let $\Mod_T(\CCC)$ be the category of left $T$-modules. The objects of $\Mod_T(\CCC)$ are objects $c \in \CCC$ together with a homotopy coherent action $T(c) \to c$. 

Let  $F: \CCC \rightleftarrows \DDD:G$ be a given adjunction. Then $T \coloneqq GF$ is a monad on $\CCC$, and $G$ factors as $G': \DDD \to \Mod_T(\CCC)$ followed by the forgetful functor $\Mod_T(\CCC) \to \CCC$. Now $F \dashv G$ is called \textit{monadic} if $G'$ is an equivalence. 

The Barr--Beck--Lurie theorem says that $F \dashv G$ is monadic if and only if $G$ is conservative  and geometric realizations of $G$-split simplicial objects exist in $\DDD$ and are preserved by $G$. Here, a \textit{$G$-split} simplicial object $X:\bDelta^\op \to \DDD$ means that $GX$ has a contracting homotopy in $\CCC$. 

Consider the adjunction $U \dashv (-)[t^{\pm 1}] : \Alg^\Z(S) \to \Alg(S)$ from Example \ref{Exm:fjadjs} induced by the map $j: \Z \to 0$, where $U$ is the forgetful functor. This corresponds, under the adjoint equivalence $\Phi_S \dashv \Psi_S$, to the adjunction $\rho^* \dashv \rho_* : \Alg^{\G_{m,S}}(S) \rightleftarrows \Alg(S)$ from Example \ref{Exm:AlgBGM}. Taking opposites on both sides gives us the adjunction
\begin{align*}
	(-) \times \G_m \dashv V: \Aff(S) \rightleftarrows \Aff^{\G_{m,S}}(S)
\end{align*}
where $V$ is the forgetful functor.

\begin{Lem}
	\label{Lem:Ucomonadic}
	The adjunction $(-) \times \G_m \dashv V$ is monadic.
\end{Lem}

\begin{proof}
	Clearly $V$ is conservative, so by the Barr--Beck--Lurie theorem it suffices to show that $V$ preserves $V$-split geometric realizations. So let $Y_\bullie$ be a $V$-split simplicial diagram in $\Aff^{\G_{m,S}}(S)$. 
	
	Let $Y$ be the colimit of $Y_\bullie$ in $\St^{\G_{m,S}}_{/S}$.	Since colimits commute past each other, we have
	\begin{align*}
	[Y/\G_{m,S}] = [\colim Y_\bullie / \G_{m,S}] \simeq \colim [Y_\bullie/\G_{m,S}]
	\end{align*}
	Also, since colimits in $\St_{/S}$ are universal, the map
	\begin{align*}
	\colim ([Y_\bullie/\G_{m,S}] \times_{B\G_{m,S}} S) \to \colim ([Y_\bullie/\G_{m,S}])  \times_{B\G_{m,S}} S 
	\end{align*}
	is an equivalence, where the colimits are taken in $\St_{/S}$. We thus have an equivalence
	\begin{align}
		\label{Eq:colimY}
	g':\colim Y_\bullie \simeq \colim ([Y_\bullie/\G_{m,S}] \times_{B\G_{m,S}} S) \to [Y/\G_{m,S}] \times_{B\G_{m,S}} S \simeq Y
	\end{align}
	where again the colimit is taken in $\St_{/S}$, after forgetting the $\G_{m,S}$-action on $Y_\bullie$ and on $Y$.
	
	Now let $Y_{-1}$ be the colimit of $VY_\bullie$ in $\Aff(S)$. After forgetting the $\G_{m,S}$-action on $Y_\bullie$ and $Y$, we have a morphism $h:Y_{-1} \to Y$ under $Y_\bullie$ in $\St_{/S}$. Now since $Y_\bullie$ is $V$-split, $Y_\bullie \to Y_{-1}$ is colimiting in $\St_{/S}$. Also, $Y_\bullie \to Y$ is colimiting in $\St_{/S}$ by (\ref{Eq:colimY}). It follows that $h$ is an equivalence. Consequently, $Y$ is affine over $S$ and is thus the colimit of $Y_\bullie$ in $\Aff^{\G_{m,S}}(S)$, from which the claim follows.
\end{proof}

\section{Projective spectra}
\label{Sec:Proj}
Throughout, let $\BB$ be an $\N$-graded, quasi-coherent $\OO_S$-algebra on a scheme $S$, with underlying $\N$-graded $\OO_S$-module $\bigoplus_d \BB_d$, such that $\pi_0 \BB$ is generated as a $\pi_0\OO_S$-algebra by $\pi_0\BB_1$, locally on $S$. We refer to this latter condition by saying that $\BB$ is \textit{generated in degree 1}.  

The goal of this section is to define the projective spectrum $\Proj \BB$ and show that it is representable by a scheme. After this, we want to give a description of the functor of points of $\BB$. This will lead us to recover the universal line bundle on $\Proj \BB$, for which we will first need Serre's twisting sheaf, which in turn is based on the exact functor $(-)^\sim$ from graded, quasi-coherent $\BB$-modules to quasi-coherent $\OO_{\Proj \BB}$-modules. 

\begin{Rem}
	The assumption that $\BB$ is generated in degree 1 will give us a nice affine cover of $\Proj \BB$ of the form $\Spec B_{(f)}$, for $f$ homogeneous of degree 1. It is also used in the proof of Lemma \ref{Lem:Twist}, where we show that the twisted modules $\OO_{\Proj \BB}(n)$ are invertible. As in the classical case, the assumption is not as strict as it may seem.
	
	For $\delta\in \N$ and $R \in \Alg$, let $\sigma_\delta$ be the functor $\Poly^\Z_R \to \Poly^\Z_R$ that sends $R[x(d_1),\dots,x(d_k)]$ to $R[x(\delta d_1),\dots,x(\delta d_k)]$. Precomposing with $\sigma_\delta$ gives us a functor 
	\begin{align*}
	\Alg^\Z_R \to \Alg^\Z_R : B \mapsto B^{(\delta)}
	\end{align*}
	where $B^{(\delta)}_d = B_{d \delta}$. Since this is natural in $R$, we get a functor $\Alg^\Z(S) \to \Alg^\Z(S): \BB \mapsto \BB^{(\delta)}$, where again $\BB^{(\delta)}_d = \BB_{\delta d}$
	
	Now suppose that $\BB$ is an $\N$-graded, quasi-coherent $\OO_S$-algebra such that $\pi_0 \BB$ is a finitely generated $\pi_0 \OO_S$-algebra. For simplicity, assume that $S$ is affine. Then there is some $\delta \in \N$ such that $\BB^{(\delta)}$ is generated in degree 1, and we simply define $\Proj \BB$ as $\Proj \BB^{(\delta)}$. To see that this agrees with the classical case,  see e.g.\ \cite[Rem.\ 13.11]{GortzAlgebraic}. This strategy can be extended to quasi-compact $S$ by taking a finite cover of affine subschemes.
\end{Rem}

\subsection{Construction and representability}
Recall from \S \ref{Par:IrrId} that we defined the irrelevant ideal of $\BB$ as $\BB_+ = \bigoplus_{n>0} \BB_n$, which comes with a fibre sequence $\BB_+ \to \BB \to \BB_0$. The closed immersion $\Spec \BB_0 \to \Spec \BB$ is written as $V(\BB_+)   \to \Spec \BB$.

The $\G_{m,S}$-action on $\Spec \BB$ restricts to a ${\G_{m,S}}$-action on $\Spec\BB \smallsetminus V(\BB_+)$. Indeed, the pullback along $V(\BB_+) \to \Spec \BB$ of the action $\Spec \BB \times_S \G_{m,S} \to \Spec \BB$ is $V(\BB_+) \times_S \G_{m,S} \to V(\BB_+)$, since $\BB \to \BB_0$ is a homogeneous map. Therefore $(\Spec \BB \smallsetminus V(\BB_+)) \times_S \G_{m,S} \to \Spec (\BB)$ lands in $\Spec \BB \smallsetminus V(\BB_+)$.

\begin{Def}
	The \textit{relative projective spectrum} $\Proj \BB$ of $\BB$ over $S$ is the  stack $[(\Spec \BB \smallsetminus V(\BB_+)) / \G_{m,S}]$ defined over $S$.
\end{Def}

We are going to show that the relative projective spectrum is a scheme. We want to argue locally, for which we need that the projective spectrum commutes with pullbacks. Note that, for given $g:T\to S$, the pullback $g^*:\Alg^\Z(S) \to \Alg^\Z(T)$ maps $\N$-graded objects to $\N$-graded objects, and is such that $(g^*\BB)_0 \simeq g^*(\BB_0)$. 

\begin{Lem}
	\label{Lem:ProjPB}
	The $\Proj$-construction is stable under pullbacks. That is, if $g:T \to S$ is given, then $\Proj g^*\BB \simeq \Proj \BB \times_S T$.
\end{Lem}

\begin{proof}
	We have that $
	V((g^* \BB)_+) \simeq \Spec  g^*\BB_0  \simeq V(\BB_+) \times_S T
	$, since the relative spectrum commutes with pullbacks.	From the pasting lemma for pullbacks, it follows that 
	$
	V(g^*\BB_+) \simeq \Spec  g^* \BB  \times_{\Spec \BB} V(\BB_+)
	$.	We then get
	\begin{align*}
	\Proj g^* \BB &= [ (\Spec g^*\BB \smallsetminus V(g^* \BB_+)) / \G_{m,T} ] \\ &\simeq [ g^* (\Spec \BB \smallsetminus V(\BB_+) ) / g^*\G_{m,S} ] \simeq \Proj \BB \times_S T
	\end{align*}
	For the last equivalence we have used that $g^*:\St_{/S} \to \St_{/T}$ commutes with quotient stacks. Indeed, $g^*$ commutes with all colimits since colimits in topoi are universal.
\end{proof}

\begin{Thm}
	\label{Thm:ProjRep}
	$\Proj \BB $ is representable by a scheme over $S$.
\end{Thm}

\begin{proof}
	Since being schematic is local, by Lemma \ref{Lem:ProjPB}  we may  assume that $\Spec \BB \to S$ corresponds to a map $R \to B$ in $\Alg$ with $B$ a graded $R$-algebra.  
	
	Note that $\{\Spec B_f \mid f \in B_1\}$  is an open cover of $\Spec  B  \smallsetminus V(B_+)$ by equivariant maps, and thus that $\{[\Spec B_{f}/\G_{m,S}] \mid f \in B_1 \}$ is an open cover of $\Proj B$, as holds in general for such a situation. It thus suffices to show that $[\Spec B_f / \G_{m,S}] \simeq \Spec B_{(f)}$. And indeed this holds, by Lemma \ref{Lem:Deg0loc} and the fact that the diagonal $\G_{m,S}$-action on $\Spec B_f \cong \Spec B_{(f)} \times_S \G_{m,S}$ is given by only acting on $\G_{m,S}$.
\end{proof}

\begin{Not}
	From here on, write $\Proj \BB$ and its structure map as 
	\begin{align*}
	Y \coloneqq \Proj \BB \xrightarrow{\tau} S
	\end{align*}
	Affine-locally on $S$, we write $D_+(f)$ for the affine open subsets of the form $\Spec B_{(f)} \subset Y$, for $f$ a homogeneous point in $\BB$ of degree 1.
\end{Not}

\subsection{Functoriality}
\label{Par:ProjFunct}
We will show that the derived $\Proj$ construction has familiar functorial behaviour. First we record:

\begin{Lem}\label{Lem:ProjClas}
	The underlying classical scheme of $\Proj \BB$ is the classical projective spectrum of $\pi_0 \BB$ over $S_\cl$. When $S, \BB$ are classical, then so is $\Proj \BB$.
\end{Lem}

\begin{proof}
	Using the gluing description of $\Proj$, it suffices to observe that for a graded $A$-algebra $B$, $A \in \Alg$, and $f \in B_1$, we have $\pi_0(B_{(f)}) \simeq \pi_0(B)_{(f)}$, where we have written the connected component containing $f$ just as $f \in \pi_0 B$. The second claim also follows from the local picture.
\end{proof}

%

Now let $\varphi: \BB \to \BB'$ be a morphism of graded $\OO_S$-algebras. Let $U(\varphi)$ be the open complement of the fibre product 
\begin{align*}
	F \coloneqq \Spec \BB' \times_{\Spec \BB} V(\BB_+)
\end{align*}
Since $V(\BB'_+) \to \Spec \BB'$ factors through $F$, it holds that $U(\varphi)$ is contained in $\Spec \BB' \setminus V(\BB'_+)$. The induced map $U(\varphi) \to \Spec \BB \setminus V(\BB_+)$ is moreover equivariant, so we get a morphism
\begin{align*}
	\Proj(\varphi): [U(\varphi)/\G_m] \to [ (\Spec \BB \setminus V(\BB_+)) / \G_m ]
\end{align*}
If $\varphi$ is a surjective map, then $U(\varphi)$ is all of $\Spec \BB' \setminus V(\BB'_+)$. If $\BB, \BB'$ are generated in degree 1, we thus in this case get a morphism
\begin{align*}
	\Proj(\varphi): \Proj \BB' \to \Proj \BB
\end{align*}

\subsection{Graded $\BB$-modules}
\label{Par:GradBMod}
By Remark \ref{Rem:Einfty}, we can consider $\BB$ as an algebra object in the symmetric monoidal category $\cn\QCoh^\N   (S)$, and therefore as an algebra object in $\QCoh^\Z(S)$. 
\begin{Def}
	The category $\GrMod_{\BB}$ of \textit{graded, quasi-coherent $\BB$-modules} is the category of modules in $\QCoh^\Z(S)$ over $\BB$. Such a module is \textit{connective} if the underlying $\Z$-graded $\OO_S$-module is connective. We write $\cn\GrMod_{\BB}$ for the category of connective objects in $\GrMod_{\BB}$.
\end{Def}

\begin{Rem}
	Informally, a graded, quasi-coherent $\BB$-module is thus a quasi-coherent module $\MM$ over the underlying algebra of $\BB$, together with a decomposition $\MM \simeq \bigoplus_{d \in \Z} \MM_d$ in $\QCoh(S)$, such that the maps $\BB_d \otimes_{\OO_S} \MM_e  \to \MM$ factor through $\MM_{d+e}$ for all $d,e$, in a homotopically coherent way. 
\end{Rem}

The category $\GrMod_{\BB}$ is stable by \cite[Prop.\ 7.1.1.4]{LurieHA}. Furthermore, it is symmetric monoidal by \cite[Thm.\ 4.5.2.1]{LurieHA}, where the monoidal structure is given by the relative tensor product $(-) \otimes_\BB (-)$, which preserves colimits in each variable separately by \cite[Cor.\ 4.4.2.15]{LurieHA}.

We will sketch the construction of a functor
\begin{align*}
	\cn\GrMod_{\BB} \to \cn\QCoh(Y): \MM \mapsto \MM^\sim
\end{align*}
We first give a local description of $\MM^\sim$ for given $\MM \in \GrMod_{\BB}$. Take an open immersion $\Spec R \to X$, and let $B \in \Alg^\Z_R$ correspond to the pullback of $\BB$ to $\Spec R$. Then $\Proj B \simeq \Proj \BB \times_X \Spec R$ has an affine open cover by schemes of the form $\Spec B_{(f)}$, where $f \in B_1$. On such $\Spec B_{(f)}$, the module $\MM^\sim$ is given by $M_{(f)}$, where $M$ is the pullback of $\MM$ to $\Spec R$, and $M_{(f)}$ is the degree-zero part of $M_f \coloneqq M \otimes_B B_{f}$. 

In the classical case, it is not hard to see that this local description glues into a quasi-coherent $\OO_Y$-module $\MM^\sim$, functorial in $\MM$. One can use the Grothendieck construction in the derived setting. Recall that this entails an equivalence, for any category $\CCC$, between $\Fun(\CCC,\Cat)$ and the category of coCartesian fibrations over $\CCC$. See e.g.\ \cite{BarwickFibrations} for an exposition. For a given functor $F:\CCC \to \Cat$, we write the associated coCartesian fibration as $\int F \to \CCC$.

Let $\AAA$ be the category $\int \Alg^\N_{(-)}$ associated to the functor $\Alg \to \Cat$ that sends $R$ to $\Alg^\N_R$. Then objects of $\AAA$ are pairs $(R,B)$ with $R \in \Alg$ and $B \in \Alg^\N_R$. A morphism $(R,B) \to (R',B')$ consists of a ring map $R \to R'$ and a map of $\N$-graded $R'$-algebras $B \otimes_R R' \to B'$. 

Likewise, let $\XXX$ be the category $\int \cn\GrMod_{(-)}$. Observe that $\XXX$ has as objects triples $(R,B,M)$, where $(R,B) \in \AAA$ and $M \in \cn\GrMod_B$. A morphism $(R,B,M) \to (R',B',M')$ is a map $(R,B) \to (R',B')$ in $\AAA$ together with a morphism $M \otimes_B B' \to M'$ in $\GrMod_{B'}$. 

Now let $\AAA_0$ be the category $\int(\Alg^\N_{(-)})^\simeq$ associated to the functor $\Alg \to \Space$ that sends $R \in \Alg$ to the core $(\Alg^\N_{R})^\simeq$ of $\Alg^\N_{R}$ (obtained by discarding non-invertible morphisms). This is the subcategory of $\AAA$ spanned by all objects, and by those morphisms $(R,B) \to (R',B')$ for which the induced map $B \otimes_R R' \to B'$ is an equivalence. Consider the functor $\AAA_0 \to \Cat$ that sends $(R,B)$ to $\cn\QCoh(\Proj B)$, where we temporarily write $\Proj B$ for $[(\Spec B \setminus V(B_+)) / \G_m]$, even though this might not be a scheme. Observe that this is well-defined by \S \ref{Par:ProjFunct}, since for $(R,B) \to (R',B')$ in $\AAA_0$ we have a morphism $f:\Proj B' \to \Proj B$, hence $f^*: \cn\QCoh(\Proj B) \to \cn\QCoh(\Proj B')$. Write then $\YYY_0$ for $\int \cn\QCoh(\Proj (-))$. 

There is a functor $F:\XXX_0 \to \YYY_0$ that sends $(R,B,M)$ to $(R,B,M^\sim)$, where the local description of $M^\sim$ is as above. This can be formalized by reducing to the case where $(R,B,M)$ is of the form $R = \Z[t_1,\dots,t_k]$, $B= R[x(d_1),\dots,x(d_n)]$, and $M$  a finite direct sum of shifted copies of $B$, i.e.\ graded $B$-modules $B(e)$ for which $B(e)_d = B_{d+e}$. This  strategy is similar to the one used in \cite[\S 25.2.1]{LurieSpectral}. We omit details.

The functor $F$ sends coCartesian edges to coCartesian edges. We thus get functors $F_{(R,B)}:\GrMod_{B} \to \QCoh(\Proj B)$, natural with respect to pairs $(R,B) \in \AAA$. To globalize this to schemes, let $\AAA'_0$ be the category of pairs $(S,\BBB)$, where $S$ is a scheme and $\BBB \in \Alg^\N(S)$. Morphisms are defined as in $\AAA_0$. Now the functor $\cn\GrMod_{(-)}$  on $\AAA'_0$ is the right Kan extensions of its restriction to $\AAA_0$; the same holds for $\cn\QCoh(\Proj(-))$. The functors $F_{(R,B)}$ thus induce a functor $\cn\GrMod_\BBB \to \cn\QCoh(\Proj(\BBB))$, written $\MM \mapsto \MM^\sim$, natural in $(S,\BBB) \in \AAA'_0$, which is locally given by $F_{(R,B)}$.

\begin{Rem}
	Evaluating at affine opens of the form $\Spec B_{(f)}$ shows $\OO_Y \simeq {\BB}^\sim$.
\end{Rem}

\begin{Prop}
	The functor $\cn\GrMod_{\BB} \to \cn\QCoh(Y): \MM \mapsto \MM^\sim$ is exact.
\end{Prop}	

\begin{proof}
	Since the question is local, it suffices to show that
	\begin{align*}
		\cn\GrMod_{B} \to \cn\Mod_{B_{(f)}} : M \mapsto M_{(f)}
	\end{align*}
	is exact, for any graded $R$-algebra $B$ and $f \in B_1$. Since the functor $\cn\GrMod_{B} \to \cn\GrMod_{B_f}: M \mapsto M_f$ is exact by \cite[Cor.\ 4.4.2.15]{LurieHA}, we only need to show that $\cn\GrMod_{B_f} \to \cn\Mod_{B_{(f)}}: M \mapsto M_0 $ is exact. For this, in turn, it suffices to show that the functor
	\begin{align*}
		H:\GrMod_{B_f} \to \Mod_{B_{(f)}}: M \mapsto M_0 
	\end{align*}
	is exact and t-exact.
	
	We have an adjunction $\mathrm{triv} \dashv (-)_0:\Mod_{B_{(f)}}  \rightleftarrows \GrMod_{B_{(f)}}$, where $\mathrm{triv}$ endows a $B_{(f)}$-module with the trivial grading. The functor 
	\begin{align*}
		\Mod_{B_{(f)}} \to \GrMod_{B_f}: N \mapsto \mathrm{triv}(N) \otimes_{B_{(f)}} B_f
	\end{align*}
	is then a left adjoint to $H$. Since $\GrMod_{B_f}$ and $\Mod_{B_{(f)}}$ are both stable, it follows that $H$ is exact. Now t-exactness follows from the fact that $\pi_n(M_d) \cong \pi_n(M)_d$. 
\end{proof}

\subsection{Serre's twisting sheaf}
\label{Sec:SerreTwist}
For $n\in \Z$, let $\BB(n)$ be the graded $\BB$-module such that $\BB(n)_d = \BB_{n+d}$.  Then put $\OO_{Y}(n) \coloneqq {\BB(n)}^\sim$, which is the derived version of Serre's twisting sheaf. 

\begin{Rem}
	Let $B$ be a graded ring. With similar machinery as in \S \ref{Subsec:AffGMtoGR}, we could endow $B(n)$ with a $\G_m$-coaction, which gives us a cosimplicial module that has $C(B(n))$  as underlying cosimplicial space. This justifies the terminology from Definition \ref{Def:dtwist}, but is not needed in what follows.
\end{Rem}

Recall, from \S \ref{Par:GradModSt}, that $\Z^\otimes$ is the symmetric monoidal category associated to the monoid $(\Z,+)$. 

\begin{Lem}
	\label{Lem:Twist}
	The modules $\OO_{Y}(n) \in \QCoh (Y)$ are invertible, and induce a lax symmetric monoidal functor $\OO_Y(-):\Z^\otimes \to \QCoh(Y)$ for which the maps $\OO_Y(n) \otimes \OO_Y(m) \to \OO_Y(n+m)$ are invertible.
\end{Lem}

\begin{proof}
	The first claim is local, so we reduce to a graded $R$-algebra $B$. We claim that the $D_+(f)$ for points $f \in B_1$ form a trivializing open cover for $\OO_{Y}(n)$. To see this, consider the $B$-module homomorphism $\times f^n: B \to B(n)$, which gives an equivalence $B_f \simeq B(n)_f$. Moreover, this equivalence preserves the grading, and thus gives an equivalence $B_{(f)} \simeq B(n)_{(f)}$. It follows that $\OO_Y(n)$ is a line bundle.
	
	For the existence of the functor $\OO_Y(-)$, as in \S \ref{Par:GradConEinft}, it suffices to give a $\Z$-graded $\OO_Y$-algebra of the form $\bigoplus_{n \in \Z} \OO_Y(n)$. We sketch the argument for why such an algebra exists.
	
	Recall the category $\AAA_0$ of pairs $(R,B)$ from \ref{Par:GradBMod}. Let $\ZZZ$ be the category of triples $(R,B,\CC)$, where $(R,B) \in \AAA_0$ and $\CC \in \Alg^\Z(\Proj \BB)$. As in the construction of $(-)^\sim$, we can define a functor
	\begin{align*}
		\AAA_0 \to \ZZZ: (R,B) \mapsto \left( R,B, \bigoplus\nolimits_{n\in \Z} \OO_{\Proj B}(n) \right)
	\end{align*}
	by defining this functor in the case where the object $(R,B) \in \AAA_0$ is of the form $(\Z[t_1,\dots,t_k],R[x(d_1),\dots,x(d_n)])$, which is straightforward. By right Kan extending this functor to schemes, we indeed get a $\Z$-graded $\OO_Y$-algebra $\bigoplus_{n \in \Z} \OO_Y(n)$.
	
	For the third claim we observe that the dual of $\OO_Y(n)$ is $\OO_Y(n)^\vee = \OO_Y(-n)$. Upon tensoring the map $h:\OO_Y(n) \otimes \OO_Y(m) \to \OO_Y(n+m)$ with $\OO_Y(-n)$ we obtain the identity on $\OO_Y(m)$. It follows that $h$ is invertible.
\end{proof}

\subsection{Functor of points}
\label{Par:FOP}
In the classical case, we know that $\Proj \BB$ represents the functor that assign to a (classical) scheme $f:T \to S$ over $S$ the set of equivalence classes of pairs $(\LL, \psi)$, where $\LL$ is a line bundle on $T$ and $\psi$ is a map $f^* \BB \to \bigoplus_{n \geq 0} \LL^{\otimes n}$ of graded $\OO_T$-algebras such that $f^*\BB_1 \to \LL$ is surjective. Here, two such pairs $(\LL,\psi)$ and $(\LL',\psi')$ are equivalent if there is an isomorphism of $\Z$-graded $f^*\BB$-algebras $\bigoplus_{n \geq 0} \LL^{\otimes n} \cong \bigoplus_{n \geq 0} (\LL')^{\otimes n}$. See e.g.\ \cite[\href{https://stacks.math.columbia.edu/tag/01O4}{Tag 01O4}]{stacks-project}. 

The goal of this section is to give a similar description in the derived setting. We thus want to make sense of graded algebras which are of the form $\bigoplus\nolimits_{n\in \N} \LL^{\otimes n}$, for some line bundle $\LL$. We do this as follows.
\begin{Def}
	Let $\XXX$ be a stack and $\AA$ an $\N$-graded, quasi-coherent $\OO_{\XXX}$-algebra. We say that $\AA$ is \textit{freely generated by a line bundle} if the canonical map $\OO_{\XXX} \to \AA_0$ is invertible, $\AA_1$ is a line bundle on $\XXX$, and the canonical maps $\AA_1^{\otimes n} \to \AA_n$ are equivalences for all $n \in \N$.
\end{Def}

\begin{Rem}
	\label{Rem:SymLine}
	Let $\AA$ be an $\N$-graded, quasi-coherent $\OO_T$-algebra on a scheme $T$. We argue that $\AA$ is freely generated by a line bundle $\LL$ if and only if $\AA$ is equivalent to $\Sym \LL$.
	
	Let first $\AA$ be freely generated by a line bundle $\LL \simeq \AA_1$. Then the inclusion $\LL \to \AA$ induces an $\N$-graded map $\Sym \LL \to \AA$ by Lemma \ref{Lem:SymZ}. We claim that this is an equivalence. Since we can check this locally, we may assume that $\LL$ is free and that $X$ is affine, say $X = \Spec A, \LL = A$. Then the map $\Sym(\LL)_n \to \AA_n$ is equivalent to the map $(At)^{\otimes n} \to A^{\otimes n}$, which is an equivalence.
	
	The converse, that $\Sym(\LL)$ is freely generated by $\LL$, is again straightforward to check locally.
\end{Rem}

\begin{Lem}
	\label{Lem:GBLplb}
	Let $\psi: \XXX' \to \XXX$ be a morphism of stacks and $\AA$ an $\N$-graded $\OO_{\XXX}$-algebra. If $\AA$ is freely generated by a line bundle, then so is $\psi^*\AA$. 
\end{Lem}

\begin{proof}
	Use that structure sheaves, line bundles and tensor products pull back, together with commutativity of the diagrams
	\begin{center}
		\begin{tikzcd}
			\Alg^\N(\XXX) \arrow[d, "(-)_n"] \arrow[r, "\psi^*"] & \Alg^\N(\XXX') \arrow[d, "(-)_n"] \\
			\cn\QCoh  (\XXX) \arrow[r, "\psi^*"] & \cn\QCoh   (\XXX')
		\end{tikzcd}
	\end{center}
	with vertical arrows induced by the forgetful functors $\Alg^\N(-) \to \cn\QCoh^\N(-)$ and the functors $(-)_n:\cn\QCoh^\N(-) \to \cn\QCoh(-)$.
\end{proof}

\begin{Exm}
	\label{Exm:UnivGmTors}
	Consider the projection map $\rho: S \to B\G_{m,S}$. Since $\rho$ is affine, it holds $S = \Spec \rho_* \OO_S$. In fact, $\rho_* \OO_S$ is a $\Z$-graded $\OO_{B\G_{m,S}}$-algebra such that $(\rho_*\OO_S)_{\geq 0}$ is freely generated by a line bundle. Indeed, since $S \to B\G_{m,S}$ is the pullback of $\nu:\Spec \Z \to B\G_{m}$, by Lemma \ref{Lem:GBLplb} it suffices to check the statement on $\nu$. In this case it follows from the fact that $\nu$ is the universal $\G_m$-torsor in the classical sense, since a $\G_m$-torsor on a classical scheme $U$ is locally of the form $U \times \G_m$.
\end{Exm}

Let $\YYY: \Sch_{/S}^\op \to \Space$ be the following functor. For $\psi: T\to S$ we let $\YYY(\psi)$ be the space of maps $\psi^* \BB \to \MM$ of $\N$-graded $\OO_T$-algebras, such that $\MM$ is freely generated by a line bundle, and such that $ \psi^*\BB_1 \to \MM_1$ is surjective. Thus, $\YYY$ is a sub-presheaf of 
\begin{align*}
	\Sch_{/S}^\op &\to \Space \\ (\psi:T \to S) &\mapsto (\Alg^\N(T){_{\psi^*\BB /}})^{\simeq}
\end{align*}
where $(-)^\simeq$ means that we discard non-invertible morphisms.

\begin{Prop}
	The functor $\YYY$ is the functor of points of $Y=\Proj(\BB)$.
\end{Prop}

\begin{proof}
	Let $\psi: T \to S$ be given. We know that $Y(T)$ is the space of commutative squares
	\begin{center}
		\begin{tikzcd}
			P \arrow[d] \arrow[r] & \Spec \BB \smallsetminus V(\BB_+) \arrow[d] \\
			T \arrow[r, "\psi"] & S
		\end{tikzcd}	
	\end{center}
	where $P \to T$ is a principal $\G_{m,S}$-bundle and $P \to \Spec \BB \smallsetminus V(\BB_+)$ is $\G_{m,S}$-equivariant. 
	
	Recall that $P \to T$ is a pullback of $S \to B\G_{m,S}$. Therefore, by Lemma \ref{Lem:GBLplb} and Example \ref{Exm:UnivGmTors}, the bundle $P \to T$ is classified by a $\Z$-graded $\OO_T$-algebra $\AA$ such that $P \simeq \Spec \AA$, and such that $\AA_{\geq 0}$ is freely generated by a line bundle.
	
	The  composition $\varphi:P \to \Spec \BB \smallsetminus V(\BB_+) \to \Spec \BB$ is equivariant, and thus corresponds to a graded map $\psi^*\BB \to \AA$ which, by Example \ref{Exm:Ageq0}, factors through $\AA_{\geq 0}$ since $\BB$ is $\N$-graded. Now the claim is that $\varphi$ landing outside $V(\BB_+)$ is exactly the condition that $\psi^*\BB \to \AA_{\geq 0}$ induces a surjection $\psi^*\BB_1 \to \AA_1$.  
	
	To see this, consider the following diagram of pullback squares in $\Sch$
	\begin{center}
		\begin{tikzcd}
			Q \arrow[dr, phantom, "\scalebox{1}{$\lrcorner$}" , very near start, color=black] \arrow[r] \arrow[d] & \Spec \psi^* \BB_0 \arrow[dr, phantom, "\scalebox{1}{$\lrcorner$}" , very near start, color=black] \arrow[d] \arrow[r] & \Spec \BB_0 \arrow[d]\\
			\Spec \AA \arrow[r] & \Spec \psi^* \BB  \arrow[r] & \Spec \BB 
		\end{tikzcd}	
	\end{center}
	Observe that $Q \simeq \Spec (\AA \otimes_{\psi^*\BB} \psi^* \BB_0)$ over $T$. Since $V(\BB_+) \simeq \Spec \BB_0$, we thus have to show that $Q = \varnothing$ if and only if $\psi^*\BB_1 \to \AA_1$ is surjective. But $Q = \varnothing$ if and only if $Q_\cl = \varnothing$, so the statement follows from the classical case.	
\end{proof}

\begin{Rem}
	The identity on $Y$ is classified by the principal $\G_{m,S}$-bundle $P = \Spec \BB \setminus V(\BB_+)$ over $Y$, which gives us that $P = \Spec \AA$ for a $\Z$-graded $\OO_Y$-algebra $\AA$ by Example \ref{Exm:UnivGmTors}. The functor of point description of $Y$ gives us a map $\theta:\tau^*\BB \to \AA_{\geq 0}$, which is universal in the sense that  $\varphi:T \to Y$ corresponds to the $T$-point of $\YYY$ given by pulling back $\theta$ along $\varphi$.
	
	We can now also make sense of $\OO_Y(1)$ as the universal line bundle. In fact, the underlying $\Z$-graded $\OO_Y$-module of $\AA$ is $\bigoplus\nolimits_{d \in \Z} \OO_Y(d)$. To see this, observe that on an affine patch of the form $\Spec B_{(f)} \subset Y$ we have
	\begin{align*}
		\OO_Y(d)(D_+(f)) = (B_f(d))_0 = (B_f)_d = \AA_d(D_+(f))
	\end{align*}
	where the last equality follows from $P = \Spec \BB \setminus V(\BB_+)$ and the proof of Theorem \ref{Thm:ProjRep}.	
\end{Rem}

\begin{Exm}
	Let us compare our functor of points description of $\Proj \BB$ with the classical one. So suppose that $S, \BB$ are classical, and let $f: T\to S$ be a scheme over $S$, also classical. Write $\YYY'$ for the functor described in \cite[\href{https://stacks.math.columbia.edu/tag/01O4}{Lem.\ 01O4}]{stacks-project}. We will see that $\YYY$ and $\YYY'$ are equivalent.
	
	As we shall see below in Lemma \ref{Lem:ProjClas}, it holds that $\Proj \BB$ is classical. Thus $\YYY(T)$ is 0-truncated. It thus suffices to show $\pi_0\YYY(T) \cong \YYY'(T)$.
	
	Any line bundle on $T$ is classical, since $T$ is. Hence $\Sym \LL$ is classical, for any line bundle $\LL$ on $T$, by \cite[Cor.\ 25.2.3.2]{LurieSpectral}. Thus a morphism $f^*\BB \to \MM$ in $\Alg^\N(T)$ where $\MM $ is freely generated in degree 1 by a line bundle $\LL$ is equivalent to a morphism $\pi_0f^*\BB \to \bigoplus_{d \geq 0} \LL^{\otimes d}$ of discrete $\N$-graded $\OO_T$-algebras. Now two such $f^*\BB \to \MM$ and $f^*\BB \to \MM'$ are equal in $\pi_0\YYY(T)$ if and only if there is a graded $\OO_T$-algebra isomorphism $\MM \to \MM'$ under $f^*\BB$. This gives $\pi_0\YYY(T) \cong \YYY'(T)$, as we wanted. 
\end{Exm}

\begin{Exm}
	We return to the derived scheme $S$. Let $\EE \in \QCoh(S)$ be locally free of finite rank. Recall that the projective bundle of $\EE$ over $S$ is defined as the scheme $\P_S(\EEE)$ with the following functor of points. For $f: T\to S$, it holds that $\P_S(\EEE)(T)$ is the space of pairs $(\LL,\psi)$, with $\LL \in \Pic(T)$ and $\psi: f^*\EE \to \LL$ a surjective map of $\OO_T$-modules. Such a map is equivalent to a graded map $f^* \Sym \EE \to \Sym \LL$ which is surjective in degree 1, and hence $\P_S(\EEE) \simeq \Proj (\Sym \EEE)$.
\end{Exm}

\section{The Rees algebra of a closed immersion}
\label{Sec:Rees}
The goal of this section is to describe the (extended) Rees algebra of a closed immersion, which will be our input for defining blow-ups in the next section. 	

The strategy will be as follows. For any given map $Z \to X$, we will define a stack $D_{Z/X}$ over $\A^1_X$ with a $\G_{m,X}$ action which satisfies a universal property closely related to the one  described for the classical case in \S \ref{Par:ClasRees}. Then we will show that for $Z \to X$ a closed immersion, $D_{Z/X}$ is representable by a scheme affine over $\A^1_X$. 

The stack $D_{Z/X}$ will be defined using Weil restrictions, which we will first review. The advantage of this approach is that $D_{(-)/X}$ is functorial by construction, which will allow us to work locally when checking representability in the proof of Theorem \ref{Thm:ReesRep}.

\subsection{Weil restrictions}
  We recall the relevant theory on Weil restrictions.

Let $Y \in \St$ and $X,Z \in \St_{/Y}$. If it exists, we write $\Map_{/Y}(X,Z) $ for the \textit{internal mapping object}, i.e.,\ an object in $\St_{/Y}$ together with an evaluation map $e: \Map_{/Y}(X,Z)\times_Y X \to Z$ such that the induced map
\begin{align*}
\St_{/Y}(T,\Map_{/Y}(X,Z)) \to \St_{/Y}(T \times_Y X, Z)
\end{align*}
is an equivalence, for all $T \in \St_{/Y}$.

For us, a morphism of stacks $f:X \to Y$ is \textit{affine} if for all $\Spec A \to Y$ the fibre product $X_A$ is representable by an affine scheme. If $X,Y$ are algebraic, this coincides with $f$ being affine in the sense of Remark \ref{Rem:RelSpec}. In general, $f$ being affine implies that it is representable in the sense of \cite[Con.\ \S 19.1.2]{LurieSpectral}.\footnote{In fact, the theory of Weil restriction is developed in \cite[Con.\ \S 19.1.2]{LurieSpectral} for representable morphisms, but we do not need this generality.}

\begin{Exm}
	The map $f:B\G_m \to [\A^1/\G_m]$ induced by the zero section is affine. Indeed, let $g:\Spec A \to [\A^1/\G_m]$ be given. Pulling back $g$ along $\A^1 \to [\A^1/\G_m]$ gives an affine scheme $\Spec B$, where $B$ is a graded algebra endowed with a graded map $\Z[t] \to B$. Let $b$ be the image of $t$ in $B$. Then the pullback of $g$ along $f$ is $\Spec ((B/(b))_0)$. 
\end{Exm}	

\begin{Prop}[{\cite[Prop.\ 19.1.2.2]{LurieSpectral}}]
	Let morphisms $X\to Y \leftarrow Z$ in $\St$ be given, with $X \to Y$ affine. Then $\Map_{/Y}(X,Z)$ exists as a stack.
\end{Prop}

The proof of this proposition involves two things. First, it is shown that the assignment $T \mapsto \St_{/Y}(T \times_Y X,Z)$ is in fact valued in (small) spaces. Second, it is shown that  this is a sheaf.

\begin{Prop}[{\cite[Con.\ 19.1.2.3]{LurieSpectral}}]
	Let $f:X \to Y$ be an affine morphism of stacks. Then the pullback functor $f^*:\St_{/Y} \to \St_{/X}$ has a right adjoint, given by
	\begin{align*}
	\Res_{X/Y}:\St_{/X} &\to \St_{/Y} \\
	Z & \mapsto \sMap_{/Y}(X,Z) \times_{\sMap_{/Y}(X,X)} Y
	\end{align*}
\end{Prop}
We call $\Res_{X/Y}$ the \textit{Weil restriction} along $f$. When convenient, we write $\Res_f$ for $\Res_{X/Y}$.
\begin{Rem}
	\label{Rem:ResFP}
	Let $Z\to X \to Y$ be morphisms of schemes. Observe that for any scheme $f:T \to Y$, we indeed have that
	\begin{align*}
	\Res_{X/Y}(Z)(T) &= \sMap_{/Y}(X,Z)(T) \times_{\sMap_{/Y}(X,X)(T)} Y(T)\\
	&= \St_{/Y}(T \times_Y X, Z) \times_{ \St_{/Y}(T \times_Y X, X) } \{f\}  \\
	&\simeq \St_{/X}(T \times_YX, Z)
	\end{align*}
	In particular, we always have a map $\Res_{X/Y}(Z) \times_Y X \to Z$ over $X$ induced by the identity on $\Res_{X/Y}(Z)$.
\end{Rem}

The following shows that Weil restrictions commute with pulling back.

\begin{Lem}
	\label{Lem:WeilPull}
	Let $f:X \to Y$ and $f':X' \to Y'$ be  affine morphisms of stacks. Suppose that we are given a commutative square in $\St$
	\begin{equation}
		\label{Eq:WeilPullCart}
		\begin{tikzcd}
			Z' \arrow[d] \arrow[r] & X' \arrow[d] \arrow[r, "f'"] & Y' \arrow[d, "g"] \\
			Z \arrow[r] & X \arrow[r, "f"] & Y
		\end{tikzcd}	
	\end{equation}	
	 with the square on the right Cartesian. Then the universal property of $\Res_{X/Y}(Z)$ induces a map $\Res_{X'/Y'}(Z') \to \Res_{X/Y}(Z)$ over $Y$. If both squares are Cartesian, then this map exhibits $\Res_{X'/Y'}(Z')$ as the pullback of $\Res_{X/Y}(Z)$ along $g$.
\end{Lem}

\begin{proof}
	Observe that
	\begin{align*}
	\Res_{X/Y}(Z) ( \Res_{X'/Y'}(Z') ) \simeq \St_{/X}(\Res_{X'/Y'}(Z')\times_{Y'} X',Z)
	\end{align*}
	Hence, composing the map $\Res_{X'/Y'}(Z')\times_{Y'} X' \to Z'$ with $Z' \to Z$ induces a map $\Res_{X'/Y'}(Z') \to  \Res_{X/Y}(Z)$ over $Y$, and thus a map
	\begin{align*}
	h: \Res_{X'/Y'}(Z') \to \Res_{X/Y}(Z) \times_Y Y'
	\end{align*}
	Now if both squares in (\ref{Eq:WeilPullCart}) are Cartesian, then $h$ is an equivalence. Indeed, for given $T' \in \St_{/Y'}$, it holds that $h(T')$ is the map
	\begin{align*}
	\St_{/X'}(T'\times_{Y'}X',Z') \to \St(T',\Res_{X/Y}(Z)) \times_{Y(T')} \{T' \to Y\} \simeq \St_{/X}(T'\times_{Y} X,Z)
	\end{align*}
	which is an equivalence.
\end{proof}

\subsection{The functor $G_A$}
We first describe a functor $G_A$, which is used in the construction of the Rees algebra. It will turn out that this functor $G_A$ is right adjoint to taking the Rees algebra when restricted to the appropriate categories.
\begin{Not}
	For $X$ a scheme and $\BB$ a $\Z$-graded $\OO_X[t^{-1}]$-algebra, where the degree of $t^{-1}$ is $-1$, we write $\overline{\BB}$ for the tensor product $\BB \otimes_{\OO_X[t^{-1}]} \OO_X$ of $\Z$-graded $\OO_X$-algebras, given by the zero section $\OO_X[t^{-1}] \to \OO_X:t^{-1} \mapsto 0$.
	
	Taking the degree-zero part of the map $\OO_X[t^{-1}] \to \overline{\BB}$ gives us a canonical map $\OO_X \to (\overline{\BB})_0$. This gives us a functor $G_X:\Alg^\Z(\A^1_X) \to \Alg(X)$ which sends $\BB$ to $G_X(\BB) \coloneqq (\overline{\BB})_0$. In the case that $X = \Spec A$ we write $G_A \coloneqq G_X$.
\end{Not}

An object $x$ in a category $\CCC$ is \textit{compact} if $\CCC(x,-)$ preserves filtered colimits.

\begin{Lem}
	\label{Lem:GAfiltcoco}
	For any $A \in \Alg$, the functor $G_A$ preserves filtered colimits.
\end{Lem}

\begin{proof}
	The functor $G_A$ factorizes as
	\begin{align*}
	\Alg^\Z_{A[t^{-1}]} \xrightarrow{(-)\otimes_{A[t^{-1}]} A} \Alg^\Z_A \xrightarrow{(-)_0} 
	\Alg_A
	\end{align*}
	Since the first arrow is cocontinuous, it suffices to show that $(-)_0$ preserves filtered colimits.
	
	Since the forgetful functor $U:\Alg_A \to \Space$  is conservative, it suffices to show that the composition $U \circ (-)_0$ preserves filtered colimits. But this is the functor $\Alg^\Z_A(A[x(0)],-)$, which preserves filtered colimits since $A[x(0)]$ is compact by Lemma \ref{Lem:PolMRcompproj}.
\end{proof}

\begin{Rem}
	\label{Rem:GANOTcoco}
	The functor $G_A$ does not preserve all colimits. Indeed, consider the tensor product
	\begin{align*}
	A[t^{-1},u] \otimes_{A[t^{-1}]} A[t^{-1},s] \simeq A[t^{-1},s,u]
	\end{align*}
	in $\Alg^\Z_{A[t^{-1}]}$, where $u$ is in degree $-1$ and $s$ in degree $1$. Then it holds
	\begin{align*}
	G_A(A[t^{-1},s,u]) \simeq A[su] \not\simeq A \otimes_A A \simeq G_A(A[t^{-1},u]) \otimes_{G_A(A[t^{-1}])} G_A(A[t^{-1},s])
	\end{align*}
\end{Rem}

\subsection{Construction}
\label{Par:const}
Let $Z\to X$ be a closed immersion. Consider $X$ as scheme over $\A^1_X$ via the zero section. Identify $\A^1_X$ with $\Spec \OO_X[t^{-1}]$, and endow $\OO_X[t^{-1}]$ with the grading where $t^{-1}$ is of degree $-1$.

Write $\zeta_X$ for the map $B\G_{m,X} \to [\A^1_X/\G_{m,X}]$ induced by the zero section. 
\begin{Def}
	Define $D_{Z/X}$ via the following pullback square.
	\begin{equation}
	\label{Eq:DefDZX}
		\begin{tikzcd}
		 {D_{Z/X}} \arrow[d] \arrow[r]& {\Res_{\zeta_X}([Z/\G_{m,X}])} \arrow[d] \\
		 \A^1_X \arrow[r] & {[\A^1_X/\G_{m,X}]}
		\end{tikzcd}
	\end{equation}
We endow $D_{Z/X}$ with the $\G_{m,X}$-action such that $D_{Z/X} \to {\Res_{\zeta_X}([Z/\G_{m,X}])}$ induces an equivalence $[D_{Z/X}/\G_{m,X}] \simeq {\Res_{\zeta_X}([Z/\G_{m,X}])}$. If $Z \to X$ is of the form $\Spec B \to \Spec A$, we write $D_{B/A} \coloneqq D_{Z/X}$. 
\end{Def}

\begin{Def}
		If $D_{Z/X}$ is affine over $\A^1_X$, then by Theorem \ref{Thm:Equiv} it holds that $D_{Z/X} = \Spec \RR$, where $\RR$ is a $\Z$-graded $\OO_X[t^{-1}]$-algebra. In this case, we call $\RR$ the \textit{extended Rees algebra} of $Z \to X$, and write $\RR_{Z/X}^\ext \coloneqq \RR$ and $\RR_{Z/X} \coloneqq \RR_{\geq 0}$. If $Z \to X$ is of the form $\Spec B \to \Spec A$, then we let $R_{B/A}^\ext$ and $R_{B/A}$ be the $\Z$-graded $A$-algebras corresponding to $\RR_{Z/X}^\ext$ and $\RR_{Z/X}$.
\end{Def}

Observe that $D_{(-)/X}$ is functorial on the category of closed immersions $Z \to X$ for which $D_{Z/X} \to \A^1_X$ is affine. In extension, $R^\ext_{(-)/A}$ is functorial on the category of $A$-algebras $B$ for which $\Spec B \to \Spec A$ is a closed immersion and for which $D_{B/A}$ is affine over $\A^1_A$.

The main goal of this chapter is to show the following:
\begin{Thm}\label{Thm:ReesRep}	
	The stack $D_{Z/X}$ is affine over $\A^1_X$.
\end{Thm}
Observe that the question is local on $X$. We will therefore for now assume that $Z \to X$ is of the form $\Spec B \to \Spec A$, where $A \to B$ is a surjective ring map.

Throughout, we endow $Z$ with trivial $\G_{m,X}$-action. As such, it holds that $[Z/\G_{m,X}] \simeq Z \times B\G_{m}$. We will stick to the notation $[Z/\G_{m,X}]$ to emphasize that this stack lives over $[X/\G_{m,X}] = B\G_{m,X}$.

Let $Q$ be a $\Z$-graded algebra over $A[t^{-1}]$, and write $T \coloneqq [\Spec Q/\G_{m,X}]$. We make the following key observation:
\begin{equation}
\label{Eq:FundComp}
\begin{split}
\St^{\G_{m,X}}_{/\A^1_X}(\Spec Q, D_{Z/X}) &\simeq \St_{/[\A_X^1/\G_{m,X}]}(T,\Res_{\zeta_X}([Z/\G_{m,X}])) \\
&\simeq \St_{/B\G_{m,X}}(T \times_{[\A^1_X/\G_{m,X}]} B\G_{m,X},[Z/\G_{m,X}]  ) \\
&\simeq \Alg^\Z_{A}(B,\overline{Q}) \\
&\simeq \Alg_A(B, G_A(Q))
\end{split}
\end{equation}
We used the  universal property of the Weil restriction mentioned in Remark \ref{Rem:ResFP}, and the fact that taking the zeroth degree of a graded algebra is right adjoint to the functor that endows an algebra with trivial grading.

\begin{Rem}
	The question of representability of mapping stacks and of Weil restrictions has a history of different incarnations. We mention one that seems closest to our situation. 
	
	In \cite[Prop.\ 5.1.14]{HalpernleistnerMapping}  the following is shown. For a flat morphism $\pi: \XXX \to \YYY$ of algebraic stacks, satisfying a formal requirement generalizing properness, with $\YYY$ Noetherian, the Weil restriction $\Res_{\pi}(\ZZZ)$ is affine over $\YYY$ for $\ZZZ$ affine over $\XXX$. 
	
	In our case, we are interested in a Weil restriction $\Res_{\pi}(\ZZZ)$ where $\pi$ is not flat, namely where $\pi$ is the map $B\G_{m,X} \to [\A^1_X/\G_{m,X}]$. Without the flatness assumption, \cite{HalpernleistnerMapping} only gives us results on the question of algebraicity of $\Res_{\pi}(\ZZZ)$, meanwhile we are interested in affineness. Our approach is thus independent of theirs.
\end{Rem}

\begin{Def}
	Write $H_B$ for the functor $\Alg^\Z_{A[t^{-1}]}\to \Space$ that sends $Q$ to $\Alg_A(B, G_A(Q))$.
\end{Def}
\begin{Rem}\label{Rem:HArep}
	Observe that $H_B$ is corepresentable if and only if $H_B$ preserves limits and $\kappa$-filtered colimits for some regular cardinal $\kappa$ by \cite[Prop.\ 5.5.2.7]{LurieHTT}. 
\end{Rem}
\begin{Lem}
	\label{Lem:RedtoHA}
	The functor $H_B$ is corepresentable by a $\Z$-graded $A[t^{-1}]$-algebra $R$ if and only if $D_{B/A}$ is affine over $\Spec A[t^{-1}]$, in which case $D_{B/A} \simeq \Spec R$.
\end{Lem}
\begin{proof}
	 This follows from the observation made in (\ref{Eq:FundComp}).
\end{proof}

\begin{Cor}
	\label{Cor:BcompDBA}
	Suppose that $B$ is compact in $\Alg_A$. Then $D_{B/A}$  is affine over $\Spec A[t^{-1}]$ if and only if $H_B$ preserves limits.
\end{Cor}

\begin{proof}
	Combine Remark \ref{Rem:HArep} with Lemma \ref{Lem:RedtoHA} and the fact that $G_A(-)$ preserves filtered colimits by Lemma \ref{Lem:GAfiltcoco}.
\end{proof}

We will first do two separate cases of the proof of Theorem \ref{Thm:ReesRep}. We will then do the general case in \S \ref{Par:ReesClos}.

\subsection{Rees algebra of a finite quotient}
A map $A \to B$ is a \textit{finite quotient} if $B$ is obtained from $A$ by attaching cells induced by a finite number of pointed maps $\sigma_i:S^{n_i} \to A$, with $n_i \geq 0$ for all $i$. Using multi-index notation, we write the induced map $\bigvee\nolimits_i S^{n_i} \to A$ as $\underline{\sigma}$, and $B$ as $A/(\underline{\sigma})$. See \S \ref{Par:FinQuot} for details.

\begin{Prop}
	\label{Prop:ReesFinQuot}
	Suppose that $B$ is a finite quotient over $A$. Then $D_{B/A}$  is affine over $\Spec A[t^{-1}]$.
\end{Prop}

\begin{proof}
	Write $B = A/(\underline{\sigma})$ for $\underline{\sigma}:S^{\underline{n}} \to A$ the pointed map associated to $B$. Since $B$ is compact over $A$, by Corollary \ref{Cor:BcompDBA} it suffices to show that $H_B$ preserves limits.
	
	Let $Q \in \Alg^\Z_{A[t^{-1}]}$, and write $q$ for the image of $t^{-1}$ under the structure map $A[t^{-1}] \to Q$. Also, write $\underline{\sigma}(Q)$ for the composition of $\underline{\sigma}$ with $A \to Q$. 
	
	We have a fibre sequence $Q \xrightarrow{\times q} Q \to \overline{Q}$, and thus a fibre sequence $Q_1 \to Q_0 \to G_A(Q)$. Therefore, in the following commutative diagram
	\begin{center}
		\begin{tikzcd}
		\Alg_A(B,G_A(Q)) \arrow[rr, bend left] \arrow[r, dashed] \arrow[d] & \Space_*(S^{\underline{n}},Q_1 ) \arrow[r] \arrow[d] & \{ 0 \} \arrow[d] \\
		\{ \underline{\sigma}(Q) \} \arrow[r] & \Space_*(S^{\underline{n}},Q_0 ) \arrow[r] & \Space_*(S^{\underline{n}},G_A(Q) ) 
		\end{tikzcd}
	\end{center}
	the square on the right is Cartesian. By the universal property of $B$ (Lemma \ref{Lem:UPfinquot}), the outer square is Cartesian as well. It follows that there is a dashed arrow as indicated that makes the square on the left  Cartesian. 
	
	We just saw that 
	\begin{align*}
	\Alg_A(B,G_A(-)) \simeq \{\underline{\sigma}(-)\} \times_{\Space_*(S^{\underline{n}},(-)_0)} \Space_*(S^{\underline{n}},(-)_1)
	\end{align*}
	as functor $\Alg^\Z_{A[t^{-1}]} \to \Space$. Since limits in $\Alg^\Z_{A[t^{-1}]}$ are computed degree-wise, this functor thus preserves limits, which was to be shown.
\end{proof}

\subsection{Quasi-smooth $A$-algebras and their Rees algebras}
Recall that a closed immersion $Z \to X$ is \textit{quasi-smooth} if locally on $X$ it is a pullback of the map $\{0\} \to \A^n$, for some $n$, along some map $X \to \A^n$. A surjection $A \to B$ is \textit{quasi-smooth} if the corresponding map on spectra is.
	
\begin{Exm}
	\label{Ex:AX/FisQS}
	Any surjection of the form $A \to B=A[\underline{X}]/(\underline{f})$, where $\underline{f} = f_1,\dots,f_n$ is a sequence of points in $A[\underline{X}]=A[X_1,\dots,X_k]$, is quasi-smooth. Indeed, by \cite[Prop.\ 2.3.8]{KhanVirtual}, it suffices to show that the shifted cotangent complex $\LL_{B/A}[-1]$ is a locally free $B$-module of finite rank. But this follows from \cite[Prop.\ 2.3.14]{KhanVirtual} and its proof.
\end{Exm}

%

\begin{Prop}
	\label{Prop:ReesQS}
	If $A \to B$ is a quasi-smooth surjection, then $D_{B/A}$ is affine over $\Spec A[t^{-1}]$.
\end{Prop}
 
\begin{proof}
	Since the question is local on $\Spec A$, we may assume that $A \to B$ is of the form $A \to A/(a_1,\dots,a_n)$ for certain points $a_i \in A$. Then the claim follows from Proposition \ref{Prop:ReesFinQuot}.
\end{proof}


%

\subsection{Rees algebra of a closed immersion}
\label{Par:ReesClos}
We now show the general case of any closed immersion $Z\to X$. Recall that we may assume $Z \to X$ to be of the form $\Spec B \to\Spec A$. The strategy is roughly as follows. We can obtain $B$ from $A$ by attaching an infinite number of cells, in order of dimension. We show that in each intermediate step $A \to B_\alpha \to B$, it holds that $D_{B_\alpha/A}$ will be affine over $\Spec A[t^{-1}]$. The corresponding $R^\ext_{B_\alpha/A}$ will then assemble into $R^\ext_{B/A}$, as the $B_\alpha$ assemble into $B$.
\begin{Lem}
	\label{Lem:ReesComp}
	Suppose that $D_{B/A}$ is affine over $\Spec A[t^{-1}]$. Let $\sigma: S^n \to B$ be a given pointed map, for $n>0$, and put $B' \coloneqq B/(\sigma)$. Then $D_{B'/A}$ is affine over $\Spec A[t^{-1}]$.
\end{Lem}

\begin{proof}
	By Remark \ref{Rem:HArep} and Lemma \ref{Lem:RedtoHA}, we need to show that $H_{B'}$ preserves limits and $\kappa$-filtered colimits for some regular $\kappa$. This functor can be written as
	\begin{align*}
	H_{B'}: \Alg^\Z_{A[t^{-1}]} &\to \Space\\
	Q & \mapsto \Alg_A(B,G_A(Q)) \times_{\Alg_A(A[S^n],G_A(Q))} \{0\}
	\end{align*}
	Since $D_{B/A}$ is affine over $\Spec A[t^{-1}]$, it suffices to show that $D_{A[S^n]/A}$ is so as well. Since $A[S^n] \simeq A/({S^{n-1} \xrightarrow{0} A})$, this is indeed so by Proposition \ref{Prop:ReesFinQuot}.
\end{proof}	

Clearly, the proof of Lemma \ref{Lem:ReesComp} uses that $n>0$. Luckily, it will turn out that this is exactly the case that we need, essentially since in assembling $B$ from $A$ by attaching $k$-cells, we can work in order of the dimension $k$. 

\begin{Lem}
	\label{Lem:ReesColim}
	Suppose that $B$ is a colimit of $A$-algebras $B_\alpha$ such that $D_{B_\alpha/A}$ is affine over $\Spec A[t^{-1}]$, for all $\alpha$. Then $D_{B/A}$ is affine over $\Spec A[t^{-1}]$.
\end{Lem}

\begin{proof}
	By Lemma \ref{Lem:RedtoHA}, it suffices to show that $H_B$ is corepresentable. Let $Q \in \Alg^\Z_{A[t^{-1}]}$ be given. Then we have natural equivalences
	\begin{align*}
	H_B(Q)=\Alg_A(B,G_A(Q)) &\simeq \lim \Alg_A(B_\alpha,G_A(Q)) \\
	& \simeq \lim \Alg^\Z_{A[t^{-1}]}(R_{B_\alpha/A}^\ext,G_A(Q)) \\
	& \simeq \Alg^\Z_{A[t^{-1}]}(\colim R_{B_\alpha/A}^\ext,G_A(Q))
	\end{align*}
	Hence $D_{B/A}$ is affine over $\Spec A[t^{-1}]$, and in fact $D_{B/A} = \Spec (\colim R_{B_\alpha/A}^\ext)$.
\end{proof}

\begin{Lem}
	\label{Lem:CWapprox}
	For any map $A \to B$, there is a sequence $A = B_{0} \to B_1 \to \dots \to B_n \to \dots$ in $\Alg_{A/B}$, indexed by $n\in \N$, such that the map $\colim B_n \to B$ is an equivalence, and with all $B_k \to B_{k+1}$ obtained by attaching infinitely many $k$-cells to $B_k$.
\end{Lem}

\begin{proof}
	This follows by a small objects type of argument, which  is essentially the same as in \cite[Thm.\ 2.10]{ElmendorfRings}.
\end{proof}

\begin{proof}[Proof of Theorem \ref{Thm:ReesRep}]
	Using Lemma \ref{Lem:CWapprox}, take a sequence $A = B_0 \to B_1 \to \dots \to B_n \to \dots$, indexed by $n \in \N$, such that $B_k \to B_{k+1}$ is obtained by attaching $k$-cells to $B_k$, and with $B = \colim B_n$. 
	
	We will first show that $D_{B_2/A}$ is affine over $\A^1_X$. As before, we observe that $A \to B_2$ is surjective. Moreover, by assumption this map is of the form $A \to A[\underline{X}]/(\underline{f})$ for variables $\underline{X} = X_1,X_2,\dots$ and relations $\underline{f} = f_1,f_2,\dots$ (both possibly infinite).
	
	Let $\CCC$ be the category of sets of the form $\{X_{i_1},\dots,X_{i_n},f_{j_1},\dots,f_{j_m}\}$, where all the $f_{j_k}$ are in $A[X_{i_1},\dots,X_{i_n}]$, and for which $A \to A[X_{i_1},\dots,X_{i_n}]/(f_{j_1},\dots,f_{j_m})$ is surjective. Morphisms in $\CCC$ are the inclusions. For each such object, we have a factorization of $A \to B_2$ as $
	A \to A[X_{i_1},\dots,X_{i_n}]/(f_{j_1},\dots,f_{j_m}) \to B_2$. Now the induced map
	\begin{align*}
	\colim\nolimits_{\{X_{i_1},\dots,X_{i_n},f_{j_1},\dots,f_{j_m}\} \in \CCC} A[X_{i_1},\dots,X_{i_n}]/(f_{j_1},\dots,f_{j_m}) \to B_2
	\end{align*}
	is an equivalence of $A$-algebras. Indeed, we can write
	\begin{align*}
	A[X_{i_1},\dots,X_{i_n}]/(f_{j_1},\dots,f_{j_m}) \simeq A[X_{i_1},\dots,X_{i_n}] \otimes_{A[F_{j_1},\dots,F_{j_m}]} A
	\end{align*}
	from which the claim follows by the fact that colimits commute with tensor products. 
	
	Since all the maps $A\to A[X_{i_1},\dots,X_{i_n}]/(f_{j_1},\dots,f_{j_m})$ are quasi-smooth by Example \ref{Ex:AX/FisQS}, we have that $D_{(A[X_{i_1},\dots,X_{i_n}]/(f_{j_1},\dots,f_{j_m}))/A}$ is affine over $\A^1_X$ by Proposition \ref{Prop:ReesQS}, and thus that $D_{B_2/A}$ is affine over $\A^1_X$ by Lemma \ref{Lem:ReesColim}.
	
	Now let $k\geq 2$, and assume that $D_{B_k/A}$ is affine over $\A^1_X$. We will show that $D_{B_{k+1}/A}$ is as well. With a similar argument as before, we write $B_{k+1}$ as a colimit of $B_k$-algebras $B_\alpha$, which are each obtained from $B_k$ by attaching a finite number of $k$-cell. By Lemma \ref{Lem:ReesComp}, each $D_{B_\alpha/A}$ is affine over $\A^1_X$, hence by Lemma \ref{Lem:ReesColim} so is $D_{B_{k+1}/A}$.
	
	By induction, we conclude that $D_{B_n/A}$ is affine over $\A^1_X$, for all $n \geq 0$. Since $B= \colim B_n$, the statement follows from another application of Lemma \ref{Lem:ReesColim}.	
\end{proof}

\subsection{Structural induction}
We will isolate the proof strategy used in Theorem \ref{Thm:ReesRep}, since we will have chance to use it again on several occasions.

\begin{Def}
	For $A \in \Alg$, write $\Sur_A$ for the full subcategory of $\Alg_A$ spanned by the surjections.
\end{Def}

Let $\mathbf{P}$ be a property of closed immersions $Z \to X$. For $A \in \Alg$, write $\mathbf{P}_A$ for the corresponding property of objects in the category $\Sur_A$.
\begin{Lem}
	\label{Lem:StrucInduc}
	Suppose that $\mathbf{P}$ satisfies the following properties:
	\begin{enumerate}
		\item It is local on $X$.
		\item It holds for any $\Spec B \to \Spec A$ for which $A \to B$ is a finite quotient.
		\item The property $\mathbf{P}_A$ is stable under filtered colimits in $\Sur_A$, for any $A \in \Alg$.
		\item If $\mathbf{P}_A$ holds for a surjection $A \to B$, then it also holds for $A \to B/(\sigma)$, for any pointed map $\sigma:S^n \to B$ with $n \geq 1$.
	\end{enumerate}
	Then $\mathbf{P}$ holds for any closed immersion $Z \to X$.
\end{Lem}
\begin{proof}
	From (1) and (2) it follows that $\mathbf{P}$ holds for any quasi-smooth map $Z \to X$. By (1), we can and will assume that $Z \to X$ is of the form $\Spec B \to \Spec A$.  
	
	Now write $A \to B$ as $A = B_0 \to \dots \to B_n \to \dots \to B$, with $B_k \to B_{k+1}$ obtained by attaching $k$-cells, and $B = \colim B_n$. By (3), it suffices to show that $\mathbf{P}_A$ holds for all $A \to B_n, n \geq 2$.
	
	From the proof of Theorem \ref{Thm:ReesRep} it follows that $B_2$ is in fact a filtered colimit of $A$-algebras $B_\alpha$, such that each $A \to B_\alpha$ is surjective and quasi-smooth. It follows that $\mathbf{P}_A$ holds for all $A \to B_\alpha$, and hence for $A \to B_2$ by (3).
	
	Now suppose that $\mathbf{P}_A$ holds for $A \to B_k$ with $k \geq 2$. Again we can write $B_{k+1}$ as a filtered colimit of $B_k$-algebras $B_\alpha$, each obtained from $B_k$ by attaching a finite number of $k$-cells. Since $\alpha \mapsto B_\alpha$ is a filtered diagram, it also holds that $B_{k+1} \simeq \colim B_\alpha$ as $A$-algebras. Hence, by (3), and by induction on the number of $k$-cells, we may assume that $B_{k+1}$ is of the form $B_k/(\sigma)$ for certain $\sigma:S^{k-1} \to B_k$. By (4), $\mathbf{P}_A$ now also holds for $A \to B_{k+1}$. 
	
	By induction, $\mathbf{P}_A$ holds for all $A \to B_n, n \geq 2$, which was to be shown. 
\end{proof}

\subsection{Base change of Rees algebras}
Let still $Z \to X$ be a closed immersion, let $f:X' \to X$ be a map of schemes,  put $Z' \coloneqq Z \times_X X'$. Recall the map $\zeta_X: B\G_{m,X} \to [\A^1_X/\G_{m,X}]$ induced by the zero section from \S \ref{Par:const}. Letting $X$ vary in this notation, it holds that $\zeta_{X'}$ is the pullback of $\zeta_X$. Observe that we then have a map 
\begin{align*}
k:D_{Z'/X'} \to D_{Z/X}
\end{align*}
induced by the maps $D_{Z'/X'} \to \Res_{\zeta_{X'}}([Z'/\G_{m,{X'}}]) \to \Res_{\zeta_X}([Z/\G_{m,X}])$ and $D_{Z'/X'} \to \A^1_{X'} \to \A^1_{X}$.
\begin{Prop}
	\label{Prop:ReesPull}
	The map $k$ exhibits $D_{Z'/X'}$ as pullback of $D_{Z/X}$ along $\A^1_{X'} \to \A^1_X$. Consequently, we have an equivalence $\RR^\ext_{Z'/X'} \simeq f^*\RR_{Z/X}^\ext$ of $\Z$-graded $\OO_X$-algebras.
\end{Prop}	
\begin{proof}
	This follows from Lemma \ref{Lem:WeilPull}.
\end{proof}

\begin{Cor}
	\label{Cor:ConePull}
	The map $k$ induces an equivalence $\overline{ \RR^\ext_{{Z'}/{X'}}} \simeq f^* \overline{ \RR^\ext_{Z/X} }$ of $\Z$-graded $\OO_X$-algebras.
\end{Cor}

\begin{proof}
	Use Proposition \ref{Prop:ReesPull} and the pasting law for pullbacks.
\end{proof}

In fact, in \S \ref{Par:ReesAdj} we will see that $\OO_X[t^{-1}] \to \RR^\ext_{Z/X}$ is an equivalence in degree $\leq 0$. It follows that $\overline{\RR_{Z/X}^\ext}$ is $\N$-graded.

\subsection{An explicit formula for the Rees algebra of a finite quotient}
\label{Par:ReesForm}
The proof of Theorem \ref{Thm:ReesRep} shows that for any closed immersion $Z \to X$, the Rees algebra $\RR_{Z/X}^\ext$ is locally on $X$ a colimit of Rees algebras of finite quotients. Here we will give an explicit formula for the Rees algebra of finite quotients.

\begin{Lem}
	\label{Lem:AdjEqv}
	For a surjection $A \to B$, we have equivalences
	\begin{align*}
	\Alg^\Z_{A[t^{-1}]}(R_{B/A}^\ext,Q) \simeq \St_{/[\A_X^1/\G_{m,X}]}(\Spec Q,D_{Z/X}) \simeq \Alg_A(B,G_A(Q))
	\end{align*}
	natural in $Q \in \Alg^\Z_{A[t^{-1}]}$ and $B \in \Alg_A$.
\end{Lem}

\begin{proof}
	This follows from Theorem \ref{Thm:ReesRep} and the equivalences in (\ref{Eq:FundComp}).
\end{proof}
 Recall that we can adjoin a simplicial set $K$ to an algebra $A$ to get a new algebra $A[K]$. There is also a graded version of this available. Using this, we define an $A$-algebra $A[v]$ with $v$ free of degree $d$ and level $n$, see \S \ref{Par:Gradn} for details. This has the universal property that $\Alg^\Z_A(A[v],B)$ is equivalent to the space of pointed maps $S^n \to B_d$, for any graded $A$-algebra $B$.
 \begin{Not}
 	For $B \in \Alg^\Z$ and $\tau:S^m \to B_d$ a pointed map, we write $B/(\tau)$ for the pushout in $\Alg^\Z$ of $A[v] \to B$ along the map $A[v] \to A$ induced by $0:S^m \to A$, where $v$ is free of level $m$, degree $d$. Similarly for multi-index notation.
 \end{Not}
\begin{Prop}
	\label{Prop:ReesForm}
	Let $A \to B$ be a finite quotient, say $B = A/(\underline{\sigma})$ for $\underline{\sigma} = \sigma_1,\dots,\sigma_k$, with $\sigma_i: S^{n_i} \to A$ pointed maps. We then have an equivalence
	\begin{align*}
	R_{B/A}^\ext \simeq {A[\underline{v},t^{-1}]}/{(\underline{v}t^{-1} - \underline{\sigma})}
	\end{align*}
	of $\Z$-graded $A[t^{-1}]$-algebras, where $\underline{v} = v_1,\dots,v_k$ with the $v_i$  free of degree 1 in level $n_i$.
\end{Prop}

\begin{proof}
	By Lemma \ref{Lem:AdjEqv} and the proof of Proposition \ref{Prop:ReesFinQuot}, we have an equivalence
	\begin{align*}
	\Alg^\Z_{A[t^{-1}]}(R_{B/A}^\ext,-) \simeq \{\underline{\sigma}(-) \}\times_{\Space_*(S^{\underline{n}},(-)_0 )} \Space_*(S^{\underline{n}},(-)_1)
	\end{align*}
	as functors $\Alg^\Z_{A[t^{-1}]} \to \Space$, where as before $\underline{\sigma}(Q)$ means the composition of $\underline{\sigma}$ with the structure map $A \to Q$.
	
	On the other hand, by construction we have a pullback square
	\begin{center}
	\begin{tikzcd}
	\Alg^\Z_{A[t^{-1}]}({A[\underline{v},t^{-1}]}/{(\underline{v}t^{-1} - \underline{\sigma})}, Q) \arrow[rr] \arrow[d] && \Space_*(S^{\underline{n}}, Q_1) \arrow[d, "{(\id,0)}"] \\
	\Space_*(S^{\underline{n}},Q_1) \arrow[rr, "{(\id, \times t^{-1} - \underline{\sigma})}"]  && \Space_*(S^{\underline{n}},Q_1) \times \Space_*(S^{\underline{n}},Q_0)
	\end{tikzcd}
	\end{center}
	of spaces, for each $Q \in \Alg^\Z_{A[t^{-1}]}$. We therefore also have an equivalence
	\begin{align*}
	\Alg^\Z_{A[t^{-1}]}({A[\underline{v},t^{-1}]}/{(\underline{v}t^{-1} - \underline{\sigma})}, -) \simeq \{\underline{\sigma}(-) \}\times_{\Space_*(S^{\underline{n}},(-)_0 )} \Space_*(S^{\underline{n}},(-)_1)
	\end{align*}
	as functors $\Alg^\Z_{A[t^{-1}]} \to \Space$, from which the claim follows.
\end{proof}

\begin{Cor}
	\label{Cor:RZXdeg1}
	The Rees algebra $\RR_{Z/X}$ is generated in degree 1.
\end{Cor}

\begin{proof}
	Say that a closed immersion $Z \to X$ has property $\mathbf{P}$ if $\RR_{Z/X}$ is generated in degree 1. We will check that $\mathbf{P}$ satisfies the conditions of Lemma \ref{Lem:StrucInduc}.
	
	Clearly, $\mathbf{P}$ is local on $X$. Observe that the case where $Z \to X$ corresponds to a finite quotient $A \to B$ immediately follows from the explicit formula in Proposition \ref{Prop:ReesForm}. Also, when $X = \Spec A$, then $\mathbf{P}$ is stable under filtered colimits in the category $\Sur_A$ of surjections $A \to B$.
	
	Now suppose that $\mathbf{P}$ holds for $\Spec B \to \Spec A$, and let $\sigma: S^n \to B$ with $n\geq 1$ be a pointed map. Put $B' \coloneqq B/(\sigma)$. From the proof of Lemma \ref{Lem:ReesComp} it  follows that $R^\ext_{B'/A} \simeq R^\ext_{B/A} \otimes_{R^\ext_{A'/A}} A[t^{-1}]$, where $A' = A/(S^{n-1} \xrightarrow{0} A)$. Hence by assumption on $B$ and by the case of finite quotients, $R^\ext_{B'/A}$ is generated in degree 1.
\end{proof}

\subsection{The Rees algebra adjunction}
\label{Par:ReesAdj}
\begin{Def}
	Let $\Rees_A$ be the full subcategory of $\Alg^\Z_{A[t^{-1}]} = (\Alg_A^\Z)_{A[t^{-1}]/}$ spanned by those $\Z$-graded $A$-algebra maps $A[t^{-1}] \to R$ for which the induced $\Z_{\leq 0}$-graded map $A[t^{-1}] \to R_{\leq 0}$ is an equivalence. Objects of $\Rees_A$ are called \textit{extended Rees algebras over $A$.}
\end{Def}

\begin{Rem}
	Neither the category $\Sur_A$ nor $\Rees_A$ is presentable, since these categories do not have all products.
\end{Rem}

\begin{Rem}
	\label{Lem:GAcoco}
	Notwithstanding Remark \ref{Rem:GANOTcoco}, the restriction $G_A:\Rees_A \to \Alg_A$ does preserves all colimits. Indeed, by \cite[Prop.\ 4.4.2.7]{LurieHTT}, it suffices to show that $G_A$ preserves pushouts and coproducts of extended Rees algebra.
	
	For the claim on pushouts, let $R',R,R''$ be extended Rees algebras. Then since $(-)\otimes_{A[t^{-1}]} A$ preserves colimits, it holds that
	\begin{align*}
	G_A(R'\otimes_R R'') \simeq (\overline{R'} \otimes_{\overline{R}} \overline{R''})_0 \simeq G_A(R') \otimes_{G_A(R)} G_A(R'')
	\end{align*}
	where in the last equivalence we have used that, by Proposition \ref{Prop:SpclAdj}, $(B'\otimes_B B'')_0 \simeq B'_0 \otimes_{B_0} B''_0$ for $\Z$-graded algebra $B',B,B''$ concentrated in degrees $\geq 0$.
	
	The claim on coproducts is similar, using that also $(-)_0$ commutes with coproducts of $\Z$-graded algebras concentrated in degree $\geq 0$, again by Proposition \ref{Prop:SpclAdj}.
	
	The reason that this argument does not work in the general case $G_A:\Alg_{A[t^{-1}]}^\Z \to \Alg_A$ is because $Q \otimes_{A[t^{-1}]}A$ may not be concentrated in degree $\geq 0$ for $Q \in \Alg^\Z_{A[t^{-1}]}$, as seen in Remark \ref{Rem:GANOTcoco}.
\end{Rem}

\begin{Prop}
	We have an adjunction $R^\ext_{(-)/A} \dashv G_A(-):  \Sur_A\rightleftarrows\Rees_A $.
\end{Prop}

\begin{proof}
	 By Lemma \ref{Lem:AdjEqv}, it suffices to show that $R^\ext_{(-)/A}(\Sur_A) \subset \Rees_A$ and that $G_A(\Rees_A) \subset \Sur_A$.
	
	To show that $R^\ext_{(-)/A}(\Sur_A) \subset \Rees_A$, we use the obvious affine version of Lemma \ref{Lem:StrucInduc}. Clearly, the question is local on $\Spec A$ and stable under filtered colimits in $\Sur_A$. The case where $A \to B$ is a finite quotient follows from Proposition \ref{Prop:ReesForm}. The case where $A \to B$ is of the form $A \to B_0/(\sigma)$, for certain $B_0 \in \Alg_A$ for which $R^\ext_{B_0/A} \in \Rees_A$ and with $\sigma:S^n \to B_0, n \geq 1$ a pointed map, follows from a similar argument as in Corollary \ref{Cor:RZXdeg1}.
	
	Conversely, let $R \in \Rees_A$. Then $\pi_0 G_A(R)$ is the underived quotient $\pi_0A/\pi_0R_1$, since $\pi_0\overline{R}$ is the underived quotient $\pi_0R/(t^{-1})$, and since $(-)_0$ commutes with $\pi_0$. It follows that $A \to G_A(R)$ is surjective.
\end{proof}

Next we want to show that $R^\ext_{(-)/A}$ is fully faithful. We first record the following variation on the well-know fact that a left adjoint $F$ is fully faithful if and only if the corresponding unit $\eta$ is an equivalence. The variation  seems folklore in the setting of $\infty$-categories, but see e.g.\ \cite[\href{https://stacks.math.columbia.edu/tag/07RB}{Tag 07RB}]{stacks-project} for the corresponding statement on 1-categories.

\begin{Lem}
	\label{Lem:ffadj}
Let $F \dashv G: \CCC \rightleftarrows \DDD$ be an adjunction with unit $\eta$. Then $F$ is fully faithful if there is some equivalence $\id_{\CCC} \simeq GF$.
\end{Lem}

\begin{proof}
	Suppose that $\id_{\CCC} \simeq GF$, so that $GF$ is fully faithful. Then for $x,x' \in \CCC$ it holds that in the following sequence
	\begin{align}
	\label{Eq:ffadj}
	\CCC(x,x') \xrightarrow{F} \DDD(Fx,Fx') \xrightarrow{G} \CCC(GFx,GFx') \xrightarrow{\eta_x^*} \CCC(x,GFx')
	\end{align}
	the composition of the first two arrows is an equivalence. By adjunction, the composition of the second two arrows is also an equivalence. By two-out-of-six, we conclude that all arrows in the diagram are equivalences. In particular, the first arrow is an equivalence, i.e.,\ $F$ is fully faithful.
\end{proof}

\begin{Prop}
	The functor $R^\ext_{(-)/A}: \Sur_A \to \Rees_A$ is fully faithful.
\end{Prop}

\begin{proof}
	Let $B \in \Sur_A$ be given. By Lemma \ref{Lem:ffadj}, it suffices to give some equivalence $\varphi_B:B \to G_A(R_{B/A}^\ext)$, natural in $B$. Write $\Spec B \to \Spec A$ as $Z \to X$.
	

	Let $\pi$ be the structure map $B\G_{m,X} \to X$. Recall that $\pi_*: \Alg(B\G_{m,X}) \simeq \Alg^{\Z}(X) \to \Alg(X)$ corresponds to $(-)_0$, and $\pi^*$ to endowing an algebra with the trivial action. Also recall, from Remark \ref{Rem:RelSpec}, that the relative spectrum functor furnishes an equivalence from $\Alg(B\G_{m,X})^\op$  to the category $\Aff_{/B\G_{m,X}}$ of algebraic stacks affine over $B\G_{m,X}$. The inverse of $\Spec$ is given by the global sections functors, denoted $\Gamma$.
	
	Write $\nu:[X \times_{\A^1_X} X / \G_{m,X}] \to B\G_{m,X}$ for the pullback of $\zeta_X$ along itself. By Lemma \ref{Lem:WeilPull}, and since $[D_{Z/X} / \G_{m,X}] \simeq \Res_{\zeta_X}([Z/\G_{m,X}])$, it holds
	\begin{equation}
	\label{Eq:first}
	\begin{split}
		 G_A(R_{B/A}^\ext) &\simeq \pi_* \Gamma [D_{Z/X} \times_{\A^1_X} X/\G_{m,X}] \\ &\simeq \pi_* \Gamma \Res_{\nu} [Z \times_{\A^1_X} X / \G_{m,X}] 
	\end{split}
	\end{equation}
	Recall from Proposition \ref{Prop:SpclAdj} that the functors $\pi_*, \pi^*$ restricted to $\N$-graded algebras also give an adjunction  $\pi_* \dashv \pi^*: \Alg^\N(S) \rightleftarrows \Alg(S)$. Observe that $\Res_{\nu} [Z \times_{\A^1_X} X / \G_{m,X}]$ is affine over $B\G_{m,X}$, since $D_{Z/X} \times_{\A^1_X}  X$ is affine over $\A^1_X$. Therefore, for any $A$-algebra $Q$, it holds
	\begin{align*}
		\Alg_A(\pi_* \Gamma \Res_{\nu} [Z \times_{\A^1_X} X / \G_{m,X}], Q) 
		 &\simeq \Alg^\N_A(\Gamma \Res_{\nu} [Z \times_{\A^1_X} X / \G_{m,X}], \pi^* Q ) \\
		&\simeq \Aff_{/B\G_{m,X}} (\Spec \pi^*Q,\Res_\nu [Z \times_{\A^1_X} X / \G_{m,X}] ) 
	\end{align*}
	Now combining this with (\ref{Eq:first}) and using the universal property of the Weil restriction, we get that
	\begin{align*}
		\Alg_A(G_A(R_{B/A}^\ext),Q ) \simeq \Alg^\Z_{A \otimes_{A[t]} A} (B \otimes_{A[t]} A,Q \otimes_{A[t]} A) \simeq \Alg^\N_A(B,Q \otimes_{A[t]} A)
	\end{align*}
	Now since $(-)_0: \Alg^\N_A \to \Alg_A$ is a left adjoint, it commutes with pushouts. It follows that
	\begin{align*}
		\Alg_A(G_A(R_{B/A}^\ext),Q ) \simeq \Alg_A(B,(Q \otimes_{A[t]} A)_0) \simeq \Alg_A(B,Q)
	\end{align*} 	
	By naturality in $Q$, we thus have an equivalence $ G_A(R_{B/A}^\ext) \simeq B$. Since all of the above equivalences are also natural in $Z$, the claim follows.
\end{proof}

\begin{Rem}
	\label{Rem:EssFibSeq}
	By what we have just seen,  the essential image of the functor $R^\ext_{(-)/A}: \Sur_A \to \Rees_A$ captures the theory of closed immersions into $\Spec A$. In the classical case, the latter is equivalent to the theory of ideals of $A$, as is well-known. In the derived setting, for a closed immersion $A \to B$,  we do have an exact sequence 
	\begin{align*}
		R_{B/A}^\ext \xrightarrow{\times t^{-1}} R_{B/A}^\ext \to \overline{ R^\ext_{B/A}}
	\end{align*}
	from which it follows that $(R_{B/A})_1$ is the fibre of $A \to B$. But observe that in the derived setting there is no obvious way of defining ideals of $A$, let alone recover $B$ from this.
\end{Rem}

\subsection{Example: Classical Rees algebra of a regular embedding}
\label{Par:ClasRees}
Let $A$ be classical, $\underline{a} = a_1,\dots,a_n \in A$ a regular sequence, and $B \coloneqq A/(\underline{a})$. Recall that $\underline{a}$ being regular is equivalent to $B$ being classical. We will show that $R_{B/A}^\ext$ is discrete, and that it coincides with the classical extended Rees algebra of $A \to \pi_0 B \cong B$. 

By Proposition \ref{Prop:ReesForm}, we know that $R^\ext_{B/A}$ is the quotient $A[\underline{v},t^{-1}]/(\underline{v}t^{-1}-\underline{a})$. Now observe that $\{v_it^{-1}- a_i\}_i$ is a regular sequence in $A[\underline{v},t^{-1}]$ by \cite[\href{https://stacks.math.columbia.edu/tag/0625}{Tag 0625}]{stacks-project}, since $\underline{a},\underline{v}$ is regular. It follows that the quotient is isomorphic to the underived, hence discrete, quotient. 

Let $I$ be the (underived) kernel of $A \to B$. Then the map $A[It,t^{-1}] \to R_{B/A}^\ext$ of $A[t^{-1}]$-algebras, defined by sending $a_it$ to $v_i$, is an isomorphism. It follows that $R_{B/A}^\ext$ coincides with the classical extended Rees algebra.

We also make the following observation:
\begin{align*}
\overline{ R^\ext_{B/A}} \simeq \bigoplus\nolimits_{n \geq 0} I^n / I^{n+1} \simeq \Sym^\cl N_{B/A} 
\end{align*}
where the quotient $I^n / I^{n+1}$, the symmetric algebra, and the conormal sheaf are classical. The first equivalence follows from the fact that $t^{-1}$ is regular, and the second from the fact that $\Spec B \to \Spec A$ is a regular embedding in the classical sense.

\subsection{Comparison to the classical Rees algebra}
We will show how to recover the classical extended Rees algebra of $\pi_0A \to \pi_0B$ from $R^\ext_{B/A}$. It will turn out that at most two corrections are needed: first, for $A \to B$ classical, we currently do not know whether $R^\ext_{B/A}$ is always discrete; and second, there are $A \to B$ classical for which $t^{-1}$ is not regular on $\pi_0R^\ext_{B/A}$. 

Let $k$ be a discrete ring. A discrete, $\Z$-graded $k[t^{-1}]$-algebra $Q$ is called \textit{$t^{-1}$-regular} if $t^{-1}$ is a non-zero divisor in $Q$. We write $\CCC$ for the 1-category of such $Q$. We then have an adjunction
\begin{align*}
	(-)^\reg:\cl\Alg^\Z_{k[t^{-1}]} \rightleftarrows \CCC : \varphi
\end{align*}
where the right adjoint $\varphi$ is the obvious forgetful functor, and we write $\cl\Alg$ for the 1-category of discrete algebras. For $Q \in \cl\Alg^\Z_{k[t^{-1}]}$, we can compute $Q^\reg$ either as the image of the natural map $Q \to Q_{t^{-1}}$, or as $Q/K$, where $K \coloneqq \bigcup_{n \geq 1} K_n$ with $K_n$ the kernel of multiplication by $t^{-n}$.

Let now $A \to B$ be a surjection in $\Alg$, and write $R=\cl R^\ext_{\pi_0B/\pi_0A}$ for the classical extended Rees algebra of $\pi_0A \to \pi_0B$. Then $R$ is $t^{-1}$-regular, and $R/(t^{-1})$ is $\pi_0B$ in degree zero. We thus have a natural map $B \to R/(t^{-1})$, which induces a map $\alpha:(\pi_0R^\ext_{B/A})^{\reg} \to R$ by the universal properties of $R^\ext_{(-)/A}$, of $\pi_0(-)$, and of $(-)^\reg$.

\begin{Prop}
	\label{Prop:pi0regisclas}
	The map $\alpha:(\pi_0R^\ext_{B/A})^{\reg} \to \cl R^\ext_{\pi_0B/\pi_0A}$ is an equivalence.
\end{Prop}

\begin{proof}
Let $Q$ be a discrete, $\Z$-graded, $t^{-1}$-regular $\pi_0A[t^{-1}]$-algebra, and write $\CCC$ for the 1-category of such $Q$. Then it holds
\begin{align*}
	\CCC(  (\pi_0R_{B/A}^\ext)^\reg, Q ) &\simeq \Alg^\Z_{A[t^{-1}]}(R^{\ext}_{B/A},Q) \simeq \Alg^\Z_A(B,Q/(t^{-1})) \\ &\simeq \cl\Alg^\Z_{\pi_0A}(\pi_0B,Q/(t^{-1}))
\end{align*}
where $\cl\Alg$ again means discrete hom-set. Here, the first equivalence follows from the universal properties of $(-)^\reg$ and  $\pi_0(-)$, the second from the universal property of $R^\ext_{(-)/A}$ from \S \ref{Par:ReesAdj}, and the third from the fact that $Q$ is $t^{-1}$-regular. It follows that $\alpha$ exhibits $(\pi_0R^\ext_{B/A})^{\reg}$ as having the same universal property as $\cl R^\ext_{\pi_0B/\pi_0A}$ as explained in \S \ref{Par:ReesClas}, and thus that $\alpha$ is an equivalence.
\end{proof}

\begin{Cor}
	Not for all discrete surjections $A \to B$ does it holds that $\pi_0R^\ext_{B/A}$ is $t^{-1}$-regular.
\end{Cor}

\begin{proof}
	Suppose to the contrary that $\pi_0R^\ext_{B/A}$ is $t^{-1}$-regular, for all $B$ over $A$ discrete. Consider maps $A' \leftarrow A \to B$ between discrete rings such that $B' \coloneqq B \otimes_A A'$ is discrete. Let $I$ be the ideal of $A \to B$, and $I'$ the ideal of $A' \to B'$. 
	
	We have  $R^\ext_{B/A} \otimes_A A' \simeq R_{B'/A'}$, by base change of Rees algebras (Proposition \ref{Prop:ReesPull}). Upon applying $\pi_0$ to this, we get that $\pi_0(I^n\otimes_A A') \cong (I')^n$, for any $n\geq 1$, by the assumption that both Rees algebras are $t^{-1}$-regular and by Proposition \ref{Prop:pi0regisclas}. Now apply $(-) \otimes_A A'$ to $I^n \to A \to A/I^n$, and consider the long exact sequence associated to the resulting fibre sequence. Since $\pi_0( A/I^n \otimes_A A' )$ is the underived quotient $A'/(I')^n$, this gives us an exact sequence
	\begin{align*}
		\cdots  \to 0\to \pi_1(A/I^n \otimes_A A') \to \pi_0(I^n \otimes_A A') = (I')^n \hookrightarrow A' \to A'/(I')^n
	\end{align*}
	 which implies that $\pi_1(A/I^n \otimes_A A') = 0$. We give an explicit counter-example to this result.
	
	Consider the case where $A \to B$ is the map $k[x,y,z]/(xy-z^2) \to k[y]$ of $k[y]$-algebras sending $x,z$ to $0$, and $A'$ is $A/(y)$. Then $y$ is a zero-divisor in $A/I^2 = \pi_0k[x,y,z]/(x^2,xy,xz,z^2)$, hence $(A/I^2) \otimes_A A'$ is non-zero in $\pi_1$.
\end{proof}

Here is a concrete example which shows that $R^\ext_{B/A}$ might not be $t^{-1}$-regular, even if it is discrete.

\begin{Exm}
	Let $k$ be a discrete ring, and let $A \to B$ be the map $k[\epsilon_0] \to k$, where $k[\epsilon_0]$ is the dual numbers, i.e., with $\epsilon_0^2 = 0$. We can obtain $B$ from $A$ by an infinite number of cell attachments as follows. First, put $B_1 \coloneqq A/(\epsilon_0)$. Then, from the fibre sequence $A \xrightarrow{\times \epsilon} A \to B_1$, it follows that $\pi_1B_1 = k$, and that $\pi_kB_1 = 0$ for $k > 1$. Let $\epsilon_1$ be the path $\epsilon_0 \simeq 0$, so that $\epsilon_1 \epsilon_0$ generates $\pi_1B_1$. By abuse of notation, we write $B_1 = k[\epsilon_0,\epsilon_1]$, with the understanding that $\epsilon_1$ is a path $\epsilon_0 \simeq 0$.
		
	Inductively, we consider an $A$-algebra $B_n = k[\epsilon_0,\epsilon_1,\dots,\epsilon_n]$, with for $k \geq 1$ each $\epsilon_{k+1}$ in level $k+1$ and exhibiting a homotopy $\epsilon_k \epsilon_0 \simeq 0$. Then we define $B_{n+1}$ as $B_n / (\epsilon_0\epsilon_n)$, which we write as $k[\epsilon_0,\dots,\epsilon_{n+1}]$, with the same convention as for $B_n$. The natural map $\colim B_n \to B$ is now an equivalence. 
	
	To compute $R_{B/A}^\ext$, we first compute all $R^\ext_{B_n/A}$, inductively on $n$. We know that $R^\ext_{B_1/A}$ is $A[t^{-1},v_0]/(v_0t^{-1}-\epsilon_0)$, with $v_0$ in degree 1 and level 0. It will turn out that 
	\begin{align*}
		R^\ext_{B_n/A} = A[t^{-1},v_0,v_1,\dots,v_{n-1},\tau_1,\tau_2,\dots,\tau_n]
	\end{align*}
	where each $v_i$ is in degree 1 and level $i$; each $\tau_j$ is in degree 0 and level $j$; for all $i \geq 0$ we have that $v_{i+1}$ is a homotopy $\epsilon_0 v_i \simeq 0$; $\tau_1$ is a homotopy $v_0 t^{-1} \simeq \epsilon_0$; and for all $j \geq 1$ we have that $\tau_{j+1}$ is a homotopy $v_jt^{-1} \simeq \tau_j \epsilon_0$.
	
	Indeed, assume that it is true for $n$. Then observe that, since $R^\ext_{(-)/A}$ commutes with pushouts, we have
	\begin{align*}
		R^\ext_{B_{n+1}/A} & \simeq R^\ext_{B_{n}/A} \otimes_{R^\ext_{A[S^{n}]/A}} A[t^{-1}] \\ &\simeq A[t^{-1},v_0,v_1,\dots,v_{n-1},\tau_1,\tau_2,\dots,\tau_n] \otimes_{A[t^{-1},w_{n-1}]/(w_{n-1}t^{-1})} A[t^{-1}]
	\end{align*}
	where $w_{n-1}$ is free in degree 1, level $n-1$, and the map $R^\ext_{A[S^{n}]/A} \to R^\ext_{B_n/A}$ sends $w_{n-1}$ to $\epsilon_0 v_{n-1}$. We will show that, by pasting pushout diagrams, we get
	\begin{align*}
		R_{B_{n+1}/A}^\ext \simeq  \left( R^\ext_{B_n/A} / (\epsilon_0 v_{n-1}) \right) \otimes_{A[t^{-1},S^n]} A[t^{-1}]
	\end{align*}
	Indeed, write $R^\ext_{B_n/A} / (\epsilon_0 v_{n-1})$ as $R^\ext_{B_n/A}[v_n]$, with the same conventions as before. Then consider the following diagram.
	\begin{center}
		\begin{tikzcd}
			A[w_{n-1}t^{-1}] \arrow[d] \arrow[r] \ar[rd, phantom,  "(\alpha)" ] & A \arrow[d] \\
			A[t^{-1},w_{n-1}] \arrow[r] \arrow[d] \ar[rd, phantom,  "(\beta)" ] & A[t^{-1},w_{n-1}]/(w_{n-1}t^{-1}) \arrow[d, "f"] \arrow[r] \ar[rd, phantom,  "(\gamma)" ] & R^\ext_{B_n/A} \arrow[d] \\
			A[t^{-1}] \arrow[r] & A[t^{-1},S^n] \ar[rd, phantom,  "(\delta)" ]\arrow[d] \arrow[r] & R^\ext_{B_n/A}[v_n] \arrow[d] \\
			& A[t^{-1}] \arrow[r] & R^\ext_{B_{n+1}/A}
		\end{tikzcd}
	\end{center}
	Since $(\alpha)$ and $(\alpha) + (\beta)$ are both pushouts, it follows that $(\beta)$ is a pushout, where the map $f$ is the quotient map 
	\begin{align*}
		A[t^{-1},w_{n-1}]/(w_{n-1}t^{-1}) \to A[t^{-1},w_{n-1}]/(w_{n-1}t^{-1}, w_{n-1}) \simeq A[t^{-1},S^n]
	\end{align*}
	Since $A[t^{-1},w_{n-1}]/(w_{n-1}t^{-1}) \to R^\ext_{B_n/A}$ sends $w_{n-1}$ to $\epsilon_0v_{n-1}$, it follows that $(\beta) + (\gamma)$, hence $(\gamma)$, is a pushout. Now, since $(\gamma) + (\delta)$ is a pushout, also $(\delta)$ is a pushout, which is what we wanted.

	 Now note that $A[t^{-1},S^n] \to  R^\ext_{B_n/A}[v_n]$ is determined by the two different homotopies $v_nt^{-1}$ and $\tau_n \epsilon_0$, both establishing $\epsilon_0 v_{n-1}t^{-1} \simeq 0$; glued together they give a map $\sigma:S^n \to R^\ext_{B_n/A}[v_n]$. Thus
	\begin{align*}
		R^\ext_{B_{n+1}/A} \simeq R^\ext_{B_n/A} / (\epsilon_0 v_{n-1}, \sigma) \simeq R_{B_n/A}^\ext[v_n,\tau_{n+1}]
	\end{align*}
	where $\tau_{n+1}$ is a homotopy $v_nt^{-1}\simeq \tau_n \epsilon_0$.

	Passing to the colimit gives us	$R^\ext_{B/A} \simeq A[t^{-1},v_0,v_1,\dots,\tau_1,\tau_2,\dots]$, with the same conventions on the $v_i$ and $\tau_j$. 	We see that $R^\ext_{B/A}$ is discrete, and thus that
	\begin{align*}
		R^\ext_{B/A} \cong \pi_0 A[t^{-1},v_0] / (v_0 t^{-1} - \epsilon_0, \epsilon_0 v_0)
	\end{align*}
	 It is however not $t^{-1}$-regular, since $v_0^2t^{-1} = v_0 \epsilon_0 = 0$, but $v_0^2 \not= 0$. Since the kernel of multiplication by $t^{-k}$ is generated by $v_0^2$ for any $k \geq 1$, we get that 
	\begin{align*}
		(R^\ext_{B/A})^\reg \cong \pi_0 A[t^{-1},v_0]/(v_0t^{-1}-\epsilon_0,v_0\epsilon_0,v_0^2) \cong A[\epsilon_0t, t^{-1}]
	\end{align*}
	where the last isomorphism is given by $v_0 \mapsto \epsilon_0t$. We thus indeed recover the classical extended Rees algebra.
\end{Exm}

\begin{Rem}
	Let $A \to B$ be discrete. Then from the long exact sequence associated to the fibre sequence $R^\ext_{B/A} \xrightarrow{\times t^{-1}} R^\ext_{B/A} \to \overline{R^\ext_{B/A}}$ we deduce two things. First, that $(R^\ext_{B/A})_1$ is always discrete. Second, that $\pi_0R^\ext_{B/A}$ is $t^{-1}$-regular, i.e., equal to the classical extended Rees algebra, if and only if $\pi_1 \overline{R^\ext_{B/A}} = 0$.
\end{Rem}

\section{Blow-ups}
\label{Sec:Bl}
We will define the blow-up of an arbitrary closed immersion $Z \to X$ using the corresponding Rees algebra, and recover the construction from \cite{KhanVirtual} in the quasi-smooth case. We close with the deformation to the normal cone, which generalizes \cite[Thm.\ 4.1.13]{KhanVirtual} to arbitrary closed immersions.
\subsection{Blow-ups via Rees algebras}
Recall that we defined the projective spectrum for quasi-coherent, $\N$-graded algebras generated in degree 1. Since  $\RR_{Z/X}$ is generated in degree 1 by Corollary \ref{Cor:RZXdeg1}, the following thus makes sense.
\begin{Def}
	Let $Z \to X$ be a closed immersion. Define the \textit{blow-up} of $Z$ in $X$ as the scheme $\Bl_ZX \coloneqq \Proj \RR_{Z/X}$ over $X$.
\end{Def}

For $Z \to X$ classical, we write $\Bl^\cl_ZX$ for the classical blow-up of $Z$ in $X$. The following example show that in general we do not have $\Bl^\cl_{Z_\cl}X_\cl \simeq (\Bl_ZX)_\cl$. 

\begin{Exm}
	Let $A$ be a discrete ring and $B \coloneqq A/(0)$ the derived quotient. Then $R \coloneqq R_{B/A}^\ext$ is the discrete graded ring $A[st,t^{-1}]/(s)$ with $\deg s =0$ and $\deg t =1$, since $s$ is regular on the discrete ring $A[st,t^{-1}]$. However, $t^{-1}$ is not a regular element of $R$, and therefore $\overline{R}$ is not discrete. Indeed, in degree $d>0$, we have a fibre sequence $R_d \xrightarrow{\times t^{-1}} R_{d-1} \to \overline{R}_{d-1}$ of $A$-modules, which is just the fibre sequence $A \xrightarrow{0} A \to B$. Hence the underlying $A$-module of $\overline{R}$ is $\bigoplus_{d \geq 0} B$.
	
	To compute the blow-up of $\Spec B$ in $\Spec A$, write $R$ as $A[v,t^{-1}]/(vt^{-1})$, where $\deg (v) = 1$. Then $R_{\geq 0} \cong A[v]$, and hence
	\begin{align*}
	\Bl_{\Spec B} \Spec A = \Proj R_{\geq 0} \cong \Proj A[v] \cong \Spec A,
	\end{align*}
	since the derived projective spectrum of a discrete ring is given by the classical Proj-construction. Observe that this result is what we expect according to \cite[Thm.\ 4.1.5.viii]{KhanVirtual}.
\end{Exm}

\begin{Exm}
	\label{Exm:Blcl}
	Let $Z \to X$ be a regular embedding of classical schemes. Then $\Bl_ZX$ is classical, and coincides with $\Bl^\cl_ZX$. This follows from the example in \S \ref{Par:ClasRees}, together with the fact that the projective spectrum of a discrete algebra $\RR$ coincides with the classical projective spectrum of $\pi_0\RR$.
\end{Exm}

\begin{Prop}
	\label{Prop:BlBase}
	Blow-ups commute with pullbacks. That is, for $Z \to X$ a  closed immersion and $X' \to X$ any map, we have $\Bl_{Z \times_{X} X'}X' \simeq (\Bl_{Z}X )\times_{X} X'$. 
\end{Prop}

\begin{proof}
	This follows from Proposition \ref{Prop:ReesPull}, together with the fact the $\Proj(-)$ commutes with pullbacks. 
\end{proof}

\begin{Cor}
	For any closed immersion $Z \to X$, the map $\Bl_Z X \to X$ is an equivalence on $X \setminus Z$.
\end{Cor}

\begin{proof}
	Note that $(\Bl_ZX) \times_X (X \setminus Z) \simeq \Bl_\varnothing (X \setminus Z)$ by Proposition \ref{Prop:BlBase}. It thus suffices to show that for any scheme $X'$, it holds that $\Bl_\emptyset X' \to X'$ is an equivalence. The latter follows from the explicit formula in Proposition \ref{Prop:ReesForm}. Namely, the map $A \to 0$ is equivalent to the quotient $A \to A/(1)$. Hence $R_{0/A} \simeq A[t]$, and thus $\Bl_\emptyset \Spec A \simeq \Proj A[t] \simeq \Spec A$.
\end{proof}

\begin{Exm}
	For $X \in \Sch$ it holds that $\Bl_X X = \varnothing$. To see this, by Proposition \ref{Prop:BlBase}, we may assume that $X = \Spec A$. Now $R_{A/A}^\ext = A[t^{-1}]$, and $\Proj A = \varnothing$, from which the claim follows.
\end{Exm}

\subsection{Naturality in the target}
Suppose we have a commutative diagram of schemes of the form 
\begin{center}
	\begin{tikzcd}
	Z \arrow[d] \arrow[r, "h"] & W \arrow[d] \\
	X \arrow[r, "f"] & Y
	\end{tikzcd}
\end{center}
where the vertical arrows are closed immersions. Recall that $\zeta_X:B\G_{m,X} \to [\A^1_X/\G_{m,X}]$ is the map induced by the zero section, likewise for $\zeta_Y$. Then, by Lemma \ref{Lem:WeilPull}, we have a map 
\begin{align*}
\alpha:\Res_{\zeta_X}([Z/\G_{m,X}]) \to \Res_{\zeta_Y}([Z/\G_{m,Y}])
\end{align*}
\begin{Lem}
	\label{Lem:DZXnat}
	The map $\alpha$ induces an equivariant map $D_{Z/X} \to D_{W/Y}$ over $\A^1_Y$.
\end{Lem}
\begin{proof}
	We have maps $D_{Z/X} \to \Res_{\zeta_X}([Z/\G_{m,X}]) \to \Res_{\zeta_Y}([Z/\G_{m,Y}])$ and $D_{Z/X} \to \A^1_X \to \A^1_Y$ that make the obvious square commute, which gives us the map $D_{Z/X} \to D_{W/Y}$ over $\A^1_Y$. It is equivariant because it is induced by equivariant maps.
\end{proof}
\begin{Prop}
	Suppose $f:X \to Y$ is a closed immersion and that $h$ is the identity. Then $D_{Z/X} \to D_{Z/Y}$ induces a closed immersion $\Bl_ZX \to \Bl_ZY$ over $Y$.
\end{Prop}
\begin{proof}
	The question is local on $Y$, so we may assume that $Z \to X \to Y$ is of the form $\Spec B \to \Spec A \to \Spec C$. Write $\psi:R_{B/C} \to R_{B/A}$ for the graded map of $C$-algebras induced by Lemma \ref{Lem:DZXnat}. We claim that $\psi$ is surjective, and that the induced map on projective spectra gives us the closed immersion $\Bl_ZX \to \Bl_ZY$. 
	
	For surjectivity, since $R_{B/A}$ is generated in degree 1, and since $\psi$ is the surjection $C \to A$ in degree 0, it suffices to show that $\psi$ is surjective in degree 1. Recall from Remark \ref{Rem:EssFibSeq} that multiplication with $t^{-1}$ induces a fibre sequence $(R_{B/A})_1 \to A \to B$. Write $(R_{B/A})_1 \coloneqq F_{B/A}$, and likewise $F_{B/C}$ for the fibre of $C \to B$. 
	
	The long exact sequence associated to $F_{B/C} \to C \to B$ lives over the sequence associated to $F_{B/A} \to A \to B$ via obvious maps. From a diagram chase on the resulting diagram, it follows that $\pi_0 F_{B/C} \to \pi_0F_{B/A}$ is surjective.

	It is clear from the construction of the relative spectrum that $\Proj_Y(R_{B/A}) \simeq \Proj_X(R_{B/A})$, since $X,Y$ are affine. Hence, since $\psi$ is surjective, it induces a closed immersion $g: \Bl_ZX \to \Bl_ZY$ over $Y$ by \S \ref{Par:ProjFunct}.	
\end{proof}

\subsection{Virtual Cartier divisors}
We review the notion of virtual Cartier divisors, which is used in \cite{KhanVirtual} to define blow-ups in quasi-smooth centres.

\begin{Def}
	Let $S$ be a scheme.
	\begin{itemize}
		\item A \textit{virtual Cartier divisor} on $S$ is a closed immersion $D \to S$ of virtual codimension one, that is, it is Zariski locally on $S$ given as $\Spec A/f \to \Spec A$ for some point $f \in A$.
		\item A \textit{generalized Cartier divisor} on $S$ is a line bundle $\LL \in \Pic(S)$, together with a map $\LL \to \OO_S$ in $\QCoh(S)$.
	\end{itemize}	
\end{Def}
If $i:D \to S$ is a virtual Cartier divisor on $S$, then the  fibre $\LL \to \OO_S \to i_* \OO_D$ gives a generalized Cartier divisor on $S$. And in fact, this construction gives an equivalence of the space of virtual Cartier divisors to the space of generalized Cartier divisors.

\begin{Def}
\label{Def:VCD}	
	Let $i:Z \to X$ be a closed immersion. The \textit{conormal sheaf} of $i$ is the shifted cotangent bundle $\NN_{Z/X} \coloneqq \LL_{Z/X}[-1]$. 
	
	Now let $S$ be a scheme over $X$. Then a \textit{virtual Cartier divisor on $S$ over $(X,Z)$}, written $(S,D)$, is a commutative square
	\begin{equation}\label{Eq:VCD}
	\begin{tikzcd}
	D \arrow[r] \arrow[d, "g"] & S \arrow[d] \\
	Z \arrow[r] & X
	\end{tikzcd}	
	\end{equation}
	such that: 
	\begin{enumerate}
		\item $D \to S$ is a virtual Cartier divisor;
		\item the square is Cartesian in the classical sense on the underlying classical schemes; 
		\item $g^* \NN_{Z/X} \to \NN_{D/S}$ is surjective.
	\end{enumerate}
	Write $\vDiv_ZX(S\to X)$ for the space of virtual Cartier divisors on $S$ over $(X,Z)$. 
\end{Def}

Observe that virtual Cartier divisors on $S$ over $(X,Z)$ pull back along maps $S' \to S$ over $X$, since $\NN_{D/S}$ is stable under  pull-backs. Thus $\vDiv_ZX$ gives a functor $\dSch_{/X}^\op \to \Space$. One of the main results in \cite[Thm.\ 4.1.5]{KhanVirtual} is the following: if $Z \to X$ is quasi-smooth, then $\vDiv_ZX$ is representable by a scheme.

\begin{Lem}
	\label{Lem:vDivPB}
	Virtual Cartier divisors over $(X,Z)$ pull  back along maps $X' \to X$. That is, for $Z' \coloneqq X\times_{X'} Z$ and $(S,D)$ a virtual Cartier divisor over $(X,Z)$, the pullback $D' \to S'$ of $D \to S$ is a virtual Cartier divisor over $(X',Z')$.
\end{Lem}	
\begin{proof}
	The first two conditions from Definition \ref{Def:VCD} are clear. The third condition follows from the fact that the conormal sheaf commutes with base change and the fact that pulling back quasi-coherent sheaves is exact.
\end{proof}

\subsection{A virtual Cartier divisor over $(X,Z)$}
\label{Par:univVCD}
Let $\BB$ be the $\OO_X$-algebra corresponding to the quasi-smooth closed immersion $i:Z \to X$. Consider $\overline{ \RR^\ext_{Z/X} }$ as a graded $\OO_Z$-algebra via the map $\BB   \to \overline{ \RR^\ext_{Z/X} }$. Throughout, put $E_{Z/X} \coloneqq  \Proj_Z (\overline{ \RR^\ext_{Z/X} })$. We will first give a  virtual Cartier divisor over $(X,Z)$ of the form $(\Bl_ZX, E_{Z/X})$,
and in the next section show that the map $\Bl_ZX \to \vDiv_ZX$ thus induced by the universal property of $\vDiv_ZX$ is an equivalence in the case that $Z \to X$ is quasi-smooth.

To get a commutative square as in (\ref{Eq:VCD}), with $(S,D) = (\Bl_ZX,E_{Z/X})$, by Proposition \ref{Prop:BlBase} it suffices to give a map $k:E_{Z/X} \to \Bl_{Z \times_X Z/Z}$ of schemes over $Z$. By \S \ref{Par:ProjFunct}, it thus suffices to give a surjective map 
\begin{align*}
\kappa:i^* \RR_{Z/X} \simeq \RR_{Z\times_XZ/Z}  \to \overline{ \RR^\ext_{Z/X} }
\end{align*}
of graded $\OO_Z$-algebras. We define $\kappa$ to be the adjoint of the map 
\begin{align*}
\RR_{Z/X} \to \RR_{Z/X}^\ext \to \overline{\RR_{Z/X}^\ext}\simeq  i_*\overline{\RR_{Z/X}^\ext} 
\end{align*}

\begin{Lem}
	\label{Lem:kappa}
	The map $\kappa$ is a surjective map of graded $\OO_Z$-algebras.
\end{Lem}

\begin{proof}
	The question is local on $X$, so we assume that $Z \to X$ corresponds to a surjection $A \to B$. Then $\kappa$ is the composition of the map $R_{B \otimes_A B/B} \to R_{B \otimes_A B/B}^\ext$ with the map
	\begin{align*}
	\kappa':R_{B \otimes_A B/B}^\ext \simeq R_{B/A}^\ext \otimes_A B \to \overline{ R^\ext_{B/A}}
	\end{align*}
	induced by the inclusion of $B$ as the degree-zero part of $\overline{ R^\ext_{B/A}}$ and by the quotient map $R_{B/A}^\ext \to \overline{R_{B/A}^\ext}$. 
	
	To see that $\kappa$ is surjective, notice that $R_{B/A}^\ext \to \overline{R_{B/A}^\ext}$ is surjective, hence that $\kappa'$ is surjective. Since $\overline{R_{B/A}^\ext}$ is concentrated in positive degrees and the map $R_{B \otimes_A B/B} \to R_{B \otimes_A B/B}^\ext$ is the identity in positive degrees, the claim follows.
\end{proof}

Consider now the square
\begin{equation}
		\label{Eq:VCDTS}
		\begin{tikzcd}
			E_{Z/X}  \arrow[d, "g"] \arrow[r, "k"] & \Bl_ZX \arrow[d] \\
			Z \arrow[r] & X
		\end{tikzcd}
	\end{equation}
induced by $\kappa$. 

\begin{Lem}
	The map $k$ is a virtual Cartier divisor.
\end{Lem}

\begin{proof}
	The question is local on $X$, so we assume that $Z \to X$ corresponds to a surjection $A \to B$. Let $f$ be a point in $(R_{B/A})_1$, and write $f'$ for the image of $f$ in $\overline{ R^\ext_{B/A} }$. The map $(R_{B/A})_f \to (R^\ext_{B/A})_f$ is an equivalence in degree 0. By exactness of localization, it follows that locally $k$ corresponds to the ring map
	\begin{align*}
		(R_{B/A})_{(f)} \to (R^\ext_{B/A}/(t^{-1}))_{(f')} \simeq (R^\ext_{B/A})_{(f')} / (t^{-1})
	\end{align*}
	from which it follows that $k$ is a virtual Cartier divisor.
\end{proof}

\begin{Lem}
	The underlying square of (\ref{Eq:VCDTS}) of classical schemes is Cartesian.
\end{Lem}

\begin{proof}
	Again we assume that $Z \to X$ is of the form $\Spec B \to \Spec A$. Write $\Proj^\cl(-)$ for the classical projective spectrum. To see that the square (\ref{Eq:VCDTS}) is Cartesian on the underlying classical schemes, it suffices to show that the natural map
	\begin{align*}
		\varphi:	\pi_0 R_{B/A} \otimes_{\pi_0A} \pi_0 B \to \pi_0 \overline{  R^\ext_{B/A} }
	\end{align*}
	is an isomorphism, where the tensor product on the left is underived.
	
	Let $I$ be the kernel of $\pi_0 A \to \pi_0 B$. Let $\tau$ be the map $\pi_0(R_{B/A}^\ext)_1 \to \pi_0 A$ given by multiplication with $t^{-1}$. Then $I = \im(\tau)$. It follows that $\varphi$ is the graded map of $\N$-graded rings
	\begin{align*}
		\pi_0(R_{B/A}) / (\im (\tau) \pi_0(R_{B/A}) ) \to \pi_0(R_{B/A}^\ext)/(t^{-1})
	\end{align*}
	induced by the inclusion $\pi_0(R_{B/A}) \to \pi_0(R_{B/A}^\ext)$. To show that this is an isomorphism, it suffices to check this on each degree $d \geq 0$. And indeed $\varphi$ is an isomorphism in degree $d$, since
	\begin{align*}
		t^{-1}(\pi_0 R^\ext_{B/A})_{d+1} \cong t^{-1}(\pi_0 R^\ext_{B/A})_{1}(\pi_0 R^\ext_{B/A})_{d} \cong \im(\tau) (\pi_0R_{B/A})_d,
	\end{align*} 
	since $\pi_0(R_{B/A})$ is generated in degree 1. 
\end{proof}

\begin{Rem}
	\label{Rem:ConnFib}
	Let $K \xrightarrow{\alpha} L \xrightarrow{\beta} M$ be a sequence of spaces. We then have a fibre sequence $\fib(\alpha) \to \fib (\beta \alpha) \to \fib(\beta)$, where $\fib(-)$ means the fibre. It follows that if $S \to S'$ and $T \to T'$ are maps of $R$-algebras with connected fibres, then so is the map $S \otimes_R T \to S' \otimes_R T'$, since it factorizes as 
	\begin{align*}
		S \otimes_R T \to S \otimes_R T' \to S' \otimes_R T',
	\end{align*}
	both of which have connected fibres.
\end{Rem}

\begin{Prop}
	\label{Prop:VCD}
	The square (\ref{Eq:VCDTS}) is a virtual Cartier divisor over $(X,Z)$.
\end{Prop}

\begin{proof}
	It remains to show that $g^*\NN_{Z/X} \to \NN_{E_{Z/X} / \Bl_ZX}$ is surjective on $\pi_1$. With the same argument as in \cite[Rem.\ 4.1.3.iii]{KhanVirtual} --- which also works in the non-quasi-smooth case --- under the facts that we have already shown, it holds that (\ref{Eq:VCDTS}) is a virtual Cartier divisor over $(X,Z)$ if and only if $\OO_{\Bl_{Z\times_X Z/Z}} \to h_* \OO_{E}$ is surjective on $\pi_1$, for the obvious map $h: E \to \Bl_{Z \times_XZ /Z}$.
	
	Say that a closed immersion $Z \to X$ has property $\mathbf{P}$ if the corresponding map $\OO_{\Bl_{Z\times_X Z/Z}} \to h_* \OO_{E}$ is surjective on $\pi_1$. We will verify that $\mathbf{P}$ satisfies the conditions of Lemma \ref{Lem:StrucInduc}, and thus conclude that all closed immersions satisfy $\mathbf{P}$. Clearly, $\mathbf{P}$ is local on $X$ and stable under pullbacks. 
	
	Now let $Z \to X$ be a closed immersion corresponding to a finite quotient $A \to B$, say where $B = A/(\underline{\sigma})$ for certain $\underline{\sigma} = \sigma_1,\dots,\sigma_k$, with $\sigma_i:S^{n_i} \to A$ pointed maps. Let $K$ be the pointed space $S^{n_1} \vee  \dots \vee S^{n_k}$. Put $\A^K \coloneqq \Spec (\Z[K])$. Observe that $Z \to X$ is a pullback of $\{0\} \to \A^K$ along the map $X \to \A^K$ induced by $\underline{\sigma}$. To show that $\mathbf{P}$ holds for $Z \to X$, it thus suffices to show that $\mathbf{P}$ holds for $\{0\} \to \A^K$, by Lemma \ref{Lem:vDivPB}.

	Put $W \coloneqq \Spec (\Z \otimes_{\Z[K]} \Z)$. It suffices to show that $\kappa:R_{W/\{0\}} \to \overline{R^\ext_{\{0\}/\A^K}}$ is surjective on $\pi_1$. By Proposition \ref{Prop:ReesForm}, $\kappa$ is the obvious map
	\begin{align*}
		\big( \Z[\underline{v},t^{-1}]/(\underline{v}t^{-1}) \big)_{\geq 0} \to \Z[\underline{s},\underline{v},t^{-1}]/(\underline{v}t^{-1} - \underline{s},t^{-1})
	\end{align*}
	where $\underline{v} = v_1,\dots,v_k$ and $\underline{s} = s_1,\dots,s_k$ are variables with $v_i, s_i$ free in degree 1 and level $n_i$. Now the map
	\begin{align*}
		  \Z[\underline{s},\underline{v},t^{-1}]/(\underline{v}t^{-1} - \underline{s},t^{-1}) \to \big( \Z[\underline{v},t^{-1}]/(\underline{v}t^{-1}) \big)_{\geq 0}
	\end{align*}
	of $\Z[\underline{v}]$-algebras, defined by $s_i \mapsto v_it^{-1}$ and $t^{-1} \mapsto 0$, is a section of $\kappa$. The claim follows.
	
	Property $\mathbf{P}$ on any closed immersion $\Spec B \to \Spec A$ is equivalent to the property that
	\begin{align*}
		\kappa_{(f)}: (R_{B \otimes_A B/B})_{(f)} \to \left( \overline{R_{B/A}^\ext} \right)_{(f')}
	\end{align*}
	is surjective on $\pi_1$, for any $f \in (R_{B \otimes_A B/B})_1$ and $f' \coloneqq \kappa(f)$. This, in turn, is  equivalent to the fibre of $\kappa_{(f)}$ being connected, for all $f$, because $\kappa_{(f)}$ is an isomorphism on $\pi_0$. From the fact that taking fibres commutes with filtered colimits, it now follows that $\mathbf{P}$ is stable under filterd colimits in $\Sur_A$.
	
	Finally, suppose that $\mathbf{P}$ holds for a closed immersion $Z \to X$ of the form $\Spec B \to \Spec A$, and let $\sigma:S^n \to B$ be given with $n \geq 1$. We need to show that $\mathbf{P}$ holds for $\Spec B' \to \Spec A$, where $B' \coloneqq A/(\sigma)$. 
	
	By twice applying Proposition \ref{Prop:ReesPull}, we see that 
	\begin{align*}
		R_{B' \otimes_A B'/B'} \simeq R_{B \otimes_A B/B} \otimes_B B' \otimes_B B'
	\end{align*}
	Also, since $R^\ext_{(-)/A}$ is a left-adjoint, we see that
	\begin{align*}
		\overline{R_{B'/A}^\ext} \simeq R_{B/A}^\ext \otimes_{R^\ext_{A[S^n]/A}} A
	\end{align*}
	Now let $f \in  (R_{B' \otimes_A B'/B'})_1$, put $f' \coloneqq \kappa(f)$. We need to show that the fibre of $\kappa_{(f)}$ is connected.

	For simplicity, assume that $f$ is of the form $g \otimes \lambda$, with $g \in R_{B \otimes_A B/B}$. By assumption, the map
	$
		(R_{B \otimes_A B/B})_{(g)} \to \left(\overline{R_{B/A}^\ext}\right)_{(g')}
	$
	has a connected fibre, where $g'$ is the image of $g$. Also, the map $B' \otimes_B B' \to B'$ is an isomorphism on $\pi_0$ and is surjective on $\pi_1$, since it has a section, and thus has a connected fibre. It follows that,  upon tensoring these two maps, we obtain the map
	\begin{align*}
		\alpha_{(f)}:(R_{B \otimes_A B / B} \otimes_B B' \otimes_B B' )_{(f)} \to \left( \overline{R^\ext_{B/A}} \otimes_B B' \right)_{(f'')}
	\end{align*}
	which has connected fibre by Remark \ref{Rem:ConnFib}, where $f''$ is the image of $f$.
	
	We can write
	\begin{align*}
		\overline{R_{B/A}^\ext} \otimes_B B' \simeq R^\ext_{B/A} \otimes_{A[S^n, t^{-1}]} A \simeq R^\ext_{A[S^n]/A} \otimes_{A[S^n,t^{-1}]} R^\ext_{B/A} \otimes_{R^\ext_{A[S^n]/A}} A
	\end{align*}
	This description gives us a map 
	\begin{align*}
		\beta: \overline{R_{B/A}^\ext} \otimes_B B' \to  R_{B/A}^\ext \otimes_{R^\ext_{A[S^n]/A}} A\simeq \overline{R^\ext_{B'/A}}
	\end{align*}
	such that $\beta_{(f'')} \circ \alpha_{(f)} \simeq \kappa_{(f)}$. Hence, since the fibre of $\alpha_{(f)}$ is connected, by Remark \ref{Rem:ConnFib}, it suffices to show that so is the fibre of $\beta_{(f'')}$. The latter follows from the fact that $\pi_0(\beta_{(f'')})$ is injective and that $\beta_{(f'')}$ has a section.
\end{proof}

\subsection{The functor of points of $\Bl_ZX$ in the quasi-smooth case}
\begin{Def}
	Write $h_{Z/X}:\Bl_ZX \to \vDiv_ZX$ for the map induced by the virtual Cartier divisor $(\Bl_ZX, E_{Z/X})$ from Proposition \ref{Prop:VCD}.
\end{Def}

\begin{Lem}
	\label{Lem:hZXBC}
	The map $h_{Z/X}$ commutes with base change. That is, for given $X' \to X$ and $Z' \coloneqq Z\times_X X'$, the following diagram 
	\begin{center}
		\begin{tikzcd}
			\Bl_{Z'}X' \arrow[d, "f"] \arrow[r, "h_{Z'/X'}"] & \vDiv_{Z'}X' \arrow[d, "g"] \\
			\Bl_ZX \arrow[r, "h_{Z/X}"] & \vDiv_ZX
		\end{tikzcd}
	\end{center}
	where the vertical arrows are induced by Lemma \ref{Prop:BlBase} and \cite[Thm.\ 4.1.5.2]{KhanVirtual}, commutes.
\end{Lem}

\begin{proof}
	Both ways around the square are induced by the virtual Cartier divisor $(\Bl_{Z'}X', E_{Z'/X'})$ over $(X,Z)$. The case $g\circ h_{Z'/X'}$ follows from the fact that $g$ is defined by composing a virtual Cartier divisor $(S,D)$ over $(X,Z)$ with the pullback square of $(X',Z')$ over $(X,Z)$. The case $h_{Z/X} \circ f$ follows since $f$ is defined via pulling back, together with Corollary \ref{Cor:ConePull}.
\end{proof}

\begin{Thm}\label{Thm:BlDiv}
	If $Z \to X$ is quasi-smooth, then the map $h_{Z/X}:\Bl_ZX \to \vDiv_ZX$ is an equivalence.
\end{Thm}

\begin{proof}
	The question is local on $X$, so we may assume that $Z \to X$ is a base change of the map $\{0\} \to \A^n$ along a map $X \to \A^n$. By Lemma \ref{Lem:hZXBC}, Proposition \ref{Prop:BlBase} and \cite[Thm.\ 4.1.5.ii]{KhanVirtual}, we may assume that $Z \to X$ is of the form $\{0\} \to \A^n$. In this case $h_{Z/X}$ is the equivalence $\Bl_{\{0\}}\A^n \simeq \Bl^\cl_{\{0\}}\A^n \simeq \vDiv_{\{0\}}\A^n$.
\end{proof}

\subsection{Deformation to the normal cone}
\label{Par:DefCZX}
The following is an application of the construction of the blow-up, although it does not use the actual blow-up. The goal is to generalize \cite[Thm.\ 4.1.13]{KhanVirtual}. We return to an arbitrary closed immersion $Z \to X$, and let $\BB$ the $\OO_X$-algebra such that $Z=\Spec \BB$.
\begin{Def}
	The \textit{normal cone} $C_{Z/X}$ of $Z$ in $X$ is the pullback of $D_{Z/X}$ along the zero section $X \to \A^1_X$. 
\end{Def}
Observe that $C_{Z/X} \to D_{Z/X}$ is a virtual Cartier divisor. Since $G_X(\RR^\ext_{Z/X}) \simeq \BB$, the map $\BB \to \overline{ \RR^\ext_{Z/X}}$, that includes $\BB$ as the degree-zero part of $\overline{ \RR^\ext_{Z/X}}$, has a canonical retraction by \S \ref{Par:IrrId}. We call the corresponding map $j_0: Z \to C_{Z/X}$ the \textit{zero-section} of $C_{Z/X}$.

By the proof of Theorem \ref{Thm:ReesRep}, it follows that we have an equivalence
\begin{align}
\label{Eq:MapEquivGlob}
\Alg^\Z(\A^1_X)(\RR^\ext_{Z/X}, \QQ ) \simeq \Alg(X)(\BB,G_X(\QQ))
\end{align}
natural in $\QQ \in \Alg^\Z(\A^1_X)$. Let $j:\A^1_Z \to D_{Z/X}$ be the canonical map corresponding to the identity on $\BB$ via the equivalence (\ref{Eq:MapEquivGlob}). Write $p:D_{Z/X} \to \A^1_X$ for the canonical projection. We call the sequence
\begin{align}
\label{Eq:DefNC}
\A^1_Z \xrightarrow{j} D_{Z/X} \xrightarrow{p} \A^1_X
\end{align}
the \textit{deformation to the normal cone}.

\begin{Prop}
	\label{Prop:DefNC}
	The deformation to the normal cone satisfies the following properties:
	\begin{enumerate}
		\item It is stable under base change.
		\item The map $j$ is a closed immersion.
		\item Restricting to $\G_{m}$ gives us the canonical sequence $\G_{m,Z} \to \G_{m,X} \to \G_{m,X}$.
		\item Restricting to $0$ gives us $Z \xrightarrow{j_0} C_{Z/X} \xrightarrow{p_0} X$, where $j_0$ is the zero-section and $p_0$ the projection.
	\end{enumerate}
\end{Prop}

\begin{proof}
	The first point is Corollary \ref{Cor:ConePull}.
	
	Observe that the sequence (\ref{Eq:DefNC}) corresponds to a sequence
	\begin{align}
	\label{Eq:AffDefNC}
	\OO_X[t^{-1}] \to \RR^\ext_{Z/X} \to \BB[t^{-1}]
	\end{align} 
	in $\Alg^\Z(X)$. In degrees $\leq 0$, this is $\OO_X[t^{-1}] \to \OO_X[t^{-1}] \to \BB[t^{-1}]$. The second map is clearly surjective, as this can be checked locally on $X$. Thus $j$ is a closed immersion.
	
	For the restriction to $\G_{m}$, let $\G_{m,Z} \to U \to \G_{m,X}$ be the pullback of (\ref{Eq:DefNC}) to $\G_{m,X}$. Define
	\begin{align*}
	W \coloneqq \Res_{\zeta_X}([Z/\G_{m,X}]) \times_{[\A^1_X/\G_{m,X}]}  [\G_{m,X} / \G_{m,X}] 
	\end{align*}
	 By construction of $D_{Z/X}$ and the pullback lemma, we have a Cartesian square
	 \begin{center}
	 	\begin{tikzcd}
	 		U \arrow[r] \arrow[d, "k"] & W \arrow[d, "h"] \\
	 		\G_{m,X} \arrow[r] & {[\G_{m,X}/\G_{m,X}]}
	 	\end{tikzcd}
	 \end{center}
 	which is the pullback of the Cartesian square (\ref{Eq:DefDZX}) that defines $D_{Z/X}$ along $\G_{m,X} \to \A^1_X$. We will show that $h \simeq \id_X$, so that $k\simeq \id_{\G_{m,X}}$.
 	
 	We know that $[\G_{m,X}/\G_{m,X}] \simeq X$. Observe that we have pullback squares
 	\begin{center}
 		\begin{tikzcd}
	 		\varnothing \arrow[r] \arrow[d] & \varnothing \arrow[d] \arrow[r] & {[\G_{m,X}/\G_{m,X}]} \arrow[d] \\
	 		{[Z/\G_{m,X}]} \arrow[r] & B\G_{m,X} \arrow[r] & {[\A^1_X/\G_{m,X}]}
 		\end{tikzcd}
 	\end{center}
 	By Lemma \ref{Lem:WeilPull}, it thus holds that $W \simeq \Res_{\varnothing/X}(\varnothing)$. As a right adjoint, $\Res_{\varnothing/X}$ preserves terminal objects, hence sends $\varnothing$ to $X$. The claim follows.
 	
 	For the last point, we may simply quotient out $t^{-1}$ in the sequence (\ref{Eq:AffDefNC}) to obtain $\OO_X \to \overline{ \RR^\ext_{Z/X}} \to \BB$.
\end{proof}

\begin{Lem}
	\label{Lem:DZCXasVCD}
	The commutative square
	\begin{equation}
	\label{Eq:DZCXasVCD}
	\begin{tikzcd}
	C_{Z/X} \arrow[r] \arrow[d] & D_{Z/X} \arrow[d] \\
	Z \times \{0\} \arrow[r] & \A^1_X
	\end{tikzcd}
	\end{equation}
	 is a virtual Cartier divisor over $Z \times \{0\} \to \A^1_X$.
\end{Lem}

\begin{proof}
	The statement is local on $X$, so we may assume that $Z \to X$ is of the form $\Spec B \to \Spec A$. It is clear that (\ref{Eq:DZCXasVCD}) is classically Cartesian on the underlying classical schemes, and we already know that $C_{Z/X} \to D_{Z/X}$ is a virtual Cartier divisor. Write $F \coloneqq \overline{ R^\ext_{B/A}} \otimes_A B$. By \cite[Rem.\ 4.1.3.iii]{KhanVirtual}, it remains to show that the map $F \to \overline{ R^\ext_{B/A}}$ induced by (\ref{Eq:DZCXasVCD}) is surjective on $\pi_1$. This follows from the fact that the composition $
	\overline{ R^\ext_{B/A}} \to F \to \overline{ R^\ext_{B/A}}
	$
	is the identity.
\end{proof}

\begin{Prop}
	In the case that $Z \to X$ is quasi-smooth, we recover the deformation to the normal cone from \cite[Thm.\ 4.1.13]{KhanVirtual}. In particular, in this case we have that $C_{Z/X}$ is equivalent to the normal bundle $\Spec \Sym \NN_{Z/X}$.
\end{Prop}

\begin{proof}
	Let $B_{Z/X}$ be $\Bl_{Z \times \{0\}}(X \times \A^1) \setminus \Bl_{Z \times \{0\}} (X \times \{0\})$. We will give an equivalence $D_{Z/X} \simeq B_{Z/X}$ over $\A^1_X$, which suffices by \cite[Thm.\ 4.1.13]{KhanVirtual} and Proposition \ref{Prop:DefNC}.
	
	By the universal property of the blow-up, we have a map 
	\begin{align*}
	h:D_{Z/X} \to \Bl_{Z \times \{0\}}(X \times \A^1)
	\end{align*}
	over $\A^1_X$ induced by the virtual Cartier divisor on $D_{Z/X}$ over $(\A^1_X,Z\times \{0\})$ from Lemma \ref{Lem:DZCXasVCD}. We claim that $h$ induces an equivalence $D_{Z/X} \to B_{Z/X}$.
	
	The statement is local on $X$ and stable under base change. Hence, we may assume that $Z \to X$ is of the form $\{0\} \to \A^n$. In this case, it follows from the fact that for a surjection $A \to B$ of discrete rings, the deformation space $\Bl^\cl_{Z \times \{0\}}(X \times \A^1) \setminus \Bl^\cl_{Z \times \{0\}}(X \times \{0\})$ is isomorphic to the spectrum of the classical extended Rees algebra via the map $h$, see e.g.\ \cite[\S 5.1]{FultonIntersection}. In our case, these classical constructions coincide with the derived construction by \S \ref{Par:ClasRees}.	
\end{proof}

\appendix

\section{$M$-graded $R$-algebras as product-preserving presheaves}
\label{Par:AlgMRpres}
Fix a ring $R$ and a commutative monoid $M$. We present a proof of the fact that $\Alg^M_R \simeq \PPP_\Sigma(\Poly^M_R)$, as announced in \S \ref{Par:AlgMR}. We keep the notation from \S \ref{Par:AlgMR}. Our situation can then be summarized in the following commutative diagram:
\begin{equation}
	\label{Eq:DEdiag}
	\begin{tikzcd}
		\Poly \arrow[r, "\epsilon"] \arrow[d, "\delta"] & \Poly^M \arrow[d, "e"] \arrow[ddr, "E", bend left] \\
		\Poly_R \arrow[r, "d"] \arrow[drr, "D", bend right] & \Poly^M_R \arrow[dr, "h"] \\
		&& \Alg^M_R
	\end{tikzcd}
\end{equation}
consisting of functors which preserve finite coproducts, where $e$ and $E$ send $\Z[x(d)]$ to $R[x(d)]$, $d$ and $D$ send $R[x]$ to $R[x(0)]$, $h$ is the inclusion, $\epsilon$ sends $Z[x]$ to $Z[x(0)]$, and $\delta$ sends $\Z[x]$ to $R[x]$.

Recall that for any category $\CCC$ with a full subcategory $\CCC_0 \subset \CCC$, we have a restricted Yoneda functor $\CCC \to \PPP(\CCC_0)$, given by composing the Yoneda embedding $\CCC \to \PPP(\CCC)$ with the restriction functor $\PPP(\CCC) \to \PPP(\CCC_0)$. Since representable functors preserve limits, we in fact get a map $\CCC \to \PPP_\Sigma(\CCC_0)$. Applying this to $\Poly^M_R \subset \Alg^M_R$ gives us a functor $\psi: \Alg^M_R \to \PPP_{\Sigma}(\Poly^M_R)$ which sends $B \in \Alg^M_R$ to
\begin{align*}
	\psi(B): (\Poly^M_R)^\op &\to \Space \\ R[x(d_1),\dots,x(d_n)] &\mapsto B_{d_1} \times \dots \times B_{d_n}
\end{align*}

Let $\varphi$ be the colimit-preserving map $\PPP_\Sigma(\Poly^M_R) \to \Alg^M_R$ induced by the inclusion $\Poly^M_R \subset \Alg^M_R$ via \cite[Prop.\ 5.5.8.15]{LurieHTT}. We will show that $\varphi$ is a two-sided inverse of $\psi$.

\begin{Lem}
	\label{psiphiid}
	The composition $\psi\varphi$ is equivalent to the identity on $\PPP_\Sigma(\Poly_R)$.
\end{Lem}

\begin{proof}
	Observe that $\psi$ preserves sifted colimits by Lemma \ref{Lem:PolMRcompproj}. Therefore, $\psi\varphi$ preserves sifted colimits. Since $\psi\varphi$ is the identity on $\Poly_R^M$, it follows that $\psi\varphi$ is the identity on all of $\PPP_\Sigma(\Poly_R^M)$ by \cite[Prop.\ 5.5.8.15]{LurieHTT}.
\end{proof}

\begin{Lem}
	\label{Lem:phipsiM0}
	In the case $M=\{0\}$, the functor $\varphi:\PPP_\Sigma(\Poly_R) \to \Alg_R$ is a two-sided inverse of $\psi$.
\end{Lem}	

\begin{proof}
	Consider the adjunction
	\begin{align*}
		\Sym : \cn\Mod_R \rightleftarrows \Alg_R  : G
	\end{align*} 
	from \cite[\S 25.2.2]{LurieSpectral}, where $G$ is the forgetful functor. Then $G$ preserves sifted colimits by the fact that $R[x]$ is compact projective and since $\Alg_R \to \Space$ factors through the conservative functor $\cn\Mod_R \to \Space$ via $G$. We can thus apply \cite[Prop. 7.1.4.12]{LurieHA} to this adjunction. Now since $R$ is a compact projective generator of $\cn\Mod_R$ by \cite[7.1.4.15]{LurieHA}, the object $R[x]$ is thus a compact projective generator of $\Alg_R$ by \cite[Prop. 7.1.4.12]{LurieHA}.
	
	By \cite[Prop.\ 5.5.8.25]{LurieHTT}, it now follows that $\varphi$ is an equivalence. By Lemma \ref{psiphiid} it has two-sided inverse $\psi$.
\end{proof}

\begin{Prop}
	\label{Prop:AlgMRYo}
	The functor $\psi:\Alg^M_R \to \PPP_\Sigma(\Poly^M_R)$ is an equivalence of categories with inverse $\varphi$.
\end{Prop}

\begin{proof}
	Keep the notation as in diagram (\ref{Eq:DEdiag}). By Lemma \ref{psiphiid}, it suffices to show  $\varphi\psi \simeq \id_{\Alg^M_R}$.
	Let $p:\Alg^M_R \to \Alg^M$ and $q:\Alg^M_R \to \Alg_R$ be the projections. By definition of $\Alg^M_R$, it suffices to show that $p\varphi\psi \simeq p$ and $q\varphi\psi \simeq q$. 
	
	For $p$, consider the following diagram
	\begin{center}
		\begin{tikzcd}
			\Alg^M_R \arrow[d, "p"] \arrow[r, "\psi"] & \PPP_\Sigma(\Poly^M_R) \arrow[r, "\varphi"] \arrow[d, "e^*"] & \Alg^M_R \arrow[r, "\psi"] \arrow[d, "p"] & \PPP_{\Sigma}(\Poly^M_R) \arrow[d, "e^*"] \\
			\Alg^M \arrow[r, "\psi'"] & \PPP_\Sigma(\Poly^M) \arrow[r, "\varphi'"] & \Alg^M \arrow[r, "\psi'"] & \PPP_{\Sigma}(\Poly^M)
		\end{tikzcd}
	\end{center}
	where $\varphi',\psi'$ are defined similarly as $\varphi, \psi$. Note that $\varphi'$ is the two-sided inverse of $\psi'$. Since the square on the left commutes, it suffices to show that the middle square commutes. Because $\psi'$ is invertible, the latter holds if the rectangle made up from the middle and the right square commutes.  And indeed this holds, since $\psi\varphi \simeq \id_{\PPP_\Sigma(\Poly^M_R)}$. For $q$ one uses a similar diagram, together with Lemma \ref{Lem:phipsiM0}.
\end{proof}

\section{Cell attachments and finiteness conditions}
We collect some additional results that came up on the road to the current constructions and arguments. Even though only \S \ref{Par:Gradn} and \S \ref{Par:FinQuot} have been used in the main text --- namely in the explicit formula of the Rees algebra in \S \ref{Par:ReesForm}, which is non-essential for the construction --- it seems that these results are interesting in their own right. Although most of what follows is well-known, there did not seem to be a satisfactory account in the literature of everything that follows.

\subsection{$\E_\infty$-algebras}
Write $\Einfty$ for the category of $\E$-infinity algebras, $\Ecninfty$  for the full subcategory of $\Einfty$ spanned by the connective objects. Let $\Theta: \Alg \to \Ecninfty$ the functor induced by $\Poly \subset \Ecninfty$, as in \cite[Con.\ 25.1.2.1]{LurieSpectral}. For $R \in \Alg$, write $\Einfty_R$ for the category $\Einfty_{\Theta(R)/}$ of $\E_\infty$-$R$-algebras, likewise $\Ecninfty_R \coloneqq \Ecninfty_{\Theta(R)/}$. 

The functor $\Theta$ induces a functor $\Alg_R \to \Ecninfty_R$, which we also write as $\Theta$. This functor has a left adjoint $\Theta_L: \Ecninfty_R \to \Alg_R$ by \cite[Prop.\ 25.1.2.4]{LurieSpectral}.

\subsection{Spectral Yoneda lemma}
Write $\Sp$ for the category of spectra, and $\Sp(-)$ for the stabilization construction. Let $\CCC$ be a stable category. The Yoneda embedding $\CCC \to \PPP(\CCC)$ into the category of space-valued presheaves on $\CCC$ induces an embedding
\begin{align*}
\CCC \simeq \Sp(\CCC_*) \to \Sp(\PPP(\CCC)_*) \simeq \Fun(\CCC^\op,\Sp)
\end{align*}
which induces a functor
\begin{align*}
\Map_\CCC(-,-):\CCC^\op \times \CCC \to \Sp
\end{align*}
See e.g.\ \cite[\S 2.3]{BlumbergK}. 

For $x,y \in \CCC$, we call $\Map_\CCC(x,y)$ the \textit{mapping spectrum}. It has the property that 
\begin{align}
\label{Eq:pinMap}
\pi_n\Map_\CCC(x,y) = \begin{cases}
\pi_n\CCC(x,y) & \text{ if } n \geq 0 \\
\pi_0\CCC(x,\Sigma^{-n} y) & \text{ if }n \leq 0
\end{cases}
\end{align}

\begin{Not}
	We write $\Map_R(-,-)$ for $\Map_{\Mod_R(-,-)}$ and $\Map(-,-)$ for $\Map_{\Sp}(-,-)$.
\end{Not}

\subsection{Tensoring modules and algebras over spectra and spaces}
Throughout, fix $R \in \Alg$. Let $\Free_R: \Mod_R \to \Einfty_R$ be the functor left adjoint to the forgetful functor $U:\Einfty_R \to \Mod_R$, see e.g.\ \cite[\S 4.4]{GepnerIntroduction} for explicit details. By \cite[\S III.1]{ElmendorfRings}, the forgetful functor $V:\Mod_R \to \Sp$ also has a left adjoint $R \wedge (-)$. 

\begin{Lem}
	\label{Lem:EnrAdj}
	The adjunction $R \wedge (-) \dashv V$ induces equivalences
	\begin{align*}
	\Map_R(R \wedge K, M) \simeq \Map(K,V(M))
	\end{align*}
	natural in $K \in \Sp$ and $M \in \Mod_R$.
\end{Lem}

\begin{proof}
	Let $\eta_K: K \to V(R\wedge K)$ be the unit of the adjunction $R\wedge (-) \dashv V$. We then have a map
	\begin{align*}
	\Map_R(R \wedge K,M) \xrightarrow{V} \Map (V(R\wedge K),V(M)) \xrightarrow{\eta_K^*} \Map(K,V(M)),
	\end{align*}
	which is an equivalence by (\ref{Eq:pinMap}), since $V(-)$ commutes with $\Omega^n$ for all $n$.
\end{proof}

\begin{Exm}
	Let $n\in \Z$, and put $\bbS^n \coloneqq \Sigma^n \bbS$, where $\bbS$ is the sphere spectrum.  We claim that $R \wedge \bbS^n$ is $R[n]$, where $R[n]$ is the shifted $R$-module $\Sigma^n R$ in the stable category $\Mod_R$. To this end, observe that 
	\begin{align*}
	\Map_R(R \wedge \bbS^n, M) \simeq \Map(\bbS^n,M) \simeq \Map(\bbS,\Omega^nM) \simeq \Omega^nM \simeq \Map_R(R[n],M)
	\end{align*}
	for all $R$-modules $M$, where we have used Lemma \ref{Lem:EnrAdj} in the first equivalence. 
\end{Exm}

The example shows that the adjunction $R \wedge (-) \dashv V$ restricts to an adjunction on connective objects. We can thus summarize the situation in the following diagram, where the horizontal arrows commute with the vertical ones
\begin{center}
	\begin{tikzcd}[column sep=large, row sep=large]
	\Space_*  \arrow[r,bend left= 15,"\Sigma^\infty"{name=A, above}] \arrow[d, equals] 
	& \Sp \arrow[l, bend left = 15, "\Omega^\infty"{name=B,below}] \ar[from=A, to=B, symbol=\dashv] \arrow[r,bend left= 15, "R\wedge(-)"{name=C, above} ]  
	& \Mod_R \arrow[l, bend left= 15, "V"{name=D,below}] \ar[from=C, to=D, symbol=\dashv] \arrow[r,bend left= 15,"\Free_R"{name=E, above}]  
	& \Einfty_R  \arrow[l, bend left = 15, "U"{name=F,below}] \ar[from=E, to=F, symbol=\dashv]\\
	\Space_* \arrow[r,bend left= 15,"\Sigma^\infty"{name=G, above}] 
	& \Sp^\cn \arrow[u, hook] \arrow[l, bend left = 15, "\Omega^\infty"{name=H,below}] \ar[from=G, to=H, symbol=\dashv] \arrow[r,bend left= 15, "R\wedge(-)"{name=I, above} ] 
	& \Mod_R^\cn \arrow[u, hook]\arrow[l, bend left= 15, "V"{name=J,below}] \arrow[r,bend left= 15,"\Free_R"{name=K, above}] \ar[from=I, to=J, symbol=\dashv]
	& \Ecninfty_R \arrow[u, hook] \arrow[l, bend left = 15, "U"{name=L,below}] \ar[from=K, to=L, symbol=\dashv]  \arrow[r,bend left= 15,"\Theta_L"{name=M, above}] 
	& \Alg_R \arrow[l, bend left = 15, "\Theta"{name=N,below}] \ar[from=M, to=N, symbol=\dashv]
	\end{tikzcd}
\end{center}
From the proof of \cite[Prop.\ 25.1.2.2]{LurieSpectral}, it follows that $U\Theta$ is the forgetful functor $\Alg_R \to \Mod_R^\cn$. Therefore, the composition $\Theta_L \Free_R$ is $\Sym_R$.


\begin{Not}
	For a pointed space $K$, write $R \wedge \Sigma^\infty K$ as $R \wedge K$, and put $R[K]\coloneqq \Sym_R(R \wedge  K)$.
\end{Not}
\begin{Prop}
	\label{Prop:AdjEqv}
	We have an equivalences $\Alg_R(R[K],A) \simeq \Space_*(K,A)$,
	natural in $A \in \Alg_R$ and $K \in \Space_*$.
\end{Prop}
\begin{proof}
	Immediate.
\end{proof}
\begin{Cor}
	For all $R \in \Alg, A \in \Alg_R, K\in \Space_*$ and $n\in \N$ it holds
	\begin{enumerate}
		\item $\pi_0 \Alg_R(R[S^n],A) \simeq \pi_n(A)$;
		\item $A[K] \simeq A \otimes_\Z \Z[K]$;
		\item $R[S^0] \simeq R[x]$;
		\item $R[\Delta^n] \simeq R$
	\end{enumerate}
\end{Cor}

\begin{proof} Straightforward. \end{proof}
\subsection{Attaching $n$-cells}
Throughout, fix $A \in \Alg_R$. 
\begin{Con}
	Let $\sigma: S^n \to A$ be a given map of pointed spaces. By Proposition \ref{Prop:AdjEqv}, $\sigma$ corresponds to a map $\sigma:R[S^n] \to A$ of $R$-algebras. Consider the pushout
	\begin{center}
		\begin{tikzcd}
		R[S^n] \arrow[d, "\sigma"] \arrow[r] & R[\Delta^{n+1}] \simeq R \arrow[d] \\
		A \arrow[r] & B
		\end{tikzcd}
	\end{center}
	We say that $B$ is obtained from $A$ by \textit{attaching an $(n+1)$-cell along $\sigma$}.
	
	We say that $A[x]$ is obtained from $A$ by attaching a $0$-cell.
\end{Con}

\subsection{Graded $n$-cells}
\label{Par:Gradn} Let $n \geq 0$. Since $\Alg_A(A[S^n],B) \simeq \Space_*(S^n,B)$, we can think of $A[S^n]$ as attaching a free cell to $A$ in level $n$. We want to generalize this to graded rings, and attach free cells of level $n$ in degree $d$. We need the following:

\begin{Lem}
	\label{Lem:SymZ}
	The forgetful functor $G:\Alg^\Z_A \to \cn\Mod^\Z_A$ has a left adjoint.
\end{Lem}

\begin{proof}
	By the adjoint functor theorem, it suffices to show that $G$ preserves limits and $\kappa$-filtered colimits for some regular cardinal $\kappa$. 
	
	Recall that the embedding $\Alg^\Z_A \simeq \PPP_\Sigma(\Poly^\Z_A) \to \PPP(\Poly^\Z_A)$ is a right adjoint (\cite[Prop.\ 5.5.8.10]{LurieHTT}), where we have used Proposition \ref{Prop:AlgMRYo} in the first equivalence. Hence limits in $\Alg^\Z_A$ are computed as limits of presheaves, i.e.,\ point-wise. It follows that $\Alg^\Z_A \to \cn\Mod^\Z_A$ preserves limits.
	
	Since the forgetful functor $\cn\Mod^\Z_A \to \cn\Mod_A$ is conservative, it suffices to show that $\gamma:\Alg^\Z_A \to \cn\Mod_A$ preserves $\kappa$-filtered colimits, for some $\kappa$. But $\gamma$ factorizes as the forgetful map $\Alg^\Z_A \to \Alg_A$, which is a left adjoint, followed by the map $\Alg_A \to \cn\Mod_A$, which preserves $\kappa$-filtered colimits for some $\kappa$. The claim follows.
\end{proof}	

\begin{Not}
	Write the left adjoint of the forgetful functor $\Alg^\Z_A \to \cn\Mod^\Z_A$ as $\Sym_A^\Z(-)$.
\end{Not}

Since the forgetful functor $\Alg^\Z_A \to \Alg_A$ is a left adjoint as well, the underlying $A$-algebra of $\Sym^\Z(M)$ is $\Sym(M)$.

\begin{Exm}
	If $M$ is concentrated in degree 0, i.e.,\ $M = M(0)$ for $M\in \Mod_A$, then it holds $\Sym^\Z_A(M) \simeq \Sym_A(M)(0)$. 
\end{Exm}

\begin{Con}
	For $K \in \Space_*$ and $d \in \Z$, we have a $\Z$-graded $A$-algebra $A[K(d)]$ with the universal property that
	\begin{align*}
	\Alg^\Z_A(A[K(d)],B) \simeq \Space_*(K,B_d)
	\end{align*}
	
	For existence, observe that we have adjunctions
	\begin{center}
		\begin{tikzcd}[column sep=large, row sep=large]
		\Space_* \arrow[r,bend left= 15,"A \wedge (-)"{name=A, above}]  & \Mod_A \arrow[r,bend left= 15,"{(-)(0)}"{name=C, above}] \arrow[l, bend left = 15, "\Omega^\infty V"{name=B,below}] \ar[from=A, to=B, symbol=\dashv] & \Mod^\Z_A \arrow[r,bend left= 15,"(-)_{\bullie - d}"{name=E, above}] \arrow[l, bend left = 15, "(-)_0"{name=D,below}] \ar[from=C, to=D, symbol=\dashv] & \Mod_A^\Z \arrow[l, bend left = 15, "(-)_{\bullie +d}"{name=F,below}]  \ar[from=E, to=F, symbol=\dashv]
		\end{tikzcd}
	\end{center}
	Here, $M_{\bullie + f}$ is the $\Z$-graded $A$-module which has $M_{d+f}$ in degree $d$ for any $\Z$-graded $A$-module $M$, and $N(0)$ is the $\Z$-graded $A$-module which is $N$ concentrated in degree 0, for $N \in \Mod_A$. 
	
	Now write $A \wedge_d K \coloneqq ( (A\wedge K)(0) )_{\bullie- d}$ for any $K \in \Space_*$, and define $A[K(d)] \coloneqq \Sym_A^\Z (A \wedge_d K)$. Then we have
	\begin{align*}
	\Alg^\Z_A(A[K(d)],B) \simeq \Mod^\Z_A(A\wedge_d K, B) \simeq  \Space_*(K,B_d)
	\end{align*}
\end{Con}

\begin{Not}
	We will write $A[S^n(d)]$ as $A[u]$, and say with words that $u$ is a free variable of level $n$ in degree $d$. More generally, we write 
	\begin{align*}
	A[S^{n_1}({d_1}),S^{n_2}(d_2), \cdots, S^{n_k}({d_k})]
	\end{align*}
	as $A[u_1,\dots,u_k]$, and say that $u_i$ is free in level $n_i$ in degree $d_i$.
\end{Not}

\subsection{Finite quotients}
\label{Par:FinQuot}

\begin{Not}
	We will use the following multi-index notation:
	\begin{itemize}
		\item We write $\underline{n}$ for a sequence of numbers $n_1,\dots,n_k$, and let $\underline{n}+1$ be $n_1+1,\dots,n_k+1$;
		\item For such given $\underline{n}$, we write $S^{\underline{n}}$ for the pointed space $S^{n_1} \vee \dots \vee S^{n_k}$;
		\item A collection of maps $\sigma_i: S^{n_i} \to A$ of pointed spaces is written as $\underline{\sigma}: S^{\underline{n}} \to A$;
		\item A collection of free variables $u_1,\dots,u_k$ in level $n_i$ and degree $d_i$ is written $\underline{u}$.
	\end{itemize}
\end{Not}

\begin{Def}
	A \textit{finite quotient} is a cell attachment of the form
	\begin{center}
		\begin{tikzcd}
		A[S^{\underline{n}}] \arrow[d, "\underline{\sigma}"] \arrow[r] & A[\Delta^{\underline{n}+1}] \arrow[d] \\
		A \arrow[r] & A/(\underline{\sigma})
		\end{tikzcd}
	\end{center}
	for some $\sigma_i:S^{n_i} \to A$ with $n_i \geq 0$.
\end{Def}

\begin{Rem}
	Let $\sigma:S^n \to A$ with $n\geq 0$ be a given pointed map, and $B \coloneqq A/(\sigma)$. Then $\pi_0(B) \cong \pi_0(A)/([\sigma])$, where $[\sigma]$ is the image of $\sigma$ in $\pi_0(A)$. Hence, if $n>0$, then $\pi_0(B) \cong \pi_0(A)$. 
	
	In fact, for all $m<n$ it holds that the map $\pi_m(A) \to \pi_m(B)$ is an isomorphism. 
	To see this, let $\tau_{\leq m}$ be the truncation functor. Then since $\tau_{\leq m}$ is a left adjoint, we have $\tau_{\leq m}(B) \simeq \tau_{\leq m}(A) \otimes_{\tau_{\leq m}(A[S^n])} \tau_{\leq m}(A)$. Then the claim follows from the fact that $\tau_{\leq m}(A[S^n]) \simeq \tau_{\leq m}(A)$, since 
	\begin{align*}
	\Alg_A(\tau_{\leq m}(A[S^n]),B) \simeq \Alg_A(A[S^n],B) \simeq * \simeq \Alg_A(\tau_{\leq m}(A),B)
	\end{align*}
	for any $m$-truncated $B\in \Alg_A$.
\end{Rem}

\begin{Lem}
	\label{Lem:UPfinquot}
	Let $A \to A/(\underline{\sigma})$ be a finite quotient. Then this map enjoys the following universal property: for $A$-algebras $B$ with structure map $f: A \to B$, it holds
	\begin{align*}
	\Alg_A(A/(\underline{\sigma}), B ) \simeq \{ f_* \underline{\sigma} \} \times_{\Space_*(S^{\underline{n}},B )} \{ \underline{0} \}
	\end{align*}
	where $\{f_* \underline{\sigma}\}$ and $\{\underline{0}\}$ are the points $(f\sigma_1,\dots,f\sigma_k)$ and $\{0,\dots,0\}$ in $\Space_*(S^{\underline{n}},B)$.
\end{Lem}

\begin{proof}
	This follows directly from the universal property of $A/(\sigma)$ as a pushout, together with Proposition \ref{Prop:AdjEqv}.
\end{proof}

\begin{Cor}
	If $\sigma,\sigma':S^n\to A$ are pointed maps which are identical as elements of $\pi_n(A)$, then $A/\sigma \simeq A/\sigma'$.
\end{Cor}

\begin{proof}
	This follows from the universal property from Lemma \ref{Lem:UPfinquot}.
\end{proof}

\begin{Exm}
	\label{Exm:mod0}
	Write $\sigma$ for the unique map $0: S^n \to A$. Then $A/(\sigma) \simeq A[S^{n+1}]$, since for any $A$-algebra $B$ we have
	\begin{align*}
	\Alg_A(A/(\sigma), B) \simeq \Omega \Space_*(S^n,B) \simeq \Space_*(S^{n+1},B) \simeq \Alg_A(A[S^{n+1}],B)
	\end{align*}
	by Lemma \ref{Lem:UPfinquot} and Proposition \ref{Prop:AdjEqv}.
\end{Exm}


\subsection{Finite cellular maps}
A map $A \to B$ is called \textit{finite cellular} if $B$ is obtained from $A$ by a finite number of cell attachments, including $0$-cells. Observe that being finite cellular is stable under base change and under composition. Also observe that the notion of being finite cellular does not depend on any possible base ring $R$: if $B$ is obtained from $A$ by attaching an $n$-cell in $\Alg_R$, then it is also obtained from $A$ by attaching an $n$-cell in $\Alg_{R'}$, for any $R' \to R$. In particular, we might as well work over $\Z$. 
\begin{Lem}
	\label{Lem:Cyl}
	Let $B$ be a finite cellular $A$-algebra. Then the multiplication map $\mu:B \otimes_A B \to B$ is also finite cellular.
\end{Lem}
\begin{proof}
	We induct on the number of cells attached to $A$ to obtain $B$. So first suppose that $B = A/(\sigma)$ for certain $\sigma:S^n \to A$. 
		
	Let $S^n \vee S^n \to \Delta^1 \wedge S^n$ be the map that sends one copy of $S^n$ to $\{0\} \times S^n$, and the other copy to $\{1\} \times S^n$. Using this map, consider the following commutative diagram
	\begin{center}
		\begin{tikzcd}
		A[S^n \vee S^n] \arrow[r] \arrow[d] & A[\Delta^{n+1}\vee \Delta^{n+1}]  \arrow[d]  &  \\
		A[\Delta^1 \wedge S^n] \arrow[d] \arrow[r] & A[S^{n+1}] \arrow[d] \arrow[r] & A[\Delta^{n+2}] \arrow[d]  \\
		A \arrow[r] & B\otimes_A B \arrow[r] & B
		\end{tikzcd}
	\end{center}
	The top left square is a pushout since $A[-]$ is a left adjoint. The two left squares together form a pushout by assumption on $B$. Hence the bottom left square is a pushout. Now the two bottom squares together from a pushout since $A[-]$ preserves equivalences an by assumption on $B$, and thus the bottom right square is a pushout. It follows that $B \otimes_A B \to B$ is finite cellular.

	The case $B = A[x]$ is straightforward.
	
	Now suppose that we have a finite cellular map $A \to B'$ for which $B' \otimes_A B' \to B'$ is finite cellular, and some $\tau:S^m \to B'$. Put $B \coloneqq B'/(\tau)$. Then since $B' \otimes_A B' \to B'$ is finite cellular, so is $f:B' \otimes_A B \to B$, since this notion is stable under taking pushouts. Now $\mu:B \otimes_A B \to B$ factorizes as
	\begin{align*}
		B \otimes_A B \simeq B \otimes_{B'} B' \otimes_A B \xrightarrow{ \id\otimes f} B \otimes_{B'} B \to B
	\end{align*}
	where the last map is the multiplication map of $B$ as $B'$ algebra. As such, this map is finite cellular by the first case. Also, $\id\otimes f$ is finite cellular since $f$ is, and since this notion is stable under base change. It follows that $\mu$ is finite cellular.
\end{proof}

\subsection{Finitely presented $A$-algebras}
\begin{Def} Let $B \in \Alg_A$ be given.
	\begin{itemize}
		\item We say that $B$ is  \textit{finitely presented} if it belongs to the smallest full subcategory of $\Alg_A$ which contains $A[x]$ and is closed under finite colimits.
		\item We say that $B$ is \textit{finite cellular} in $\Alg_A$ if the structure map $A \to B$ is finite cellular.
	\end{itemize}
	Write $\Alg_A^\fp$ for the full subcategory of $\Alg_A$ spanned by the finitely presented $A$-algebras.
\end{Def}

\begin{Lem}
	\label{Lem:FinCelMap}
	Let $\varphi:B \to C$ be a map between finite cellular $A$-algebras. Then $C$ is a finite cellular $B$-algebra.
\end{Lem}

\begin{proof}
	Consider the following commutative diagram
	\begin{center}
		\begin{tikzcd}
		A \arrow[d] \arrow[r] & B \arrow[d] \\
		B \arrow[d] \arrow[r] & B \otimes_A B \arrow[r] \arrow[d] &  B \arrow[d] \\
		C \arrow[r] & B \otimes_A C \arrow[r] &  C
		\end{tikzcd}
	\end{center}
	consisting of pushout squares. Since $B\otimes_AB \to B$ is finite cellular by Lemma \ref{Lem:Cyl}, so is $B \otimes_A C \to C$. Likewise, since $A \to C$ is finite cellular, so is $B \to B \otimes_A C$. It follows that $B \to C$ is finite cellular. Also, since the composition $B \to B$ in the middle row is homotopic to the identity, the composition $B \to C$ is homotopic to $\varphi$. 
\end{proof}

\begin{Prop}
	\label{Prop:FPisFC}
	An $A$-algebra $B$ is  finitely presented if and only if it is finite cellular.
\end{Prop}	

\begin{proof}
	It is clear that if $B$ is finite cellular, then it is of finite presentation. For the converse, let $\CCC$ be the full subcategory of $\Alg_A$ spanned by the finite cellular $A$-algebras. Then $A[x] \in \CCC \subset \Alg_A^\fp$. Hence, by definition of $\Alg_A^\fp$, it suffices to show that $\CCC$ is closed under finite colimits.
	
	By \cite[Cor.\ 4.4.2.4]{LurieHTT}, it suffices to show that $\CCC$ is closed under pushouts. First observe that for finite cellular $A$-algebras $B',B''$, the fibre product $B'\otimes_A B''$ is finite cellular over $A$ as well. This follows again from the fact that being finite cellular is stable under pushouts and under composition.
	
	Now for the general case, suppose we have maps $B \to B', B \to B''$ with $B',B,B''$ finite cellular. By Lemma \ref{Lem:FinCelMap}, we may assume that $B',B''$ are finite cellular $B$-algebras. By the previous case, we know that $B' \otimes_B B''$ is a finite cellular $B$-algebra as well, and thus we have a sequence $B = B_0 \to \dots \to B_n = B' \otimes_B B''$ where each $B_k \to B_{k+1}$ is obtained by attaching a finite number of cells of any dimension. Composing this sequence with a sequence which established $A \to B$ as finite cellular gives us what we want.
\end{proof}

\begin{Rem}
	\label{Rem:lfp}
	There is also the notion of a locally finitely presented $A$-algebra. This is an $A$-algebra $B$ which is a retract of a finitely presented $A$-algebra. Classically, for a retract $X \xrightarrow{i} Y \xrightarrow{p} X$ it holds that $X$ is the coequalizer of $ip$ and $\id_Y$, hence the difference between finitely presented and locally finitely presented disappears in this context. In the $\infty$-setting this approach does not work, essentially because $p$ need not be an epimorphism in this setting.
	
	Locally finitely presented $A$-algebras are precisely the compact objects in $\Alg_A$, i.e.,\ those $A$-algebras $B$ for which $\Alg_A(B,-)$ commutes with filtered colimits. 
\end{Rem}

\begin{Prop}
	\label{Prop:Fincellseq}
	Let $B$ be a finitely presented $A$-algebra. Then we can arrange the sequence $A=B_0 \to \dots \to B_n=B$ of cell attachments in such a way that each $B_k \to B_{k+1}$ is obtained by attaching a finite number of $k$-cells to $B_k$.
\end{Prop}

\begin{proof}
	Let a sequence of cell attachments $A = B_0 \to B_1 \to \dots \to B_n = B$ be given. 
	
	First suppose that there is some $i$ such that $B_{i-1} \to B_i \to B_{i+1}$ is of the form $B_{i-1} \to B_{i-1}/(\underline{\sigma}) \to (B_{i-1}/(\underline{\sigma}))[\underline{x}]$ for some pointed map $\underline{\sigma}:S^{\underline{n}} \to B_{i-1}$ and $\underline{x} = x_1,\dots,x_m$. Write also $\underline{\sigma}$ for the composition $S^{\underline{n}} \to B_{i-1} \to B_{i-1}[\underline{x}]$. Then $(B_{i-1}/(\underline{\sigma}))[\underline{x}] \simeq B_{i-1}[\underline{x}]/(\underline{\sigma})$. We can therefore replace $B_{i-1} \to B_{i-1}/(\underline{\sigma}) \to (B_{i-1}/(\underline{\sigma}))[\underline{x}]$ by $B_{i-1} \to B_{i-1}[\underline{x}] \to B_{i-1}[\underline{x}]/(\underline{\sigma})$. Repeating this procedure gives us that we can arrange the sequence in such a way that $B_0 \to B_{1}$ is of the form $A \to A[\underline{y}]$ for certain $\underline{y} = y_1 \dots y_l$ and that no other $B_k \to B_{k+1}, k>0$ is of this form.
	
	Now suppose that there is some $i$ such that $B_{i-1} \to B_i \to B_{i+1}$ is of the form $B_{i-1} \to B_{i-1}/(\sigma) \to (B_{i-1}/(\sigma))/(\tau)$, with $\sigma:S^n \to B_{i-1}$ and $\tau:S^m \to (B_{i-1})/(\sigma)$ with $n>m$. Using that $\pi_m(B_{i-1}) \cong \pi_m(B_{i-1}/(\sigma))$, take some $\tau':S^m \to B_{i-1}$ such that the composition $S^m \to B_{i-1} \to B_{i-1}/(\sigma)$ is equivalent to $\tau$. Write $\sigma'$ for the composition $S^n \to B_{i-1} \to B_{i-1}/(\tau')$. Consider the following diagram consisting of pushout squares	
	\begin{center}
		\begin{tikzcd}
		& B_{i-1}[S^n] \arrow[d, "\sigma"] \arrow[r, "0"] & B_{i-1} \arrow[d] \\
		B_{i-1}[S^m]\arrow[d, "0"] \arrow[r, "\tau'"] & B_{i-1} \arrow[d]\arrow[r] & B_{i-1}/(\sigma)\arrow[d] \\
		B_{i-1} \arrow[r] & B_{i-1}/(\tau') \arrow[r] & C
		\end{tikzcd}
	\end{center}
	where we denote a map $B_{i-1}[K] \to B_{i-1}$ induced by some $\alpha:K \to B_{i-1}$ simply by $\alpha$. Then the two bottom squares and the two right squares establish equivalences 
	\begin{align*}
	(B_{i-1}/(\sigma))/(\tau) \simeq C \simeq (B_{i-1}/(\tau'))/(\sigma')
	\end{align*}
	In the given sequence $B_0 \to \dots \to B_n$, we can thus interchange all cell attachments which are in the wrong order to get what we want.
\end{proof}

\bibliographystyle{style}
\bibliography{refs}

\end{document}